\documentclass{article}
\usepackage{blkarray}
\usepackage{amsmath}
\usepackage[boxed, linesnumbered, ruled]{algorithm2e}
\usepackage{amsfonts}
\usepackage{graphicx}
   \usepackage{amsmath, amssymb, amsthm}
\usepackage{mathrsfs}

\usepackage{tikz}
 \newtheorem{thm}{Theorem}[section]
 \newtheorem{cor}[thm]{Corollary}
 \newtheorem{lem}[thm]{Lemma}
 
 \newtheorem{prop}[thm]{Proposition}
 \newtheorem{obs}[thm]{Observation}
 \theoremstyle{definition}
 \newtheorem{defn}[thm]{Definition}
 \newtheorem{exmp}[thm]{Example}

\numberwithin{equation}{section}

 \newcommand{\set}[1]{\left\{#1\right\}}
 \newcommand{\Set}[2]{\set{#1\ \vert\ #2}}

\usetikzlibrary{arrows}

\begin{document}

\title{On some subclasses of circular-arc catch digraphs}
\author{Sanchita Paul\thanks{Department of Mathematics, Jadavpur University, Kolkata - 700 032, India. sanchitajumath@gmail.com}\thanks{(Corresponding author)},
Sayan Gupta \thanks{Department of Mathematics, Jadavpur University, Kolkata - 700 032, India.   
sayangupta4u@gmail.com} and
Shamik Ghosh\thanks{Department of Mathematics, Jadavpur University, Kolkata - 700 032, India. ghoshshamik@yahoo.com}}

\date{}

\maketitle

\begin{abstract}
\noindent
{\footnotesize {\em Catch digraphs} was introduced by Hiroshi Maehara in 1984 as an analog of intersection graphs where a family of pointed sets represents a digraph. 
After that Prisner continued his research particularly on {\em interval catch digraphs} by characterizing them diasteroidal triple free.
It has numerous applications in the field of real world problems like network technology and telecommunication operations. Recently, we characterized three important subclasses of interval catch digraphs. In this article we introduce a new class of catch digraphs, namely {\em circular-arc catch digraphs}. The definition is same as interval catch digraph, only the intervals are replaced by circular-arcs here. We present the characterization of {\em proper circular-arc catch digraphs}, which is a natural subclass of circular-arc catch digraphs where no circular-arc is contained in other properly. For this we introduce a concept, namely {\em monotone circular ordering} for the vertices of the augmented adjacency matrix of it. Next we find that underlying graph of a {\em proper oriented circular-arc catch digraph} is a proper circular-arc graph. Also we characterize proper oriented circular-arc catch digraphs by defining a certain kind of circular vertex ordering of its vertices. Another interesting result is to characterize {\em oriented circular-arc catch digraphs} which are tournaments in terms of forbidden subdigraphs. Further we study some properties of an oriented circular-arc catch digraph. In conclusion we discuss the relations between these subclasses of circular-arc catch digraphs.}
\end{abstract}

\noindent
{\scriptsize Keywords:} {\footnotesize interval graph; circular-arc graph; proper circular-arc graph; interval catch digraph; oriented graph; tournament}
  
\noindent
{\scriptsize 2010 Mathematical Subject Classification:} {\footnotesize Primary:
05C20, 05C62, 05C75}

\section{Introduction}
Intersection graphs have many significant uses in real world situational problems. Owing to its numerous applications in the field of network science various types of intersection graphs were invented to model different types of geometric objects. Most popular among them is the class of interval graphs. In 1984 Mahera  \cite{Maehara} proposed a similar concept for digraphs, known as {\em catch digraphs}. A {\em catch digraph} of $\mathcal{F}$ is a digraph $G=(V,E)$ in which $uv\in E$ if and only if $u\neq v$ and $p_{v}\in S_{u}$ where $\mathcal{F}=\Set{(S_{u},p_{u})}{u\in V}$ is a family of pointed sets (a set with a point within it) in a Euclidean space. The digraph $G$ is said to be {\em represented} by $\mathcal{F}$. Eventually it was noticed that the digraphs which can be represented by families of pointed intervals can easily be characterized.  
Continuation to this study {\em interval catch digraphs} (in brief, ICD) was introduced by Prisner \cite{Prisner} in 1989 where $S_{u}$ is represented by an interval $I_{u}$ in the real line.  Analogous to the concept of asteroidal triples \cite{Lekker}, he introduced {\em diasteroidal triples} among digraphs and characterized interval catch digraphs by proving them diasteroidal triple free. He also described a special vertex ordering for vertices of the graph in which the augmented adjacency matrix satisfies the consecutive $1$'s property for rows. In continuation to this literature recently we studied and characterized three important subclasses of interval catch digraphs, namely, central interval catch digraphs, oriented interval catch digraphs and proper interval catch digraphs \cite{PG} and eventually disproved a conjecture of Maehara stated in \cite{Maehara}.

\vspace{0.7em}

\noindent {\em Circular-arc graphs} are another important variation of intersection graphs of arcs around a circle. This class of graphs generalizes interval graphs where each vertex is mapped to a circular-arc around a circle so that any two vertices are adjacent if and only if their corresponding circular-arcs intersect. In 1971, Tucker \cite{Tucker} first characterized circular-arc graphs by introducing {\em quasi-circular $1$'s property} of the augmented adjacency matrix associated to it. In 2003 McConnell \cite{Mc} recognized circular-arc graphs in linear time. There are several structural characterizations and recognitions has already been done for circular-arc graphs and many of its subclasses. Among them proper circular-arc graphs, unit circular-arc graphs, helly circular-arc graphs and co-bipartite circular-arc graphs are few to mention, which are all summarised in a recent survey paper by M.C.Lin and J.L.Szwarcfiter \cite{Lin}. Circular-arc graphs have obtained considerable attention for its beautiful structural properties and diverse applications in areas such as genetics, traffic control and many others. To know more about circular-arc graphs one may see books by Golumbic \cite{G} and  by Roberts \cite{Roberts} respectively. Also see \cite{Con, Hubert, Robert2, Stahl, Stou, Ste, Tuck, Tuc, Tu, T} for applications. Among digraphs, the {\em oriented graphs} are the ones that contain no directed cycle of length $2$, that is for every pair of vertices $u,v$, at most one of $u\rightarrow v$ ($uv\in E$) or $v\rightarrow u$ ($vu\in E$) may occur in an oriented graph $G=(V,E)$. Although characterization, polynomial time recognition and many other combinatorial problems \cite{ Prisner1} of interval catch digraphs are already done many years ago, but no significant combinatorial result such as forbidden subdigraph characterization, have been found so far for oriented catch digraphs.


\vspace{0.7em}
\noindent In this article we introduce a new class of catch digraphs, namely {\em circular-arc catch digraphs} (cf. Definition \ref{cac}) where $S_{u}$ is represented by circular-arc  $I_{u}$ around circle. We first characterize circular-arc catch digraphs which is an easy extension of the analogous theorem for interval catch digraphs. We mainly focus our study to two of its subclasses namely, {\em oriented circular-arc catch digraphs} (cf. Definition \ref{CG}) and {\em proper circular-arc catch digraphs} (cf. Definition \ref{pcac}). To obtain the characterization of proper circular-arc catch digraphs we implement a new concept {\em monotone circular ordering} (cf. Definition \ref{MCO}) on its vertex set. The main idea of this result is about finding the structural significance of the corresponding augmented adjacency matrix of it when it satisfies monotone circular ordering after independent row and column permutation of it. Various properties of oriented circular-arc catch digraphs are also examined. Another worth mentioning result we obtain is that the underlying graph of a {\em proper oriented circular-arc catch digraph} (cf. Definition \ref{pocac}) is a proper circular-arc graph. Next we find that an oriented circular-arc catch digraph always possesses hamiltonian path when it is {\em unilateral}. We also discover that underlying undirected graph of an oriented circular-arc catch digraph can not contain $\overline{C_{n}},n\geq 8$ as induced subgraph. Further we study those oriented circular-arc catch digraphs which are tournaments and from their structural behavior we able to characterize them by identifying the entire list of forbidden subdigraphs. Finally we able to characterize {\em proper oriented circular-arc catch digraphs} by specifying a certain circular vertex ordering of its vertices. In conclusion we discuss relationships between these classes of digraphs.

\section{Preliminaries}

\noindent A matrix whose entries are only $0,1$ is a \textit{binary} matrix. The {\em augmented adjacency matrix} of a digraph $G$ is obtained from the adjacency matrix of $G$ by replacing $0$ by $1$ along the principal diagonal \cite{G} and is denoted by $A(G)$ throughout the article. In 1970 Tucker \cite{Tucker} first characterized circular-arc graphs and proper circular-arc graphs in terms of their augmented adjacency matrices. For this purpose he defined {\em circular $1$'s property} of a binary matrix, which is generalized version of consecutive $1$'s property used to characterize interval graphs and proper interval graphs. A  binary matrix is said to satisfy {\em circular $1$'s property along rows} if the columns can be permuted so that the $1$'s in each row are {\em circular}, i.e., they appear in a circularly consecutive way. Note that a binary matrix has the circular $1$'s property along rows if and only if it has circular $0$'s property along rows. Analogously {\em circular $1$'s property along columns} is defined by reversing the roles of rows and columns. A matrix satisfies {\em circular $1$'s property along rows and columns} if $1$'s in each row and column are circular for an independent permutation of rows and columns of it. In a digraph $G=(V,E)$, $\{v_{1},v_{2},v_{3}\}\subseteq V$ is called {\em diasteroidal triple} \cite{Prisner} if for every permutation $\sigma$ of $\{1,2,3\}$ there exists a $v_{\sigma_{1}}$ avoiding \footnote{A $v_{2}-v_{3}$ chain $\vec{P}=(v_{2}=a_{0}, a_{1},\hdots, a_{n}=v_{3})$ in $G$ is a digraph with $V(\vec{P})\subseteq V(G)$ and $E(\vec{P})\subseteq E(G)$ where either $a_{i}a_{i+1}$ or $a_{i+1}a_{i}$ is an arc (but not both) for each $i$, $0\leq i\leq n-1$. A $v_{2}-v_{3}$ chain $\vec{P}$ is called $v_{1}$ avoiding if $v_{1}\notin V(\vec{P})$ and $a_{i}v_{1}\notin E(G)$ for any $i$ with $a_{i}a_{i+1}\in E(\vec{P})$ or $a_{i}a_{i-1}\in E(\vec{P})$ \cite{GT}.} $v_{\sigma_{2}}-v_{\sigma_{3}}$ chain in $G$. The following characterization is known for interval catch digraphs.

\begin{thm} \cite{Maehara, Prisner}\label{diate}
Let $G=(V,E)$ be a simple digraph. Then the following are equivalent:
\begin{enumerate}
\item  $G$ is an interval catch digraph.
\item $G$ is diasteroidal triple-free.
\item  there exists an ordering $``<"$ of $V$ such that
\begin{equation}\label{icd1}
 for\hspace*{.5em}  x<y<z\in V,\hspace*{.5em}\ xz\in E\Longrightarrow xy\in E \hspace*{.5em}\text{ and }\hspace*{.5em} zx\in E\Longrightarrow zy\in E.
 \end{equation}
\end{enumerate}
\end{thm}

\noindent We call the ordering (\ref{icd1}) as an {\em ICD ordering} of $V$. This ordering is not unique for an ICD. For a simple digraph $G=(V,E)$ we denote the set of all {\em outneighbors} ({\em inneighbors}) of a vertex $u\in V$ by $N^{+}(u)$ ($N^{-}(u)$) i.e., $N^{+}(u)=\{v\in V|uv\in E\}$ and $N^{-}(u)=\{v\in V|vu\in E\}$ are the sets of open neighbors of $u$. For convenience we refer to $N^{+}(u)$, $N^{-}(u)$ as {\em outset} and {\em inset} of the vertex $u$ in $G$, respectively. Also we use the notation $N^{+}[u]$ ($N^{-}[u]$) to  denote the sets of closed neighbors of $u$ where $N^{+}[u]=N^{+}(u)\cup \{u\}$ and $N^{-}[u]=N^{-}(u)\cup \{u\}$. Also $|N^{+}(u)|$ is the {\em outdegree}, $|N^{-}(u)|$ is the {\em indegree} of $u$. The {\em underlying graph} of a digraph $G = (V,E)$ is an undirected graph $U(G) = (V^{\prime},E^{\prime})$ where $V = V^{\prime}$ and $E^{\prime} = \{uv | uv \in E \hspace{0.2em} or \hspace{0.2em} vu \in E\}$.
The underlying graph of a digraph $G = (V,E)$ is an undirected graph $U(G) = (V^{'},E^{'})$ where $V = V^{'}$ and $E^{'} = \{uv | uv \in E \hspace{0.2em} or \hspace{0.2em} vu \in E\}$.  We call a circular-arc catch digraph as {\em connected} when its underlying undirected graph is connected. A {\em tournament} is a digraph $G$ having exactly one direction for each edge of it so that the underlying graph $U(G)$ is a complete graph. The unique acyclic $n$-tournament for $n\geq 1$ is termed as {\em transitive tournament} and is denoted by $TT_{n}$ in \cite{Bang}. This has an ordering $v_{1},v_{2},\hdots,v_{n}$ of its vertices so that $v_{i}v_{j}$ is an arc whenever $1\leq i<j\leq n$. A digraph is said to be {\em unilaterally connected or unilateral} if for every pair of vertices $u,v$ either there is a directed path from $u$ to $v$ or from $v$ to $u$ (or both).

\vspace{0.7em}

\noindent  {\em Oriented interval catch digraphs} (in brief, oriented ICD) are those oriented graphs which are interval catch digraphs \cite{PG}.

\begin{thm}\cite{PG}\label{dag}
An oriented ICD is a directed acyclic graph (DAG).
\end{thm}

\noindent Again a {\em proper oriented interval catch digraph} (in brief, proper oriented ICD)  \cite{PG} is an oriented interval catch digraph where no interval contains other properly. The following is an important result on proper oriented interval catch digraphs which will be used while characterizing oriented circular-arc catch digraphs, which are tournaments.

\begin{figure}[h]
\hspace{7em}
\includegraphics[scale=0.39]{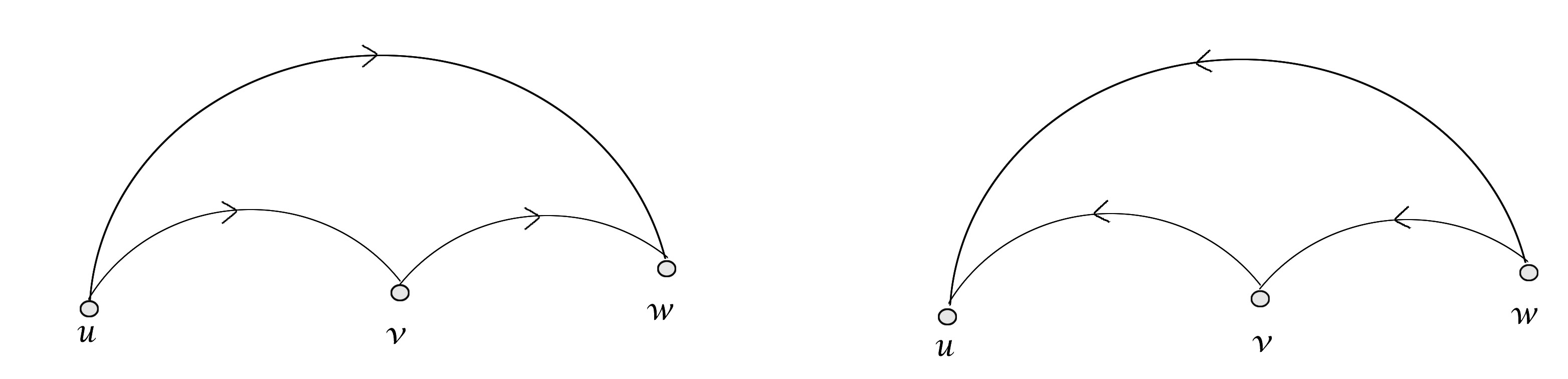}
\caption{proper oriented ICD ordering}
\end{figure}

\begin{thm}\label{poicd1}\cite{PG}
Let $G=(V,E)$ be an oriented graph. Then $G$ is a proper oriented interval catch digraph if and only if there exists a vertex ordering which satisfy the following
\begin{equation} \label{poicd}
\text{for} \ u<v<w, \text{if} \ uw\in E \ \text{then} \ uv,vw\in E \ \text{and if} \ wu\in E \ \text{then} \ wv,vu\in E.
\end{equation}
\end{thm}

\noindent The above ordering mentioned in (\ref{poicd}) is called {\em proper oriented ICD ordering} of $V$. 
\vspace{0.7em}

\noindent In \cite{PG1} we have provided forbidden subdigraph characterization of those proper oriented interval catch digraphs whose underlying graphs are chordal. As each tournament is a chordal graph, we can immediately conclude the following result as a corollary of Theorem $7.6$ of \cite{PG1}. 

\begin{cor}\label{tforbid}
Let $G=(V,E)$ be an oriented ICD which is a tournament. Then $G$ is a proper oriented ICD.
\end{cor}

\noindent By {\em trivial rows} of a binary matrix we mean zero rows (i.e.,  all entries as $0$) or full rows (i.e., all entries as $1$). Thus a {\em nontrivial row} is a row having at least one entry $1$ and at least one entry 0. Similar concepts can be defined for columns of a binary matrix. For a symmetric binary matrix $M$, {\em circularly compatible $1$'s property} \cite{Tucker} is defined in such a way so that the columns of $M$ satisfy circular $1$'s property and permuting order of rows and corresponding columns results in a matrix where the last $1$ (in cyclically descending order) of the circular stretch in the second column is always at least as low as the last $1$ of the circular stretch in the first column unless one of these column are trivial columns. A circular-arc graph is said to be a {\em proper circular-arc graph} if there exists a circular-arc representation where no arc is contained in another properly, i.e., it is the intersection graph of an inclusion-free family of circular-arcs. Proper circular-arc graphs admit several characterizations \cite{DHH, HBH, Lin, Skrien, Tucker}. Among them the following is a popular one.

\begin{thm}\cite{Tucker, Tuc}
A graph $G=(V,E)$ is a proper circular-arc graph if and only if there is an 
arrangement of the augmented adjacency matrix $A(G)$ satisfying circularly compatible $1$'s property.  
\end{thm}

\vspace{0.3em}

\noindent  A bipartite graph (in brief, bigraph) \cite{BDGS} $B=(U,V,E)$ is an {\em intersection bigraph} if there is a family $\mathcal{F}=\{I_{v}|v\in U\cup V\}$ of sets such that $uv\in E$ if and only $I_{u}\cap I_{v}\neq \emptyset$ for all $u\in U,v\in V$. An intersection bigraph is said to be an 
{\em interval bigraph} (respectively, a circular-arc bigraph) if $\mathcal{F}$ is a family of intervals on the real line (respectively, arcs on a circle). The submatrix of the adjacency matrix $B$ containing rows corresponding to the vertices of $U$ and columns corresponding to the vertices of $V$ is the {\em biadjacency matrix} of $B$. A {\em proper circular-arc bigraph} is an intersection bigraph of two families $\mathcal{F}_{1}$ and $\mathcal{F}_{2}$ of arcs on a circle where neither $\mathcal{F}_{1}$ nor $\mathcal{F}_{2}$ contains two arcs such that one is properly contained in other. Basu et.al in \cite{BDGS} used the idea of circularly compatible $1$'s to characterize {\em proper circular-arc bigraphs}. In this context {\em monotone circular arrangement} property was defined \cite{BDGS}. 
A $m\times n$ binary matrix $A$ not containing any full rows or columns is said to satisfy {\em monotone circular arrangement property} if the following holds

\begin{enumerate}
\item $A$ admits circular $1$'s property along rows and columns.
 
\item If $u_{1}, u_{2},\hdots, u_{m}$ are the rows of $A$ and
$[a_{i},b_{i}]$ be the circular stretch of $1$'s attached to $u_{i}$ for each $i\in\{1,2,\hdots,m\}$ then there exists a linear order of rows such that
$a_{1},a_{2},\hdots, a_{m}$ and $b_{1},b_{2},\hdots, b_{m}$ satisfy $a_{i}\leq a_{i+1}$ ($a_{m+1}=a_{1}$) and $b_{i}\leq b_{i+1}$ ($b_{m+1}=b_{1}$) where $\leq $ stands for a  non-decreasing circular order (clockwise).
\end{enumerate} 
 
\vspace{0.3em}

\noindent  In \cite{BDGS} proper circular-arc bigraphs have been characterized with the help of monotone circular arrangement property of their corresponding biadjacency matrices.

\vspace{0.3em}
 \vspace{0.3em}
 \begin{thm}\cite{BDGS}
Let $B=(U,V,E)$ be a bigraph. Then $B$ is a proper circular-arc bigraph if and only if the biadjacency matrix of $B$ satisfies monotone circular arrangement.
\end{thm}
 
\vspace{0.3em}

\noindent  We generalize the notion of monotone circular arrangement for a $m\times n$ binary matrix containing all types of rows (nontrivial/full) 
 in Section \ref{mc} to characterize proper circular-arc catch digraphs. 
Any nontrivial row can be considered as a subset of a full row by taking the columns containing $1$'s (along the rows) as elements of the sets. Therefore apparently one may think that full rows could not appear in the augmented adjacency matrix $A(G)$ of a proper circular-arc catch digraph $G$. However, this is not true (see Figure \ref{pcad}). In fact the situation becomes very difficult when we allow full rows in the matrix. To handle this challenge we define {\em monotone circular ordering} of a binary matrix containing non-zero rows of all types (i.e., where full rows are allowed).

\vspace{0.7em}

\noindent Although Tucker \cite{Tuc} characterized a proper circular-arc graph in two ways: first, in terms of adjacency matrix structure, and second, by identifying forbidden subgraphs of the corresponding graph, Hell, Bang-Jensen and Huang \cite{HBH}, Skrien \cite{Skrien}, Deng, Hell and Huang \cite{DHH} have given different characterizations of the same graph class.
The following is the main concept of characterizing a proper circular-arc graph introduced by them. They generalizes tournaments, i.e; orientations of a complete graph.

\vspace{0.7em}

\noindent  A {\em local tournament} is an oriented graph in which the inset, as well as the outset of every vertex induce tournaments. An oriented graph $G=(V,E)$ is {\em round} \cite{DHH} if the vertices of $G$ can be circularly ordered 
$v_{1},\hdots, v_{n}$ so that, for each vertex $v_{i}$, there are non-negative integers $l_{i},r_{i}$ such that the inset of $v_{i}$ are 
$\{v_{i-l_{i}},\hdots, v_{i-1}\}$ and the outset of $v_{i}$ are $\{v_{i+1},\hdots, v_{i+r_{i}}\}$ (here subscript additions and subtractions are done in modulo $n$) such that if $v_{i}v_{j}$ is an arc in $G$ (regardless of whether $i<j$ or $j<i$) then the subgraph induced by the vertices $v_{i},v_{i+1},\hdots,v_{j}$ (the subscripts are computed in modulo $n$) must form a transitive tournament. Such a circular ordering of vertices is called a {\em round enumeration} of $G$. An undirected graph is said to have 
a {\em round orientation} \cite{DHH} if it admits an orientation which is a round oriented graph. It is clear that a round oriented graph is a local tournament. A graph is {\em orientable as a local tournament} \cite{HBH} if there is an orientation of all its edges so that the corresponding oriented graph becomes a local tournament.  

\vspace{0.7em}
\begin{thm}\cite{DHH, HBH, Skrien}\label{pca}
The following are equivalent for a connected graph:
\begin{enumerate}
\item $G$ is a proper circular-arc graph.
\item $G$ is orientable as a local tournament.
\item $G$ has a round orientation.
\end{enumerate}
\end{thm}

\section{Circular-arc catch digraph}

\begin{defn}\label{cac}
Let $G=(V,E)$ be a simple directed graph. Then $G$ is said to be {\em circular-arc catch digraph} if for each $v\in V$ we can associate a circular-arc $I_{v}=[a_{v},b_{v}]$ and a point $p_{v}$ within $I_{v}$ such that $uv\in E$ if and only if $u\neq v$ and $p_{v}\in I_{u}$. Moreover, $G$ is said to have a {\em circular-arc catch digraph representation} $\{(I_{v},p_{v})|v\in V\}$.
\end{defn}

\noindent The following Theorem  characterizes a circular-arc catch digraph.

\begin{thm}\label{cop1}
Let $G=(V,E)$ be a simple digraph. Then $G$ is a circular-arc catch digraph if and only if there exists a vertex ordering with respect to which the augmented adjacency matrix $A(G)$ satisfies circular $1$'s property along rows.
\end{thm}

\begin{proof}
Let $G=(V,E)$ be a circular-arc catch digraph with representation 
$(I_{i},p_{i})$ where $I_{i}=[a_{i},b_{i}]$ be the circular-arc associated with vertex $v_{i}$ and $p_{i}$ be a point within $I_{i}$ such that $v_{i}v_{j}\in E$ if and only if $p_{j}\in I_{i}$. We now arrange the vertices of $V$ according to the increasing order of $p_{i}$'s along both rows and columns of $A(G)=(a_{i,j})$. Suppose $a_{i,j_{1}}=a_{i,j_{2}}=1$, where $j_{2}$ is circularly right (clockwise) of $j_{1}$ along the circular stretch of the $i^{\text{th}}$ row. Let $j$ be lying between $j_{1}$ and $j_{2}$ according to the circular order. Since $p_{j_{1}},p_{j_{2}}\in I_{i}$ and $I_{i}$ is a circular-arc, we have $p_{j}\in I_{i}$. Hence $v_{i}v_{j}\in E$, which implies $a_{i,j}=1$. Thus $A(G)$ satisfies circular $1$'s property along rows. 

\vspace{0.37em}
\noindent Conversely, Let $A(G)$ satisfy circular $1$'s property along rows of it with respect to the vertex ordering $v_{1},\hdots,v_{n}$.
Now if in the $i^{\text{th}}$ row $1$'s are consecutive we assign $a_{i}=i_{1}$ and $b_{i}=i_{2}$ and $p_{i}=i$ where $i_{1},i_{2}$ denote the first and last column containing $1$'s in the $i^{\text{th}}$ row respectively. Again if in the $i^{\text{th}}$ row zeros are consecutive we assign $a_{i}=i_{3}$, $b_{i}=i_{4}$ and $p_{i}=i$ where $i_{3}(i_{4})$ denote the first (last) column containing $1$'s after(before) the zeros stretch. Therefore $v_{i}v_{j}\in E$ if and only if $i_{1}\leq j\leq i_{2}$ in the first case and $i_{3}\leq j\leq n$ or $1\leq j\leq i_{4}$ in the second case. Now for each vertex $v_{i}\in V$ if we set circular-arc $I_{i}=[a_{i},b_{i}]$ and the corresponding point $p_{i}$ as defined above, then $\{(I_{i},p_{i})|v_{i}\in V\}$ clearly gives the circular-arc catch digraph representation of $G$. (note that as $a_{i,i}=1$, $p_{i}\in I_{i}$)
\end{proof}

\noindent We call the ordering described in above theorem as {\em circular catch ordering} of $V$. The ordering is not unique for a circular-arc catch digraph. Let $p_{i},p_{j}$ be two distinct points around the circle, then $[p_{i},p_{j}]$ denote the clockwise arc between these two points rounding the circle including the points $p_{i},p_{j}$.

\section{Oriented circular-arc catch digraphs}

\begin{defn}\label{CG}
A circular-arc catch digraph $G=(V,E)$ is said to be {\em oriented} \footnote{For an oriented catch digraph $G$, we can consider each point $p_{v}$ to be distinct otherwise if $p_{u}=p_{v}$ for some $u\neq v$ then $uv,vu\in E$ introduces contradiction.} if for all $u,v\in V$, $uv\in E$ implies $vu\notin E$. 
In other way we can say a simple digraph $G$ is an {\em oriented circular-arc catch digraph} if and only if $G$ is oriented and there exists an ordering of vertices of $V$ such that the augmented adjacency matrix $A(G)=(a_{i,j})$ satisfies circular $1$'s property along rows of it (see Theorem \ref{cop1}) and for each pair $i\neq j$, $a_{i,j}=1$ implies $a_{j,i}=0$.
\end{defn}

\begin{lem}\label{unilateral}
Let $G=(V,E)$ be an oriented circular-arc catch digraph having no directed cycle of length $n=|V|$. Then $G$ always contain a vertex of outdegree zero.
\end{lem}

\begin{proof}
On the contrary let us assume $G$ does not contain any vertex of outdegree zero. Let  $v_{1}\hdots v_{n}$ be a circular catch ordering of $V$. Then from Theorem \ref{cop1} it follows that the augmented adjacency matrix $A(G)=(a_{i,j})$ satisfies circular $1$'s property along rows with respect to this ordering. As $a_{i,i}=1$ for all $i\in \{1,2,\hdots,n\}$, one can easily verify now that at least one among $v_{1}v_{2},v_{1}v_{n}$ must belong to $E$.

\vspace{0.6em}

\noindent  If $v_{1}v_{2}\in E$ then $v_{2}v_{1}\notin E$ as $G$ is oriented, which imply $v_{2}v_{3}\in E$ otherwise $v_{2}$ becomes an outdegree zero vertex, which is not true from our assumption. Proceeding similarly we get $v_{3}v_{4},\hdots,v_{n-1}v_{n}\in E$. Hence $v_{n}v_{n-1}\notin E$ and therefore $v_{n}v_{1}\in E$. Thus we get a directed $n$-cycle $(v_{1},\hdots,v_{n})$ in $G$ which is a contradiction.

\noindent Also note that $v_{1}v_{2}$ and $v_{1}v_{n}$ can not be arcs of $E$ simultaneously. As when $v_{1}v_{2}\in E$, from above we get $v_{n}v_{1}\in E$. 
Hence $v_{1}v_{n}\notin E$ as $G$ is oriented.

\vspace{0.6em}

\noindent Next we consider the case when $v_{1}v_{n}\in E$. Then $v_{n}v_{1}\notin E$ as $G$ is oriented and hence $v_{n}v_{n-1}\in E$ otherwise $v_{n}$ becomes a vertex of out-degree zero. Now proceeding similarly as above starting with vertex $v_{n}$ one can able to find a directed $n$-cycle $(v_{n},\hdots, v_{1})$ in $G$, which again contradicts our assumption.

\vspace{0.6em}

\noindent Hence there must exist at least one vertex of out-degree zero in $G$. 
\end{proof}

\noindent Below we ensure the existence of hamiltonian path in an oriented circular-arc catch digraph when it is unilateral.

\begin{thm}
Let $G=(V,E)$ be an oriented circular-arc catch digraph which is unilateral. Then $G$ always possesses a hamiltonian path.
\end{thm}

\begin{proof}
Let $|V|=n$. If $G$ contains a directed $n$-cycle then it possesses a hamiltonian path clearly. Now if $G$ does not contain a directed $n$-cycle then from Lemma \ref{unilateral} it must contain a vertex (say $v_{o}$) of out-degree zero. 

\vspace{0.3em}
\noindent If the subdigraph induced by the vertices of $V\setminus \{v_{o}\}$ contains a directed $(n-1)$-cycle (say $C=(v_{1},\hdots,v_{i},v_{i+1}$,\\
$\hdots,v_{n})$), then it has a hamiltonian path containing its vertices. Now as $G$ is unilateral there must exist at least one vertex $v_{i}$ in $C$ which for which $v_{i}v_{o}\in E$. Hence $(v_{i+1},\hdots,v_{n},v_{1},\hdots, v_{i},v_{o})$ will be the required hamiltonian path of $G$ where $v_{i+1}$ is the vertex occurred right after $v_{i}$ in $C$ satisfying $v_{i}v_{i+1}\in E$. 

\vspace{0.3em}

\noindent Again if the subdigraph induced by $V\setminus \{v_{o}\}$ does not contain any directed $(n-1)$-cycle then it must possesses a vertex (say $v_{l}$) of zero out-degree by Lemma \ref{unilateral}. Clearly $v_{l}$ has out-degree one in $G$ otherwise there can not exist any directed path between $v_{o},v_{l}$, which contradicts the fact that $G$ is unilateral. Therefore $v_{l}v_{o}\in E$. Now if the subdigraph induced by $V\setminus \{v_{o},v_{l}\}$ contain a directed $(n-2)$-cycle (say $C^{'}=(v_{1},\hdots,v_{j},v_{j+1},\hdots,v_{n})$), then there is no directed path between a vertex of $C^{'}$ and $v_{l}$ which pass through $v_{o}$ in $G$.
But as $G$ is unilateral there must exist at least one vertex (say $v_{j}$) in $C^{'}$ for which $v_{j}v_{l}\in E$. Therefore $(v_{j+1},\hdots,v_{n},v_{1},\hdots,v_{j},v_{l},v_{0})$ is the required hamiltonian path of $G$ where $v_{j+1}$ is the vertex that occurred immediately after $v_{j}$ in $C^{'}$. 

\vspace{0.3em}
\noindent Again if the subdigraph formed by $V\setminus \{v_{o},v_{l}\}$ does not contain any directed $(n-2)$-cycle, then it must contain a vertex (say $v_{r}$) of out-degree zero and as the graph $G$ is oriented and unilateral we have $v_{r}v_{l}\in E$ and hence there exists a directed path $(v_{r},v_{l},v_{0})$ from $v_{r}$ to $v_{0}$. Next we continue the process as before. As $|V|=n$, proceeding similarly after some finite number of steps one can able to find a hamiltonian path of $G$.
\end{proof}

\begin{lem}\label{3cycle}
Let $\vec{C}=(v_{1},v_{2},v_{3})$ be a directed $3$-cycle. Then $\vec{C}$ is a circular-arc catch digraph
and in any circular-arc catch digraph representation of it, arcs corresponding to the vertices of 
$\vec{C}$ must cover the whole circle. 
\end{lem}

\begin{proof}
One can observe that $v_{1}\prec v_{2}\prec v_{3}$ or any of its cyclic permutation (i.e., $v_{2}\prec v_{3}\prec v_{1}$ and $v_{3}\prec v_{1}\prec v_{2}$) are the possible circular catch orderings of $V(\vec{C})$. Hence it is easy to verify that any of these vertex orderings satisfy Theorem \ref{cop1}. Therefore $\vec{C}$ is a circular-arc catch digraph. 

\noindent Let $v_{1}\prec v_{2}\prec v_{3}$ be a circular catch ordering of $V(\vec{C})$, then $p_{1}<p_{2}<p_{3}$ where $<$ is the increasing (clockwise) ordering of the points around a circle. Now as $v_{1}v_{2},v_{2}v_{3},v_{3}v_{1}\in E(\vec{C})$ and $\vec{C}$ is oriented, the result follows. Similar logic also hold true for other for other circular catch orderings of $V(\vec{C})$. 
\end{proof}

\noindent Prisner in \cite{Prisner} noted a lot of interesting characteristics of interval catch digraphs. One among them is that the underlying graph of an interval catch digraph is always perfect. But we find that this result does not hold true for oriented circular-arc catch digraphs as any directed cycle of length $n\geq 3$ belongs to this graph class having circular-arc catch digraph representation $I_{i}=[i,i+1]$, $p_{i}=i$ for all $i=1,2,\hdots,n-1$, $I_{n}=[n,1]$, $p_{n}=n$. In fact there exists an oriented circular-arc catch digraph $D$ with representation $I_{1}=[3,6], p_{1}=5, I_{2}=[1.9,2.1], p_{2}=2, I_{3}=[2.9,3.1], p_{3}=3, I_{4}=[5.9,2], p_{4}=6, I_{5}=[2,4.1], p_{5}=4, I_{6}=[0.9,3],p_{6}=1$ for which its underlying graph is $\overline{C_{6}}$. Again the oriented graph $D^{'}$ with circular-arc catch digraph representation $I_{1}=[4.9,7.1], p_{1}=5, I_{2}=[0.9,3.1], p_{2}=1, I_{3}=[3.9,6.1], p_{3}=4, I_{4}=[6.9,2.1], p_{4}=7, I_{5}=[2.9,5.1], p_{5}=3, I_{6}=[5.9,1.1],p_{6}=6, I_{7}=[1.9,4.1], p_{7}=2$ possesses its underlying graph as $\overline{C_{7}}$. However, we observe that complement of a cycle having length greater than $7$ can never be an underlying graph of an oriented circular-arc catch digraph.

\begin{lem}\label{c9}
Let $G$ be an oriented circular-arc catch digraph. Then the underlying graph of $G$, i.e., $U(G)$ does not contain $\overline{C_{8}}$ or $\overline{C_{9}}$ as an induced subgraph.
\end{lem}

\begin{proof}
Let $G=(V,E)$ be an oriented circular-arc catch digraph with circular-arc catch digraph representation $\{(I_{v}=[a_{v},b_{v}],p_{v})|v\in V\}$. Then there exists an ordering (say $<_{\sigma}$) of vertices of $V$ so that the augmented adjacency matrix $A(G)$ satisfies Theorem \ref{cop1}.
Let us assume on the contrary that $U(G)$ contains $\overline{C_{9}}$ as induced subgraph. Let $v_{1},\hdots, v_{9}$ denote the vertices occurred in circularly consecutive way (clockwise or anticlockwise) in $C_{9}$. Now consider the vertices of the set $S=\{v_{2},v_{6},v_{4},v_{9}\}$ in $G$.
It is evident from the adjacencies of the vertices of $S$ that they form a tournament. Also note that any three vertices of $S$ form $C_{3}$ in $U(G)$.


\vspace{0.2em}

\noindent We first consider the $3$-cycle $C=\{v_{2},v_{6},v_{4}\}$. Now if $C$ is directed then according to Lemma \ref{3cycle}, $I_{2}\cup I_{6}\cup I_{4}$ must cover the whole circle. Let $v_{2}v_{6},v_{6}v_{4},v_{4}v_{2}\in E$. Without loss of generality we assume $v_{2}<_{\sigma} v_{6}<_{\sigma} v_{4}$ and hence $p_{1}<p_{6}<p_{4}$. Then as $v_{5}$ is not adjacent to $v_{4},v_{6}$ in $U(G)$, $p_{5}\notin I_{4}, I_{6}$. Therefore $p_{5}$ must belong to $I_{2}$. Again as $G$ is oriented, $p_{2}\notin I_{5}$. Infact as $v_{5},v_{6}$ are nonadjacent in $U(G)$,

\begin{equation}\label{c1}
I_{5} \subset (p_{2},p_{6})
\end{equation}

\noindent Similarly one can show $I_{3}\subset (p_{6},p_{4})\subset I_{6}$ as $v_{3}$ is nonadjacent to $v_{2},v_{4}$ in $U(G)$ and $v_{3}v_{6}\notin E$ as $G$ is oriented. Now as $(p_{6},p_{4})\subset (p_{6},p_{2})$, from above we get 

\begin{equation}\label{c2}
I_{3}\subset (p_{6},p_{2})
\end{equation}

\noindent Therefore $v_{3},v_{5}$ become nonadjacent in $U(G)$ follows from (\ref{c1}), (\ref{c2}). Hence contradiction arises. Again when $v_{6}v_{2},v_{2}v_{4},v_{4}v_{6}\in E$, the same contradiction occurs.







\vspace{0.2em}

\noindent Similarly one can show that $\{v_{2},v_{6},v_{9}\}, \{v_{2},v_{4},v_{9}\}$ or $\{v_{4},v_{6},v_{9}\}$ can not be directed $3$-cycles in $G$. Hence we can conclude that any $3$-cycle formed by the vertices of the set $S$ must be transitively oriented in $G$. Therefore each of the vertices of $S$ possesses different in-degree and out-degree restricted on $S$. Without loss of generality we assume $v_{2}v_{6},v_{2}v_{4},v_{2}v_{9},v_{6}v_{4},v_{6}v_{9},v_{4}v_{9}\in E$.  
Then the possible circular catch orderings of the vertices of $S$ satisfying 
Theorem \ref{cop1} are $i)$ $v_{2}\prec v_{6}\prec v_{9}\prec v_{4}$, $ii)$ $v_{2}\prec v_{4}\prec v_{9}\prec v_{6}$, $iii)$ $v_{2}\prec v_{6}\prec v_{4}\prec v_{9}$, $iv)$ $v_{2}\prec v_{9}\prec v_{4}\prec v_{6}$ where $\prec=<_{\sigma}|_{S}$ and their cyclic permutations. 






\vspace{0.6em}


\noindent $i)$ First we consider the case when $\mathbf{v_{2}\prec v_{6}\prec v_{9}\prec v_{4}}$. 
\vspace{0.6em}

\noindent Then $p_{2}< p_{6}< p_{9}< p_{4}$. Now $v_{4}v_{9},v_{6}v_{9}\in E$ imply $v_{9}v_{4},v_{9}v_{6}\notin E$ as $G$ is oriented. Therefore $p_{4},p_{6}\notin I_{9}$. Hence,

\begin{equation}\label{c5}
I_{9}\subset (p_{6},p_{4})
\end{equation}

\noindent Note that $v_{5}$ is nonadjacent to $v_{6},v_{4}$ in $U(G)$. Now if
$I_{5}\subset (p_{6},p_{4})$ then $p_{5}\in I_{5}\subset (p_{6},p_{4})\subseteq I_{6}$ as $v_{6}v_{9},v_{6},v_{4}\in E$ and $v_{6}v_{2}\notin E$, which is a contradiction. Therefore,

 
\begin{equation}\label{cc6}
 I_{5}\subset (p_{4},p_{6})
\end{equation}
  
\noindent Hence $v_{5},v_{9}$ become nonadjacent in $U(G)$ from (\ref{c5}) and (\ref{cc6}). Thus contradiction arises. 

\vspace{0.6em}
\noindent $\mathbf{ii)}$ Similar argument as case $(i)$ will lead to contradiction when $\mathbf{v_{2}\prec v_{4}\prec v_{9}\prec v_{6}}$.

\vspace{0.7em}

\noindent $\mathbf{iii)}$ Next we consider the case when $\mathbf{v_{2}\prec v_{6}\prec v_{4}\prec v_{9}}$.

\vspace{0.3em}
\noindent In this case $p_{2}< p_{6}< p_{4}< p_{9}$ clearly. 
\vspace{0.3em}

\noindent a) Now if $I_{2}$ intersects $p_{9}$ anticlockwise and intersects $p_{6}$ clockwise, then $I_{2}\cup I_{6}$ covers the whole circle as $v_{6}v_{9}\in E$ and $v_{6}v_{2}\notin E$ follows from the fact $v_{2}v_{6}\in E$ and $G$ is oriented. 
Now as $v_{3}$ is nonadjacent to $v_{2}$ in $U(G)$, $p_{3}\notin I_{2}$. Therefore $p_{3}\in I_{6}$.
Again $v_{4}v_{9}\in E$ imply $(p_{4},p_{9})\subseteq I_{4}$. But as $v_{3},v_{4}$ are nonadjacent 
$p_{3}$ can only belong to the arc $(p_{6},p_{4})$. Therefore 

\begin{equation}\label{c7}
I_{3}\subset (p_{6},p_{4})  
 \end{equation}
 
\noindent as $G$ is oriented and $p_{4}\notin I_{3}$. Again from adjacency of $v_{9}$ it is clear that $I_{9}\subset (p_{4},p_{2})\subset (p_{4},p_{6})$. Hence $v_{3},v_{9}$ become nonadjacent in $U(G)$ from (\ref{c7}), which is a contradiction.

\vspace{0.2em}

\noindent b) Next we consider the case when $I_{2}$ intersects $p_{6},p_{4},p_{9}$ clockwise around the circle. As $v_{1},v_{2}$ are nonadjacent in $U(G)$, $p_{1}\in (p_{9},p_{2})$. In particular 

\begin{equation}\label{c8}
I_{1}\subset (p_{9},p_{2}) \hspace{0.3em} \mbox{and} \hspace{0.3em} I_{9}\subset (p_{4},p_{1})
\end{equation}

\noindent as $v_{1},v_{9}$ are nonadjacent and $p_{4}\notin I_{9}$ as $G$ is oriented and $v_{4}v_{9}\in E$. Now as $v_{6},v_{1}$ are adjacent in $U(G)$, $v_{6}v_{1}\in E$ as $I_{1}\subset (p_{9},p_{2})$ from (\ref{c8}). Again $p_{7}\notin I_{6}$ imply $p_{7}\in (p_{1},p_{2})$ or $(p_{2},p_{6})$.
If $p_{7}\in (p_{2},p_{6})$ then

\begin{equation}
I_{7}\subset (p_{2},p_{6})\subset (p_{1},p_{4})
\end{equation}
   
 \noindent which contradicts the adjacency between $v_{7},v_{9}$ in $U(G)$ from (\ref{c8}). Hence $p_{7}\in (p_{1},p_{2})$.  Clearly, $p_{7}\notin I_{9}$ as $I_{9}\subset (p_{4},p_{1})$ from (\ref{c8}). Now as $v_{7},v_{9}$ are adjacent in $U(G)$, $v_{7}v_{9}\in E$. Further one can note that $v_{7}v_{1}\in E$ as $p_{1}\in(p_{9},p_{7})$ and $v_{7}v_{9}\in E$. This imply 
 
\begin{equation} \label{c99}
 I_{1}\subset (p_{9},p_{7})\subseteq I_{7}
\end{equation}
 
\noindent as $G$ is oriented and $\{v_{1},v_{9}\}$, $\{v_{7},v_{6}\}$ are nonadjacent in $U(G)$.
Now as $p_{8}\notin I_{7}$, $p_{8}\in(p_{7},p_{9})$ from (\ref{c99}). Then $I_{8}\subset(p_{7},p_{9})$ as $G$ is oriented and $v_{8}$ is nonadjacent to $v_{7},v_{9}$. This contradicts the adjacency between $v_{8},v_{1}$ in $U(G)$ from (\ref{c99}).

\vspace{0.2em}

\noindent c) Again if $I_{2}$ intersects $p_{9},p_{6},p_{4}$ anticlockwise around the circle. Then since $p_{3}\notin I_{2}$, $p_{3}\in (p_{2},p_{6})$. More specifically  $I_{3}\subset (p_{2},p_{4})$ as $v_{3}$ is nonadjacent to $v_{2},v_{4}$. But this again contradict the adjacency between $v_{3},v_{9}$ as $I_{9}\subset (p_{4},p_{2})$ as $v_{4}v_{9}\in E$ and $G$ is oriented.

\vspace{0.6em}
\noindent  $\mathbf{iv)}$ When $\mathbf{v_{2}\prec v_{9}\prec v_{4}\prec v_{6}}$,
 a similar contradiction will arise as in case $(iii)$. 

\vspace{0.6em}
\noindent Now proceeding similarly as the case of $\overline{C_{9}}$ we able to find contradiction for $\overline{C_{8}}$. Due to its lengthy verification, we leave the details for the readers.
\end{proof}

\begin{lem}\label{c10}
Let $G$ be an oriented circular-arc catch digraph. Then $U(G)$ does not contain 
$\overline{C_{9}}\setminus \{v\}$ as an induced subgraph where $v$ is any vertex of $\overline{C_{9}}$.
\end{lem}

\begin{proof}
The result follows directly from proof of Lemma \ref{c9} as we do not need nine vertices to get contradiction in any of the cases. 
\end{proof}

\begin{thm}
Let $G$ be an oriented circular-arc catch digraph. Then $U(G)$ does not contain any $\overline{C_{n}}$ for $n>7$ as an induced subgraph. 
\end{thm}

\begin{proof}
As $G=(V,E)$ is a circular-arc catch digraph, there must be an ordering (say $\sigma$) of vertices of $V$ satisfying Theorem \ref{cop1}. On the contrary let us assume $U(G)$ contains $\overline{C_{n}}, n>7$ as induced subgraph. Let $v_{1},\hdots, v_{n}$ denote the vertices occurred in circularly consecutive way (clockwise or anticlockwise) in $C_{n}$. We choose eight consecutive vertices from them. Then the subgraph induced by these vertices in $U(G)$ is same as $\overline{C_{9}}\setminus \{v\}$ where $v$ is any vertex of $\overline{C_{9}}$.  Rest of the proof follows from Lemma \ref{c9} and Lemma \ref{c10}.  
\end{proof}

\begin{defn}\label{outin}
Let $G=(V,E)$ be a simple digraph. Then $A \subseteq V$ is said to form an {\em out-fountain} if there exists a total order ($\prec$ say) of the vertices of $A$ such that for each $v_i\prec v_j,$ $v_{i}v_{j}\in E$ (it is referred as transitive tournament in \cite{Bang}). Similarly $B\subseteq V$ is said to form an {\em in-fountain} if there exists a total order of the vertices of $B$ such that for each $v_{j}\prec v_{i},$ $v_{i}v_{j}\in E$. A set is said to be a {\em fountain}\index{fountain} if and only if it is either an {\em out-fountain} or an {\em in-fountain}. Further we denote the first element by $v_{f}$ and the last by $v_{l}$ of any fountain. Note that if we reverse the order of an out-fountain, it becomes an in-fountain and vice versa. Moreover any fountain must be a tournament. 
\end{defn}



\begin{lem}\label{Out}
Let $G=(V,E)$ be an proper oriented ICD which is a tournament. Then the set of all vertices of $G$, i.e., $V(G)$ forms an out-fountain in $G$.
\end{lem}

\begin{proof}
As $G$ is a proper oriented ICD, it must possesses a proper oriented ICD ordering (say $P$) for the vertices of it from Theorem \ref{poicd1}. 
\vspace{0.2em}

\noindent Let $v_{i}<v_{j}<v_{k}$ in $P$ such that $v_{i}v_{j}\in E$. We show $v_{i}v_{k}, v_{j}v_{k}\in E$. If not, then $v_{k}v_{i}\in E$ as $G$ is a tournament, which imply $v_{j}v_{i}\in E$ as $P$ is a proper oriented ICD ordering. Hence contradiction arises as $v_{i}v_{j}\in E$ and $G$ is oriented. Therefore $v_{i}v_{k}\in E$ and hence $v_{j}v_{k}\in E$. Similarly one can show that if $v_{j}v_{i}\in E$ then $v_{k}v_{i},v_{k}v_{j}\in E$. 

\vspace{0.2em}
\noindent Therefore if $P$ is $v_{1}<v_{2}<\hdots<v_{n}$ and $v_{1}v_{2}\in E$ we take the ordering $P$ and if $v_{2}v_{1}\in E$ then we take the reverse ordering of $P$. Now from above verification it is clear now that $V(G)$ forms an out-fountain with respect to $P$ or reverse of it in $G$.
\end{proof}

\begin{figure}\centering 
\includegraphics[scale=0.7]{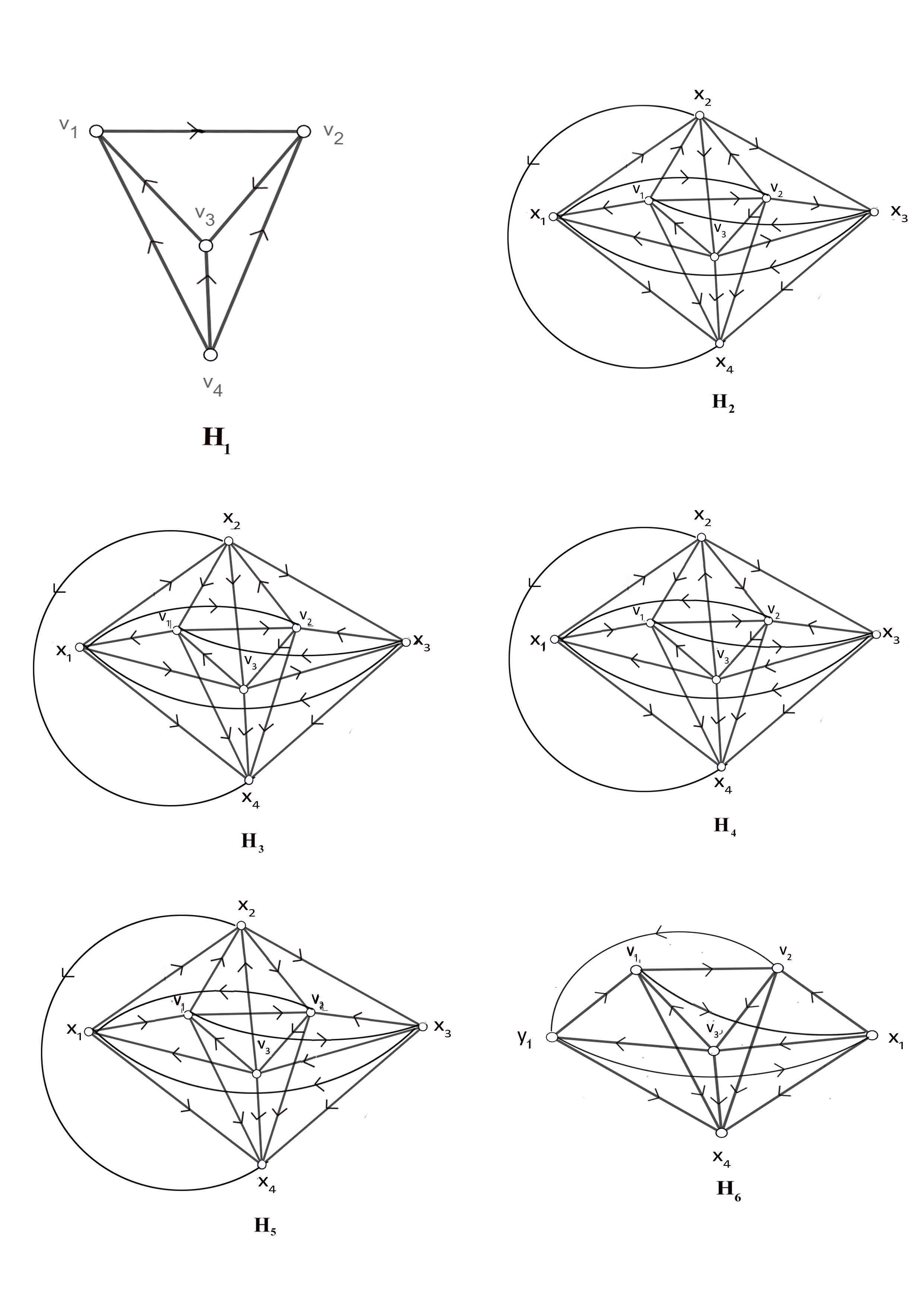}
\caption{Forbidden digraphs of an oriented circular-arc catch digraph which is a tournament}\label{forbid6}
\end{figure}

\begin{thm}\label{d3cycle}
Let $G=(V,E)$ be an oriented graph which is a tournament. Then $G$ is an interval catch digraph if and only if $G$ does not contain any directed $3$-cycle as induced subdigraph.
\end{thm}

\begin{proof}
If $G$ is an interval catch digraph, then from Theorem \ref{dag} it follows that $G$ is a directed acyclic graph. Therefore $G$ can not contain any directed $3$-cycle as induced subdigraph.

\vspace{0.6em}
\noindent Conversely, let $G$ does not contain any directed $3$-cycle as induced subdigraph. Then each $3$-cycle of $U(G)$ is transitively oriented in $G$. To prove $G$ is an interval catch digraph, it is sufficient to show that $G$ is diasteroidal triple-free. Then from Theorem \ref{diate} the result will follow. Now on the contrary, let $\{x,y,z\}$ be a diasteroidal triple in $G$. As $G$ is a tournament, $\{x,y,z\}$ must form a transitively oriented $3$-cycle $C$ in $G$
and hence it must contain one vertex (say $x$) of out-degree $2$ in $C$. Now from the definition of diasteroidal triple there must exist a chain $P=(x,x_{1},\hdots,x_{n-1}, y)$ in $G$ between $x,y$ which is avoided by $z$. Now as $xz\in E$ and $z$ is avoiding the path $P$, $x_{1}x,zx_{1}\in E$. Therefore $\{z,x_{1},x\}$ form a directed $3$-cycle in $G$ which contradicts our assumption. Hence the result follows.
\end{proof}

\noindent Below we characterize those oriented circular-arc catch digraphs which are tournaments, in terms of forbidden subdigraphs.

\begin{thm}
Let $G=(V,E)$ be an oriented graph which is a tournament. Then $G$ is a circular-arc catch digraph if and only if it does not contain $H_{1}, H_{2}, H_{3}, H_{4}, H_{5}, H_{6}$ (cf. Figure \ref{forbid6}) as induced subdigraph. 
\end{thm}

\begin{proof}
\textit{Necessity:} Let $G$ be a circular-arc catch digraph with representation
 $\{(I_{v},p_{v})|v\in V\}$ where $I_{v}=[a_{v},b_{v}]$ is the circular-arc and $p_{v}$ is a point within $I_{v}$ attached to the vertex $v$. On the contrary let us assume that $G$ contains any of the digraphs of Figure \ref{forbid6} as induced subdigraph. It is easy to observe that $\vec{C}=\{v_{1},v_{2},v_{3}\}$ forms a directed $3$-cycle in each of the above digraphs. Hence from Lemma \ref{3cycle}, $I_{v_{1}}\cup I_{v_{2}}\cup I_{v_{3}}$ must cover the whole circle. Hence any point $p_{v}$ for $v\in V\setminus V(\vec{C})$ either belong to $I_{v_{1}}$ or $I_{v_{2}}$ or $I_{v_{3}}$. Without loss of generality we assume $v_{1}<_{\sigma} v_{2}<_{\sigma} v_{3}$ and therefore $p_{v_1}< p_{v_2}<p_{v_3}$ where $\sigma$ is a circular catch ordering of $V$. 

\vspace{0.7em}
\begin{itemize}

\item \noindent In $H_{1}$ as $v_{4}v_{1},v_{4}v_{2},v_{4}v_{3}\in E$ and $G$ is oriented, $v_{1}v_{4},v_{2}v_{4},v_{3}v_{4}\notin E$. Therefore $p_{v_{4}}\notin I_{v_{1}}$, $I_{v_{2}}$ and $I_{v_{3}}$. But this is impossible from above. 

\vspace{0.7em}

\item \noindent Again note that in $H_{2}, H_{3}$, $x_{4}$ has all the vertices of $\vec{C}$ as its in-neighbor. Therefore $p_{x_{4}}\in I_{v_1}\cap I_{v_2}\cap I_{v_3}$. Without loss of generality we assume $p_{x_4}$ to be placed within the arc
$(p_{v_1},p_{v_2})$.

\vspace{0.7em}

\noindent In $H_{2}$, $x_{2}v_{3}\in E$ imply $v_{3}x_{2}\notin E$, which imply $p_{x_{2}}\notin I_{v_3}$. Therefore  $p_{x_{2}}\notin(p_{v_3},p_{v_1})$. Again as $v_{1}x_{2},v_{2}x_{2}\in E$, $p_{x_{2}}\in I_{v_1}\cap I_{v_2}\subset (p_{v_1},p_{v_3})$. Now if $p_{x_{2}}\in(p_{v_1},p_{v_2})$ then $I_{x_{2}}\subset (p_{v_1},p_{v_2})$ as $H_{2}$ is oriented. Therefore $x_{2}v_{3}\notin E$. Again when $p_{x_{2}}\in(p_{v_2},p_{v_3})$ then $I_{x_{2}}\subset (p_{v_2},p_{v_1})$ and hence $x_{2}x_{4}\notin E$ as $p_{x_{4}}\in (p_{v_{1}},p_{v_{2}})$ and $H_{2}$ is oriented. Hence contradiction arises in both cases.

\vspace{0.7em}

\noindent Next in $H_{3}$ as $v_{1}x_{1}\in E$, $p_{x_{1}}\in I_{v_1}$. Again as $x_{1}v_{2},x_{1}v_{3}\in E$, $v_{2}x_{1},v_{3}x_{1}\notin E$ which imply $p_{x_{1}}\notin I_{v_2}, I_{v_3}$ as $H_{3}$ is oriented. In particular $p_{x_{1}}\in(p_{v_1},p_{v_2})$. But this situation can not happen, as in that case either $p_{x_{1}}\in I_{v_{1}}\cap I_{v_{2}}$ or $I_{v_{1}}\cap I_{v_{2}}\cap I_{v_{3}}$ or $I_{v_{1}}\cap I_{v_{3}}$ due to the fact 
$p_{x_{4}}\in(p_{v_1},p_{v_2})$.

\vspace{0.7em}

\noindent As $x_{1},x_{2},x_{3}$ has same degree nature (in-degree $2$, out-degree $1$ in $H_{2}$ and in-degree $1$, out-degree $2$ in $H_{3}$) with vertices of  
$\vec{C}$ and $x_{i}x_{4}\in E$ for all $i\in\{1,2,3\}$, if we place $p_{x_{4}}$ within $(p_{v_2},p_{v_3})$ or $(p_{v_3},p_{v_1})$ we get similar contradiction.

\vspace{0.7em}

\item \noindent Next in $H_{4}$ as $v_{2}x_{1},v_{3}x_{1}\in E$ and $v_{1}x_{1}\notin E$, $p_{x_{1}}\in I_{v_2}\cap I_{v_3}\subset (p_{v_2},p_{v_1})$. Now if $p_{x_{1}}\in(p_{v_2},p_{v_3})$ then $I_{x_{1}}\subset (p_{v_2},p_{v_3})$ as $H_{4}$ is oriented. Therefore $x_{1}v_{1}\notin E$. Hence                                                                        $p_{x_{1}}\in (p_{v_3},p_{v_1})$.

\vspace{0.2em}

\noindent  Again $v_{3}x_{2}\in E$ imply $p_{x_{2}}\in I_{v_3}$.
As $x_{2}v_{1},x_{2}v_{2}\in E$, $v_{1}x_{2},v_{2}x_{2}\notin E$. This imply $p_{x_{2}}\notin I_{v_1}, I_{v_2}$. In particular $p_{x_{2}}\in (p_{v_3},p_{v_1})$. From the adjacency of $x_{1},x_{2}$ it is clear now that $p_{x_{1}}$ occurs prior (anticlockwise) to $p_{x_{2}}$ within the arc $(p_{v_3},p_{v_1})$. Again $v_{1}x_{3},x_{3}v_{2},x_{3}v_{3}\in E$ imply $p_{x_{3}}\in I_{v_1}$ but $p_{x_{3}}\notin I_{v_2}, I_{v_3}$ as $H_{4}$ is oriented. Therefore $p_{x_{3}}\in (p_{v_1},p_{v_2})$.

\vspace{0.2em}
\noindent Hence from the positions of $p_{x_{2}},p_{x_{3}}$ as placed in the above paragraph, it is clear that $p_{x_{4}}$ can not occur within $(p_{v_3},p_{v_1})$ or $(p_{v_1},p_{v_2})$ as $p_{x_{4}}\in I_{v_1}\cap I_{v_2}\cap I_{v_3}$. Therefore $p_{x_{4}}$ must belong to $(p_{v_2},p_{v_3})$. But this is also not possible as in that case $x_{1}x_{4}\notin E$
as $p_{x_{1}}\in I_{v_2}\cap I_{v_3}\cap (p_{v_3},p_{v_1})$ and $H_{4}$ is oriented.

\vspace{0.45em}
\item \noindent Again in $H_{5}$, $v_{1}x_{3}\in E$ imply $p_{x_{3}}\in I_{v_1}$. Again $x_{3}v_{2},x_{3}v_{3}\in E$ imply $v_{2}x_{3},v_{3}x_{3}\notin E$ as $H_{5}$ is oriented. Hence $p_{x_{3}}\notin I_{v_2},I_{v_3}$. In particular $p_{x_{3}}\in (p_{v_1},p_{v_2})$. From the adjacency of $x_{3}$, it is clear that $p_{x_{4}}\notin (p_{v_1},p_{v_2})$ as $x_{4}\in I_{v_1}\cap I_{v_2}\cap I_{v_3}$. Therefore $p_{x_{4}}\in (p_{v_2},p_{v_3})$ or $(p_{v_3},p_{v_1})$.

\vspace{0.2em}

\noindent Again $v_{1}x_{2},v_{3}x_{2}\in E$ and $v_{2}x_{2}\notin E$, imply $p_{x_{2}}\in I_{v_1}\cap I_{v_3}\subset (p_{v_3},p_{v_2})$. If $p_{x_{2}}\in (p_{v_3},p_{v_1})$ then $x_{2}v_{2}\notin E$ as $I_{x_{2}}\subset (p_{v_3},p_{v_1})$ due to the fact $H_{5}$ is oriented. Therefore $p_{x_{2}}\in (p_{v_1},p_{v_2})$. Using similar type argument one can show that $p_{x_{1}}\in (p_{v_3},p_{v_1})$ as $v_{2}x_{1},v_{3}x_{1},x_{1}v_{1}\in E$. 

\vspace{0.2em}

\noindent  Now if $p_{x_{4}}\in (p_{v_2},p_{v_3})$ then $x_{1}x_{4}\notin E$ and if $p_{x_{4}}\in (p_{v_3},p_{v_1})$ then $x_{2}x_{3}\notin E$ as $H_{5}$ is oriented. Hence in both of these cases contradiction arises.

\vspace{0.45em}

\item \noindent Now in $H_{6}$ as $v_{1}x_{1}\in E$, $p_{x_{1}}\in I_{v_1}$. Again $x_{1}v_{2},x_{1}v_{3}\in E$ imply $v_{2}x_{1},v_{3}x_{1}\notin E$, which imply $p_{x_{1}}\notin I_{v_2}, I_{v_3}$. More specifically $p_{x_{1}}\in (p_{v_1},p_{v_2})$. Therefore $p_{x_{4}}\notin (p_{v_1},p_{v_2})$ as $x_{4}\in I_{v_1}\cap I_{v_2}\cap I_{v_3}$ follows from the adjacency of $x_{4}$.

\vspace{0.2em}

\noindent Again $v_{2}y_{1}, v_{3}y_{1}\in E$ imply $p_{y_{1}}\in I_{v_2}\cap I_{v_3}\subset (p_{v_2},p_{v_1})$. Now if $p_{y_{1}}\in (p_{v_2},p_{v_3})$, then $I_{y_{1}}\subset (p_{v_2},p_{v_3})$ as $H_{6}$ is oriented. But in that case $y_{1}x_{1}\notin E$. Therefore $p_{y_{1}}\in (p_{v_3},p_{v_1})$. 

\vspace{0.2em}

\noindent Now if $p_{x_{4}}\in (p_{v_2},p_{v_3})$, then
$y_{1}x_{4}\notin E$. Hence $p_{x_{4}}\in (p_{v_3},p_{v_1})$. It is evident from the adjacency of $y_{1},x_{4}$ that $p_{y_{1}}$ occurs prior (anticlockwise) to $p_{x_{4}}$ within $(p_{v_3},p_{v_1})$. 
Now as $x_{1}x_{4}\in E$, $x_{1}y_{1}$ forms a directed edge in $G$.
But this is not true as $y_{1}x_{1}\in E$ and $H_{6}$ is oriented.

 \end{itemize}
    
\vspace{0.7em}

\noindent Hence $G$ can not contain any of the digraphs of Figure \ref{forbid6} as induced subdigraph.

\vspace{0.7em}
\begin{figure} \centering
\includegraphics[scale=0.52]{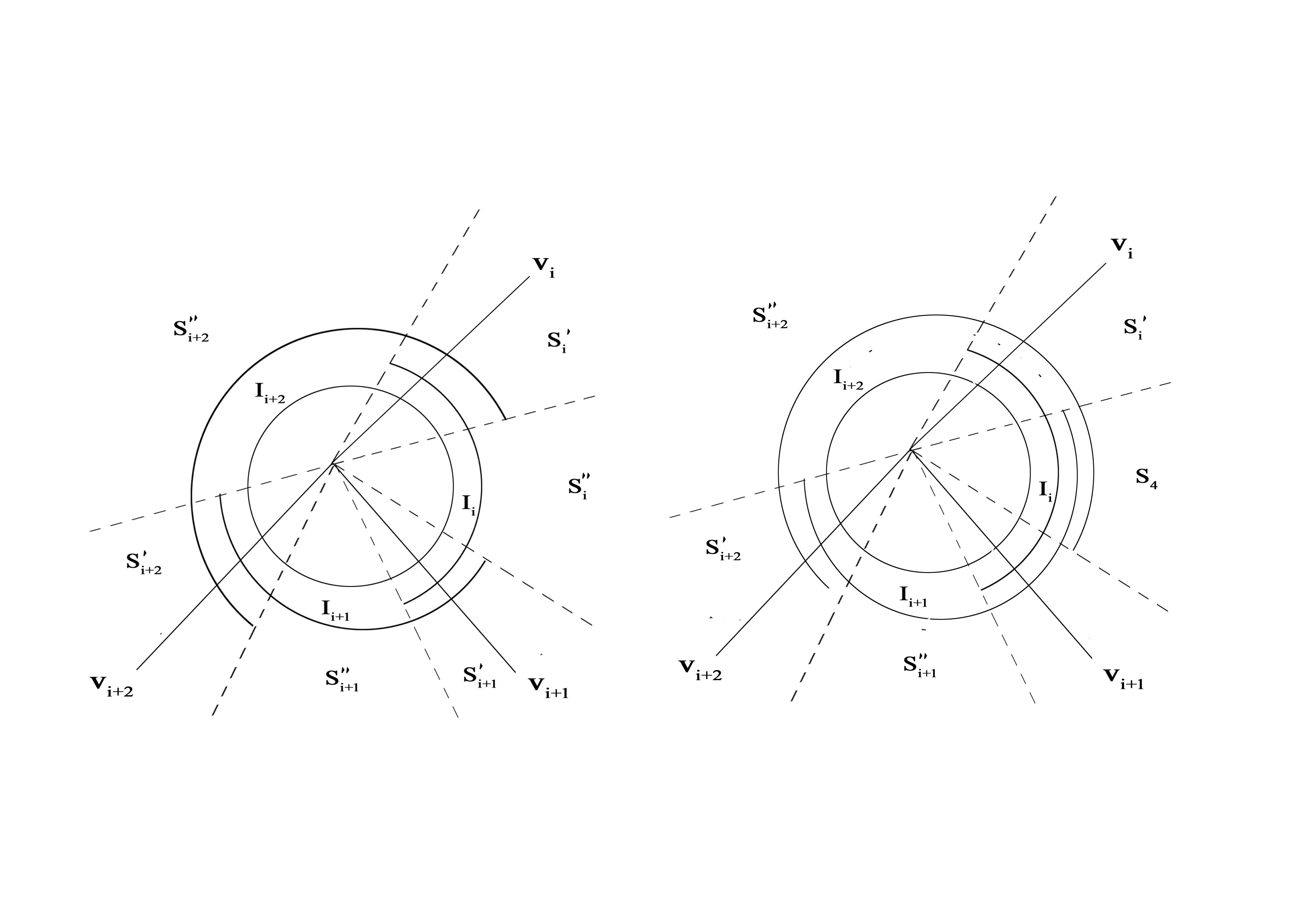}
\caption{Positions of the vertices in $V_{S_{P}}$ when $S_{4}=\emptyset$ (left), when $S_{4}\neq \emptyset$ and $S_{i}^{''}=S^{'}_{i+1}=\emptyset$ (right)}\label{cr}

\includegraphics[scale=0.477]{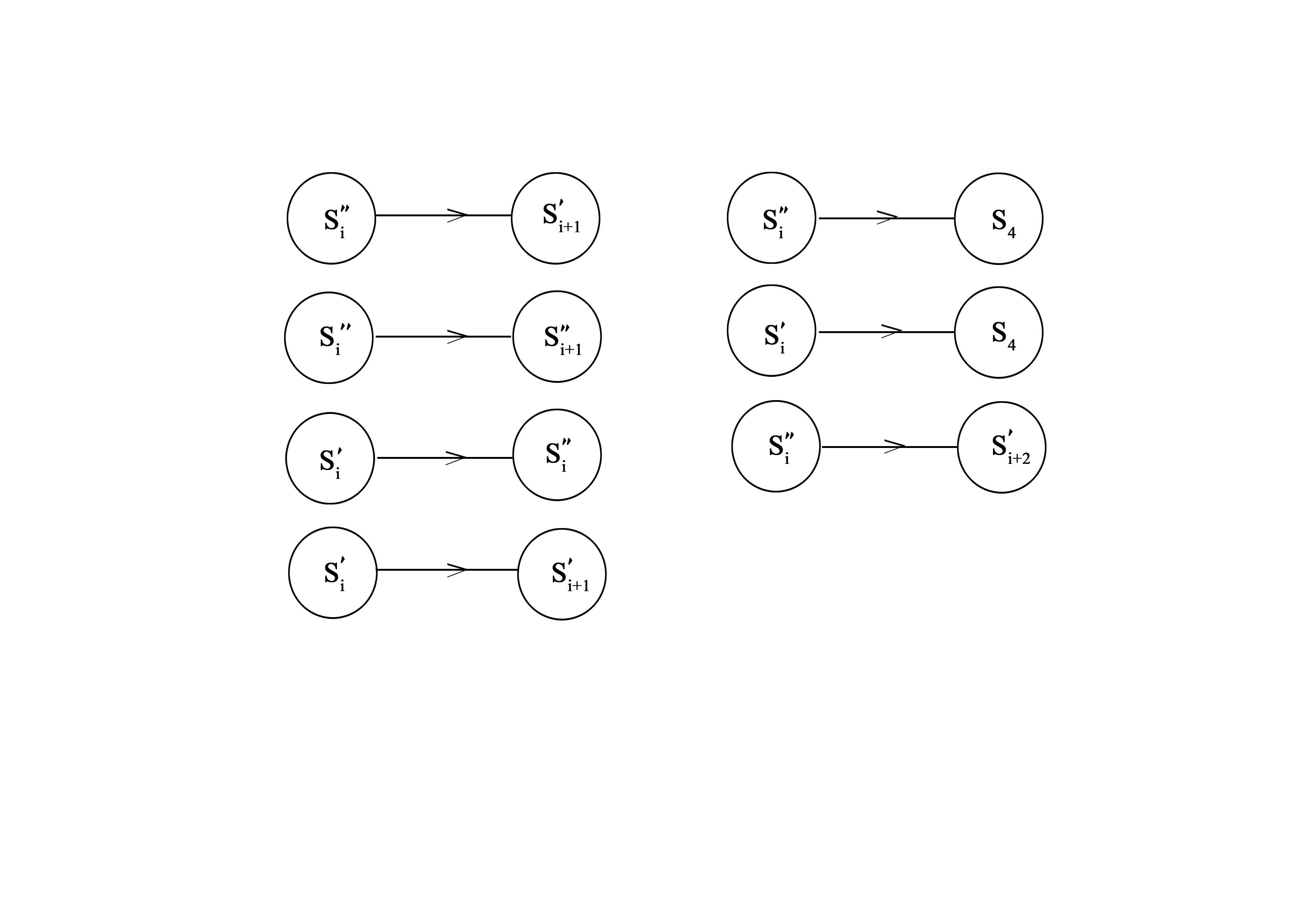}
\caption{Adjacency relations between vertices of $S^{'}_{i}, S^{''}_{i}$ (left), $S_{4},S^{'}_{i}, 
S^{''}_{i}$ (when $S_{4}\neq \emptyset$) (right)}\label{tr}                                
\end{figure}


\noindent \textit{Sufficiency:} Conversely, let $G$ does not contain $H_{1}, H_{2}, H_{3}, H_{4}, H_{5}, H_{6}$ as induced subdigraph. Now two cases may arise.

\vspace{0.7em}
\noindent \textbf{Case (i)}
\vspace{0.7em}

\noindent  Let $G$ contains a directed $3$-cycle $\vec{C}=(v_{1},v_{2},v_{3})$ with arcs $v_{1}v_{2},v_{2}v_{3},v_{3}v_{1}$ from $E$. Therefore from Lemma \ref{3cycle},  $I_{v_{1}}\cup I_{v_{2}}\cup I_{v_{3}}$ covers the whole circle. As $H_{1}$ can not be an induced subdigraph of $G$ each vertex of $V\setminus V(\vec{C})$ must possess at least one in-neighbor from $\vec{C}$.

\vspace{0.7em}
\noindent We consider $\{v_{i},v_{i+1}, v_{i+2}\}$ as a cyclic permutation \footnote{cyclic permutations of the set $\{1,2,3\}$ are $(1,2,3),(2,3,1),(3,1,2)$} of $\vec{C}=\{v_{1},v_{2},v_{3}\}$. To simplify the notion we often use $\{v_{i},v_{i+1}, v_{i+2}\}$ to denote $\vec{C}=\{v_{1},v_{2},v_{3}\}$. 
Let $S_{i}^{'}=\{v|v_{i}v,v_{i+2}v,vv_{i+1}\in E\}$, $S_{i}^{''}=\{v|v_{i}v,vv_{i+1},vv_{i+2}\in E\}$ for $i=1,2,3$ and $S_{4}=\{v|v_{1}v,v_{2}v,v_{3}v\in E\}$. It is important to note that all vertices which are coming from $V\setminus V(\vec{C})$ must belong to one of these above sets. Below we obtain some important Lemmas (see Figure \ref{tr} as reference) which will lead us to the main proof.
 
 \vspace{0.7em} 
\begin{lem}\label{la}
Let $x_{i}\in S_{i}^{''}$. Then for any $y_{i+1}\in S_{i+1}^{'}$, $x_{i}y_{i+1}\in E$.  Again for any $x_{i+1}\in S_{i+1}^{''}$, $x_{i}x_{i+1}\in E$.
\end{lem}
  
\begin{proof} 
If not, then $y_{i+1}x_{i}\in E$ imply existence of $H_{1}$, formed by the vertices $\{y_{i+1},x_{i},v_{i+1},v_{i}\}$, which contradicts our assumption. Again if $x_{i+1}x_{i}\in E$ then $H_{1}$ get induced by the vertices $\{x_{i},v_{i+2},v_{i},x_{i+1}\}$ in $G$ which contradicts our assumption.
\end{proof}

\begin{lem}\label{lb}
Let $y_{i}\in S_{i}^{'}$. Then for any $x_{i}\in S_{i}^{''}$, $y_{i}x_{i}\in E$. Again for any $y_{i+1}\in S_{i+1}^{'}$, $y_{i}y_{i+1}\in E$.
\end{lem}

\begin{proof}
On the contrary, if $x_{i}y_{i}\in E$ then $\{y_{i},v_{i+1},v_{i+2},x_{i}\}$ form $H_{1}$, which can not happen as $H_{1}$ is forbidden in $G$. Again if $y_{i+1}y_{i}\in E$, then $H_{1}$ get induced by the vertices $\{v_{i+1},y_{i+1},y_{i},v_{i}\}$ in $G$ which contradicts our assumption.
\end{proof}

\begin{lem}\label{lc}
Let $x_{i}\in S_{i}^{''}$ and $y_{i+2},z_{i+2}\in S^{'}_{i+2}$. Then if $x_{i}y_{i+2},z_{i+2}x_{i}\in E$ then $y_{i+2}z_{i+2}\in E$.
\end{lem}

\begin{proof}
In other case $H_{1}$ is formed by the vertices $\{v_{i},x_{i},y_{i+2},z_{i+2}\}$, which is a contradiction.
\end{proof}


\begin{lem}\label{ld}
Let $x_{i},z_{i}\in S_{i}^{''}$ and $y_{i+2}\in S^{'}_{i+2}$. Then if $x_{i}y_{i+2},y_{i+2}z_{i}\in E$. Then $z_{i}x_{i}\in E$.
\end{lem}

\begin{proof}
In other case $H_{1}$ is formed by the vertices $\{y_{i+2},z_{i},v_{i+1},x_{i}\}$, which is a contradiction.
\end{proof}

\begin{lem}\label{l11}
Let $S_{4}\neq\emptyset$ then $x_{i}x_{4},y_{i}x_{4}\in E$ for all $x_{i}\in S_{i}^{''}, y_{i}\in S_{i}^{'}$.
\end{lem}

\begin{proof}
On the other hand, if $x_{4}x_{i}\in E$ ($x_{4}y_{i}\in E$) for some $x_{i}\in S_{i}^{''}$ ($y_{i}\in S_{i}^{'}$) then $\{x_{4},x_{i},v_{i+1},v_{i}\}$ ($\{x_{4},y_{i},v_{i+1},v_{i}\}$) induce $H_{1}$, which is a contradiction. 
\end{proof}

\begin{lem}\label{l22}
Let $S_{4}\neq \emptyset$ then $x_{i}y_{i+2}\in E$ for any $x_{i}\in S^{''}_{i}, y_{i+2}\in S^{'}_{i+2}$.
\end{lem}

\begin{proof}
Let $x_{4}\in S_{4}$. Then from Lemma \ref{l11} $x_{i}x_{4},y_{i+2}x_{4}\in E$. Now if on the contrary $y_{i+2}x_{i}\in E$ then $H_{6}$ get formed by the vertices of $\vec{C}$ and $x_{i},y_{i+2},x_{4}$, which is a contradiction.
\end{proof}

\begin{lem}\label{l3}
Let $S_{4}\neq\emptyset$ then $S_{1}^{'}$ or $S_{2}^{'}$ or $S_{3}^{'}$ must be empty.
\end{lem}

\begin{proof}
On the contrary, let $S_{1}^{'}, S_{2}^{'}, S_{3}^{'} \not= \emptyset$. Let $x_{1}\in S_{1}^{'}, x_{2}\in S_{2}^{'}, x_{3}\in S_{3}^{'}$ and $x_{4}\in S_{4}$. Then from Lemma \ref{l11}, $x_{1}x_{4},x_{2}x_{4},x_{3}x_{4}\in E$. Again $x_{1}x_{2}, x_{2}x_{3},x_{3}x_{1}\in E$ from Lemma \ref{lb}. Now one can easily verify that $H_{2}$ is induced in $G$ by vertices 
of $\vec{C}$ and $x_{1},x_{2},x_{3},x_{4}$. Hence contradiction arises. Therefore $S_{1}^{'}=\emptyset$ or $S_{2}^{'}=\emptyset$ or $S_{3}^{'}=\emptyset$.
\end{proof}

\begin{lem}\label{l4}
Let $S_{4}\neq\emptyset$ then at least one of the sets $S_{1}^{''},S_{2}^{''},S_{3}^{''}$ is empty. 
\end{lem}

\begin{proof}
On the contrary, let $S_{1}^{''},S_{2}^{''},S_{3}^{''} \not= \emptyset$. Let $x_{1}\in S_{1}^{''}, x_{2}\in S_{2}^{''}, x_{3}\in S_{3}^{''}$ and $x_{4}\in S_{4}$. Then from Lemma \ref{l11}, $x_{1}x_{4},x_{2}x_{4},x_{3}x_{4}\in E$. Again from Lemma \ref{la} we get $x_{1}x_{2},x_{2}x_{3},x_{3}x_{1}\in E$. Hence it is easy to verify now that $H_{3}$ is induced  by vertices of $\vec{C}$ and $x_{1},x_{2},x_{3},x_{4}$, which contradicts our assumption. Therefore $S_{1}^{''}=\emptyset$ or $S_{2}^{''}=\emptyset$ or $S_{3}^{''}=\emptyset$.
\end{proof}

\begin{lem}\label{l5}
Let $S_{4}\neq \emptyset$. Then there exist some $i\in\{1,2,3\}$ satisfying $S_{i}^{''}=S_{i+1}^{'}=\emptyset$. 
\end{lem}

\begin{proof}
We assume on the contrary that no such $i$ exists. Let $S_{1}^{''}\neq \emptyset$. Now from Lemma \ref{l4} either $S_{2}^{''}=\emptyset$ or $S_{3}^{''}=\emptyset$. When $S_{2}^{''}=\emptyset$ then $S_{3}^{'}\neq\emptyset$ from assumption. Now let, $S_{3}^{''}\neq \emptyset$ and $x_{1}\in S_{3}^{'}, x_{2}\in S^{''}_{3}, x_{3}\in S_{1}^{''}, x_{4}\in S_{4}$. Then $x_{2}x_{3}, x_{1}x_{2}, x_{3}x_{1}\in E$ from Lemma \ref{la}, Lemma \ref{lb} and Lemma \ref{l22} respectively. Again $x_{1}x_{4},x_{2}x_{4},x_{3}x_{4}\in E$ from Lemma \ref{l11}. Hence $H_{4}$ get induced by vertices of $\vec{C}$ and $x_{1},x_{2},x_{3},x_{4}$ in $G$, which is a contradiction. 

\vspace{0.2em}

\noindent Therefore $S_{3}^{''}=\emptyset$ and hence $S_{1}^{'}\neq \emptyset$ from our assumption. Now let, $x_{1}\in S_{3}^{'}, x_{2}\in S_{1}^{'}, x_{3}\in S_{1}^{''}$. Then $x_{1}x_{2}, x_{2}x_{3}\in E$ from Lemma \ref{lb} and  $x_{3}x_{1}\in E$ from Lemma \ref{l22}. Again $x_{1}x_{4},x_{2}x_{4},x_{3}x_{4}\in E$ from Lemma \ref{l11}. Hence $H_{5}$ is formed by vertices of $\vec{C}$ and $x_{1},x_{2},x_{3},x_{4}$ in $G$. Therefore contradiction arises.
Similar contradiction will arise if we start with $S_{3}^{''}=\emptyset$.
\end{proof}

\begin{lem}\label{l6}
Let $S_{4}\neq \emptyset$. Then the induced subdigraph $G_{S_{4}}=(S_{4},E(S_{4}))$ of $G$ is a proper oriented ICD.
Moreover $S_{4}$ forms an out-fountain in $G_{S_{4}}$.
\end{lem}

\begin{proof}
First we show that $G_{S_{4}}$ can not contain any directed $3$-cycle. On the contrary, let the vertices $x_{1},x_{2},x_{3}$ form a directed $3$-cycle in  $G_{S_{4}}$ with arcs $x_{1}x_{2},x_{2}x_{3},x_{3}x_{1}$ belonging to  $E(S_{4})$. Then $\{x_{1},x_{2},x_{3},v_{i}\}$ form $H_{1}$ in $G$, which is a contradiction.
Again as $G_{S_{4}}$ is a tournament, from Theorem \ref{d3cycle} it becomes an oriented ICD. Furthermore from Corollary \ref{tforbid} we can conclude that $G_{S_{4}}$ is a proper oriented ICD. Therefore $S_{4}$ forms an out-fountain in $G_{S_{4}}$ from Lemma \ref{Out}. 
\end{proof}

\begin{lem}\label{l7}
Let $S^{''}_{i}\neq \emptyset$. Then the induced subdigraph $G_{S^{''}_{i}}=(S^{''}_{i}, E(S^{''}_{i}))$ of $G$ is a proper oriented ICD. Moreover $S^{''}_{i}$ forms an out-fountain in $G_{S^{''}_{i}}$.
\end{lem}

\begin{proof}
Similar as Lemma \ref{l6}.
\end{proof}


\begin{lem}\label{l8}
Let $S^{'}_{i}\neq \emptyset$. Then the induced subdigraph $G_{S^{'}_{i}}=(S^{'}_{i}, E(S^{'}_{i}))$ is a proper oriented ICD. Moreover $S^{'}_{i}$ forms an out-fountain in $G_{S^{'}_{i}}$.
\end{lem}

\begin{proof}
Similar as Lemma \ref{l6}.
\end{proof}

\begin{lem}\label{l9}
For each $i\in \{1,2,3\}$, $S^{'}_{i} \cup S^{''}_{i} \cup \{v_{i+1}\} \cup S^{'}_{i+1}$ and $S^{''}_{i} \cup \{v_{i+1}\} \cup S^{'}_{i+1} \cup S^{''}_{i+1}\cup \{v_{i+2}\}$ form out-fountains in $G$.  
\end{lem}

\begin{proof}
As $G_{S^{''}_{i}}$ and $G_{S^{'}_{i}}$ are induced subdigraphs of $G$, it is clear that $S^{''}_{i}$ and $S^{'}_{i}$ form out-fountains in $G$ from Lemma \ref{l7} and Lemma \ref{l8} respectively. Let $\sigma_{S^{'}_{i}}, \sigma_{S^{''}_{i}},\sigma_{S^{'}_{i+1}}$ denote the orderings with respect to which $S^{'}_{i}, S^{''}_{i}, S^{'}_{i+1}$ form out-fountains. We concatenate these orderings together by taking the order $\sigma_{S^{'}_{i}}\prec \sigma_{S^{''}_{i}}\prec v_{i+1}\prec \sigma_{S^{'}_{i+1}}$.
From Lemma \ref{lb} and definition of $S^{'}_{i}$, we get to know that vertices 
of $S^{''}_{i} \cup \{v_{i+1}\} \cup S^{'}_{i+1}$ are out-neighbors of every vertex of $S^{'}_{i}$. Again Lemma \ref{la} and definition of $S^{''}_{i}$ tells us that vertices of $\{v_{i+1}\} \cup S^{'}_{i+1}$ are out-neighbors of every vertex of $S^{''}_{i}$. Clearly $S^{'}_{i+1}$ is an out-neighbor set of the vertex $v_{i+1}$. Therefore $S^{'}_{i} \cup S^{''}_{i} \cup \{v_{i+1}\} \cup S^{'}_{i+1}$ forms an out-fountain.
  
 \vspace{0.2em}
  
\noindent   Similar proof holds for $S^{''}_{i} \cup \{v_{i+1}\} \cup S^{'}_{i+1} \cup S^{''}_{i+1}\cup \{v_{i+2}\}$ by taking the ordering $\sigma_{S^{''}_{i}}\prec v_{i+1}\prec \sigma_{S^{'}_{i+1}}\prec  \sigma_{S^{''}_{i+1}} \prec v_{i+2}$.
\end{proof}

\vspace{0.39em}

\noindent Now we will arrange all the vertices of $V$ in a single circular sequence $V_{S_{P}}$ in such a way so that with respect to this ordering
if we arrange the rows and columns of the augmented adjacency matrix $A(G)$ of $G$, then it satisfies circular $1$'s property along rows of it and therefore from Theorem \ref{cop1}, $G$ becomes a circular-arc catch digraph. 
Below we describe the method of constructing $V_{S_{P}}$.

\vspace{0.7em}
\noindent\textbf{Construction of the vertex circular sequence $V_{S_{P}}$:}

\vspace{0.6	em}

\noindent First we arrange the vertices $v_{1},v_{2},v_{3}$ of $\vec{C}$ clockwise rounding circle, i.e., $v_{1}<_{\sigma} v_{2}<_{\sigma} v_{3}$ in $V_{S_{P}}$ where $<_{\sigma}$ is the clockwise order of vertices around the circle (see Figure \ref{cr}). Next we place the vertices of 
$V\setminus V(\vec{C})$ in $V_{S_{P}}$ in following way according to the cases $S_{4}\neq \emptyset$ or $S_{4}=\emptyset$.

\vspace{0.77em}

\noindent $\mathbf{(i)\hspace{0.3em} S_{4}\neq \emptyset}$

\vspace{0.7em}

\noindent Since $G$ is oriented it is evident from the definition of $S_{4}$ that each vertex of $S_{4}$ has no out-neighbor in $V(\vec{C})$. Moreover from Lemma \ref{l5} we get to know about existence of some $i\in\{1,2,3\}$ satisfying $S^{''}_{i}=S^{'}_{i+1}=\emptyset$. Let $S^{''}_{1}=S^{'}_{2}=\emptyset$. From Lemma \ref{l6}, we get to know that $V(S_{4})$ forms an out-fountain with respect to some ordering (say $\sigma_{S_{4}}$). We place all the vertices of $S_{4}$ clockwise within $(v_{1},v_{2})$ arranged in $\sigma_{S_{4}}$ in $V_{S_{P}}$.

\vspace{1em}

\noindent Next we place the vertices which have exactly one in-neighbor from $V(\vec{C})$. From Lemma \ref{l4} we already know that either $S^{''}_{1}$ or $S^{''}_{2}$ or $S^{''}_{3}=\emptyset$. Note that we have taken $S^{''}_{1}=\emptyset$ in above paragraph. Again from Lemma \ref{l7} we get to know that $V(S^{''}_{2})$, $ V(S^{''}_{3})$ form out-fountains (when they are non empty) with respect to some orderings $\sigma_{S^{''}_{2}}$, $\sigma_{S^{''}_{3}}$  respectively. We place the vertices of $S^{''}_{2}$ clockwise  within $(v_{2},v_{3})$ arranged in $\sigma_{S^{''}_{2}}$ and vertices of $S^{''}_{3}$ clockwise within $(v_{3},v_{1})$ arranged in $\sigma_{S^{''}_{3}}$ in $V_{S_{P}}$.

\vspace{1em}

\noindent Now we place the vertices which have exactly two in-neighbors from $V(\vec{C})$. From Lemma \ref{l3} we get that either $S^{'}_{1}$ or $S^{'}_{2}$ or $S^{'}_{3}=\emptyset$. In first paragraph we have considered $S^{'}_{2}=\emptyset$. Again from Lemma \ref{l8} we get that $V(S^{'}_{3})$, $V(S^{'}_{1})$ form out-fountains (when they are nonempty) with respect to some orderings $\sigma_{S^{'}_{3}}$, $\sigma_{S^{'}_{1}}$ respectively. We place the vertices of $S^{'}_{3}$ clockwise within $(v_{3},v_{1})$  arranged in $\sigma_{S^{'}_{3}}$ before the first vertex of $S^{''}_{3}$ and vertices of $S^{'}_{1}$ clockwise within $(v_{1},v_{2})$ arranged in $\sigma_{S^{'}_{1}}$  before the first vertex of $S_{4}$ in $V_{S_{P}}$.
 
\vspace{1em}

\noindent Hence $V_{S_{P}}$ takes the form $\mathbf{v_{1}<_{\sigma} \sigma_{S^{'}_{1}}<_{\sigma} \sigma_{S_{4}} <_{\sigma}  v_{2} <_{\sigma} \sigma_{S^{''}_{2}}<_{\sigma} v_{3} <_{\sigma} \sigma_{S^{'}_{3}} <_{\sigma} \sigma_{S^{''}_{3}}<_{\sigma} v_{1}}$ when $S^{''}_{1}=S^{'}_{2}=\emptyset$.

\vspace{1em}
\noindent Placing all the vertices in $V_{S_{P}}$ is analogous for the cases when $S^{''}_{2}=S^{'}_{3}=\emptyset$ or $S^{''}_{3}=S^{'}_{1}=\emptyset$. If more than one pair of sets among $\{ S^{''}_{1}, S^{'}_{2}\}, \{S^{''}_{2}, S^{'}_{3}\}$ or $\{ S^{''}_{3}, S^{'}_{1}\}$ become empty simultaneously (say $S^{''}_{1}=S^{'}_{2}=S^{''}_{2}=S^{'}_{3}=\emptyset$) then we place all the vertices of $S_{4}$ either within $(v_{1},v_{2})$ or within $(v_{2},v_{3})$ (not both simultaneously). 

\vspace{1em}
\noindent $\mathbf{(ii)\hspace{0.3em} S_{4}=\emptyset}$

\vspace{1em}

\noindent In this case first we place the vertices of $S^{''}_{1}, S^{''}_{2}, S^{''}_{3}$ within $(v_{1},v_{2})$, $(v_{2},v_{3})$, $(v_{3},v_{1})$ respectively according to the cases $S^{''}_{1}\neq \emptyset$, $S^{''}_{2}\neq \emptyset$, $S^{''}_{3}\neq \emptyset$. While placing the vertices of a certain set (say $S^{''}_{1}$) we arrange the vertices according to the order (say $\sigma_{S^{''}_{1}}$) with respect to which $S^{''}_{1}$ forms an out-fountain and place them clockwise within $(v_{1},v_{2})$ in $V_{S_{P}}$. Similarly vertices of $S^{''}_{2}$, $S^{''}_{3}$ can be placed within $(v_{2},v_{3})$, $(v_{3},v_{1})$ arranged in $\sigma_{S^{''}_{2}}$, $\sigma_{S^{''}_{3}}$ respectively.

\vspace{1em}

\noindent Next we place vertices of $S^{'}_{1},S^{'}_{2}, S^{'}_{3}$ within $(v_{1},v_{2})$, $(v_{2},v_{3})$, $(v_{3},v_{1})$ respectively according to the cases $S^{'}_{1}\neq \emptyset$, $S^{'}_{2}\neq \emptyset$, $S^{'}_{3}\neq \emptyset$. While placing the vertices of a certain set (say $S^{'}_{1}$) we arrange them according to the order (say $\sigma_{S^{'}_{1}}$) with respect to which $S^{'}_{1}$ forms an out-fountain and place them clockwise within $(v_{1},v_{2})$ before the first point of $S^{''}_{1}$ in $V_{S_{P}}$. Analogously vertices of $S^{'}_{2}$, $S^{'}_{3}$ are to be placed within $(v_{2},v_{3})$, $(v_{3},v_{1})$ arranged in $\sigma_{S^{'}_{2}}$, $\sigma_{S^{'}_{3}}$ before the first vertex of $S^{''}_{2}, S^{''}_{3}$ respectively in $V_{S_{P}}$.

\vspace{1em}

\noindent Hence $V_{S_{P}}$ takes the form $\mathbf{v_{1}<_{\sigma} \sigma_{S^{'}_{1}}<_{\sigma} \sigma_{S^{''}_{1}} <_{\sigma}  v_{2} <_{\sigma}\sigma_{S^{'}_{2}} <_{\sigma} \sigma_{S^{''}_{2}}<_{\sigma} v_{3} <_{\sigma} \sigma_{S^{'}_{3}} <_{\sigma} \sigma_{S^{''}_{3}}<_{\sigma} v_{1}}$.

\vspace{0.9em}

\begin{lem}\label{l10}
Let $S_{4}=\emptyset$. Then out-neighbors of a vertex $w_{i}\in S^{'}_{i+2}$ in $S^{''}_{i}$ occur consecutively in $V_{S_{P}}$. Again out-neighbors of $w_{i}\in S^{''}_{i}$ in $S^{'}_{i+2}$ occur consecutively in $V_{S_{P}}$.
\end{lem}

\begin{proof}
We assume on the contrary that $x_{i}<_{\sigma} y_{i}<_{\sigma} z_{i}$ such that $w_{i}x_{i},y_{i}w_{i},w_{i}z_{i}\in E$ where $w_{i}\in S^{'}_{i+2}$ and $x_{i},y_{i},z_{i}\in S^{''}_{i}$. Then as $w_{i}z_{i},y_{i}w_{i}\in E$ from Lemma \ref{ld} it follows that $z_{i}y_{i}\in E$. Again as $S^{''}_{i}$ forms an out-fountain, from construction of $V_{S_{P}}$ we get $z_{i}<_{\sigma} y_{i}$. But this contradicts our assumption. Therefore the result follows. Other case also holds true using Lemma \ref{lc}.
\end{proof}

\vspace{0.2em}
\noindent  \textit{main proof:} Now to prove $A(G)$ satisfies circular $1$'s property along rows of it when the vertices are ordered according to $V_{S_{P}}$, it is sufficient to show that out-neighbors of every vertex of $V_{S_{P}}$ occur clockwise right or anticlockwise left (or both) side of it consecutively in $V_{S_{P}}$.

\vspace{0.6 em}
\noindent a) We first consider the case when $S_{4}\neq \emptyset$. Then from Lemma \ref{l5} there exist some $i\in\{1,2,3\}$ satisfying $S^{''}_{i}=S^{'}_{i+1}=\emptyset$. Let $S^{''}_{1}=S^{'}_{2}=\emptyset$. Without loss of generality we can assume vertices of $S_{4}$ are placed within $(v_{1},v_{2})$  as per the construction of $V_{S_{P}}$. Hence $V_{S_{P}}$ is of the form 

\vspace{0.9em}

$v_{1}<_{\sigma} \sigma_{S^{'}_{1}}<_{\sigma} \sigma_{S_{4}} <_{\sigma}  v_{2} <_{\sigma} \sigma_{S^{''}_{2}}<_{\sigma} v_{3} <_{\sigma} \sigma_{S^{'}_{3}} <_{\sigma} \sigma_{S^{''}_{3}}<_{\sigma} v_{1}$. 

\vspace{0.9em}

\noindent From Lemma \ref{l9}, we get to know that $S^{''}_{2} \cup \{v_{3}\} \cup S^{'}_{3} \cup S^{''}_{3}\cup \{v_{1}\}$ and $S^{'}_{3} \cup S^{''}_{3} \cup \{v_{1}\}\cup S^{'}_{1}$ form out-fountains in $G$ with respect to their occurrence in $V_{S_{P}}$. 
Again following Lemma \ref{l22} every vertex of $S^{'}_{1}$ becomes out-neighbor of vertices of the set $S^{''}_{2}$. Moreover Lemma \ref{l11} ensures that each vertex of $S_{4}$ is an out-neighbor of the vertex sets $S^{''}_{i}, S^{'}_{i}, i\in\{1,2,3\}$. Therefore from the adjacencies of the vertices of $V_{S_{P}}$ it is clear now that out-neighbors of every vertex occur consecutively right (clockwise) or consecutively left (anticlockwise) or both side of it in $V_{S_{P}}$. 

\vspace{0.4em}
\noindent Other cases (i.e., when $S^{''}_{2}=S^{'}_{3}=\emptyset$ or $S^{''}_{3}=S^{'}_{1}=\emptyset$) can be verified in analogous way.

\vspace{0.7 em}
\noindent b) Next we consider the case when $S_{4}=\emptyset$. In this case $V_{S_{P}}$ is of the form 
\vspace{0.51em}

$v_{1}<_{\sigma} \sigma_{S^{'}_{1}}<_{\sigma} \sigma_{S^{''}_{1}} <_{\sigma}  v_{2} <_{\sigma}\sigma_{S^{'}_{2}} <_{\sigma} \sigma_{S^{''}_{2}}<_{\sigma} v_{3} <_{\sigma} \sigma_{S^{'}_{3}} <_{\sigma} \sigma_{S^{''}_{3}}<_{\sigma} v_{1}$.

\vspace{0.7em}
\noindent From Lemma \ref{l9}, one can see that for each $i\in\{1,2,3\}$, $S^{'}_{i} \cup S^{''}_{i} \cup \{v_{i+1}\}\cup S^{'}_{i+1}$ and  $S^{''}_{i} \cup \{v_{i+1}\} \cup S^{'}_{i+1} \cup S^{''}_{i+1}\cup \{v_{i+2}\}$ form out-fountains with respect to their occurrences in $V_{S_{P}}$.
Again from Lemma \ref{l10} we get that out-neighbors of vertices corresponding to the sets $S^{'}_{i+2}, S^{''}_{i}$ occur consecutively in $V_{S_{P}}$. Rest of the verification follows from the adjacency of every vertex of $V_{S_{P}}$.

\vspace{0.9 em}
\noindent \textbf{Case (ii)}

\noindent Let $G$ does not contain any directed $3$-cycle as induced subdigraph. Then from Theorem \ref{d3cycle} we get to know that $G$ is an interval catch digraph. Therefore $G$ becomes a circular-arc catch digraph with the same representation.
\end{proof}

\section{Proper circular-arc catch digraph}\label{mc}

\begin{defn}\label{pcac}
A {\em proper circular-arc catch digraph}\index{proper circular-arc catch digraph} is a circular-arc catch digraph where no circular-arc is contained in another properly. 
\end{defn}

\noindent In this section, we obtain the characterization of proper circular-arc catch digraphs in terms of the augmented adjacency matrix structure of it. For this we need to define {\em monotone circular ordering} of a binary matrix, which will be used as a main tool for this characterization.

\begin{figure}[b]
\[
\begin{array}{c|c|c|c|c|c|c|c|ccc}

\multicolumn{1}{c}{} & \multicolumn{1}{c}{v_1} & \multicolumn{1}{c}{v_{2}} & \multicolumn{1}{c}{v_3} & \multicolumn{1}{c}{v_4} &\multicolumn{1}{c}{v_{5}}& \multicolumn{1}{c}{v_6} & \multicolumn{1}{c}{v_7}  &\multicolumn{1}{c}{} & \lambda_{i} &\mu_{i} \\
\cline{2-8}
v_2 & \multicolumn{1}{c}{1} &\multicolumn{1}{c}{1} & \multicolumn{1}{c}{1} & 1 & \multicolumn{1}{c}{0}  & \multicolumn{1}{c}{0}& 0&  &1 & 4 \\ 

\cline{2-3}
\cline{6-6}
v_4 & \multicolumn{1}{c}{0} &0 & \multicolumn{1}{c}{1} & \multicolumn{1}{c}{1} & 1 &\multicolumn{1}{c}{0} & 0 & &3& 5 \\ 

\cline{7-7}

v_3 & \multicolumn{1}{c}{0} &0 & \multicolumn{1}{c}{1} & \multicolumn{1}{c}{1} & \multicolumn{1}{c}{1}  & 1 &0  &  &3& 6 \\ 
\cline{2-3}
\cline{8-8}

v_1 & \multicolumn{1}{c}{1} & 1& \multicolumn{1}{c}{1} & \multicolumn{1}{c}{1} & \multicolumn{1}{c}{1}  & \multicolumn{1}{c}{1}& 1 & & 3& 9 \\ 

v_5 & \multicolumn{1}{c}{1} & 1& \multicolumn{1}{c}{1} & \multicolumn{1}{c}{1} & \multicolumn{1}{c}{1}  & \multicolumn{1}{c}{1}& 1 & &3 & 9 \\

\cline{4-6}

v_7 & \multicolumn{1}{c}{1} &1 & \multicolumn{1}{c}{0} & \multicolumn{1}{c}{0} &0  &\multicolumn{1}{c}{1} &1 & &6& 9 \\
\cline{4-4} 

v_6 & \multicolumn{1}{c}{1} &\multicolumn{1}{c}{1} & 1 & \multicolumn{1}{c}{0} & 0 & \multicolumn{1}{c}{1}& 1&  &6 & 10 \\ 

\cline{2-8}
\end{array}
\]
\caption{Matrix $B$ satisfying monotone circular ordering}\label{fg1}
\end{figure}

\vspace{0.2em}
\begin{defn}\label{MCO}
\noindent Let $B$ be a $m\times n$ binary matrix containing rows and columns of all types except zero rows. Let $B$ satisfy circular $1$'s property along  both rows and columns of it. We call a nontrivial row ($i^{\text{th}}$ row say) as row of \textit{$1$st type} if it contains exactly one stretch of $1$'s (i.e., maximum two stretch of $0$'s) and a row as row of \textit{$2$nd type} if it contains exactly two stretch of $1$'s (i.e., one stretch of $0$'s). Let $[i_{1},i_{2}]$ and $[i_{3},i_{4}]$ be the circular stretch of $1$'s in the $i^{\text{th}}$ row (nontrivial) of $B$ when the row is of $1$st type and $2$nd type respectively, where $i_{1},i_{2}$ $(i_{3},i_{4})$ denote the column numbers where circular stretch of $1$'s in the $i^{\text{th}}$ row of $B$ starts and ends for rows of $1$st ($2$nd) type (cf. in matrix $B$ of Figure \ref{fg1}, rows corresponding to $v_{2},v_{4},v_{3}$ are of $1$st type and $v_{7},v_{6}$ are of $2$nd type). Note that $i_{1}<i_{2}$ if the $i^
\text{th}$ row is of  $1$st type row and $i_{4}<i_{3}$ when it is of $2$nd type
(cf. in Figure \ref{fg1}, when $i=1$ (row $v_{2}$), $i_{1}=1, i_{2}=4$ and when $i=6$ (row $v_{7}$), $i_{3}=6,i_{4}=2$).

\vspace{0.3em}
\noindent We define for nontrivial rows,

\begin{equation}
\begin{split}
\begin{cases}
\lambda_{i}=i_{1}\\
\mu_{i}=
i_{2} & \mbox{if i is a 1st type row}\\
\end{cases}
\end{split}
\hspace{1em}
\begin{split}
\begin{cases}
\lambda_{i}=i_{3}\\
\mu_{i}=
i_{4}+n & \mbox{if i is a 2nd type row}\\
\end{cases}
\end{split}
\end{equation}

\vspace{0.7em}
\noindent A full row can be treated as row of 1st type or $2$nd type according to its occurrence in $B$. If the first row of $B$ is a full row then we take $\lambda_{1}=1$ and $\mu_{1}=n$ and consider it as $1$st type row. In other case, let $i-1$ be a nontrivial row and $i$ be a full row (if there is more than one full rows occur consecutively in $B$ then we consider the row which occurs first among them), then if $(i-1)_{1}=1$, we consider the $i^{\text{th}}$ row as row of $1$st type, otherwise (i.e., when $(i-1)_{1}>1$ or the $(i-1)^{\text{th}}$ row is of $2$nd type) we treat it as row of $2$nd type (cf. in matrix $B$ of Figure \ref{fg1}, $i=4$ (row $v_{1}$) is of $2$nd type). For such a full row we define, \\
\vspace{0.4em}
\begin{equation}
\begin{split}
\begin{cases}
\lambda_{i}=\lambda_{i-1} \hspace{5.3em} \mbox{where} \hspace{0.3em}  i-1\hspace{0.25em} \mbox{is a nontrivial row}\\
\mu_{i}=\begin{cases}n &\mbox{if} \hspace{0.3em}(i-1)_{1}=1 \\
\lambda_{i-1}-1+n & \mbox{otherwise}\\
\end{cases}
\end{cases}
\end{split}
\end{equation}

\vspace{0.7em}
\noindent Note that if there are more than one full rows occur consecutively in $B$ then for any other full rows (say $j^{\text{th}}$ row) which occur after $i^{\text{th}}$ one in a consecutive manner, we keep $\lambda_{j}=\lambda_{i},\mu_{j}=\mu_{i}$. Furthermore, we consider them as of the same type as the $i^{\text{th}}$ row (cf. in matrix $B$ of Figure \ref{fg1}, $i=5$ (row $v_{5}$) is of $2$nd type).   

\vspace{0.7em}
\noindent We call $B$ has a {\em monotone circular ordering}
if there exists a row ordering $v_{1},\hdots,v_{m}$ of $V$ such that $\{\lambda_{1},\lambda_{2},\hdots,\lambda_{m}\}$ and $\{\mu_{1},\mu_{2},\hdots,\mu_{m}\}$ form non-decreasing sequences when $B$ satisfies circular $1$'s property along both rows and columns of it.

\vspace{0.7em}
\noindent  We can select the end points of the circular stretch of $1$'s corresponding to the $i^\text{th}$ full row as $\lambda_{i}=1, \mu_{i}=n$ when it is of $1$st type and $\lambda_{i}, \lambda_{i}-1$ when it is of $2$nd type (cf. in matrix $B$ of Figure \ref{fg1}, when $i=4$ (row $v_{1}$), we consider $[3,2]$ as the corresponding circular stretch). This understanding is important for the construction of stair partition in $B$, which separates $0$'s from the $1$'s (see Figure \ref{fg1}). 
\end{defn}

\noindent \textbf{Notations:} We denote the set of all rows of a binary matrix $B$ by row$(B)$. Let $r,s$ be any two rows of a binary matrix $B$ satisfying monotone circular ordering. By $r^\ast$ we denote the set of column numbers having $1$ as their entries along the $r^{\text{th}}$ row. Then $r^{\ast}-s^{\ast}$ denote the set difference  of circular stretches of $1$'s of rows $r,s$ in $B$. Similarly, $r^{\ast}-s^{\ast}-t^{\ast}$ can be defined as set of those columns of $r^{\ast}-s^{\ast}$ which contain $1$ as their entries and do not contain $1$ along the $t^{\text{th}}$ row in $B$ (cf. in matrix $B$ of Figure \ref{fg1}, when $r=4,s=1,t=6$,
we have $r^\ast=\{1,2,3,4,5,6,7\}$, $s^\ast=\{1,2,3,4\}$, $t^\ast=\{1,2,6,7\}$. So $r^\ast-s^\ast=\{5,6,7\}$ and $r^\ast-s^\ast-t^\ast=\{5\}$).

\vspace{0.4em}

\noindent When $s,r$ both are nontrivial, $s^{\ast}\subset r^{\ast}$ denote that the circular stretch of $1$'s corresponding to row $s$ to be properly contained in the circular stretch of $1$'s corresponding to row $r$, i.e., both end points of the circular stretch of $1$'s corresponding to row $s$ are strictly within the circular stretch of $1$'s of row $r$, which means if both $s,r$ are of $1$st type then $r_{1}<s_{1}\leq s_{2}<r_{2}$ and if $s$ is of $1$st type, $r$ is of $2$nd type then either $s_{2}<r_{4}$ or $r_{3}<s_{1}$ and if both $s,r$ are of $2$nd type then $s_{4}<r_{4}<r_{3}<s_{3}$. We will follow similar notation $I_{s}\subset I_{r}$ to denote proper containment of the arc $I_{s}$ within $I_{r}$.

\vspace{0.4em}
\noindent We consider any nontrivial row $s$ to be always contained within a full row $r$.  For better understanding we refer circular stretch of $1$'s as {\em circular interval} throughout this section. One can note from above that for any two nontrivial rows $s,r $, $s^{*}\subset  r^{*}$ if and only if $r^{\ast}-s^{\ast}$ is not a circular interval. Moreover, when $r$ is a full row and $s$ is a nontrivial row, $r^{\ast}-s^{\ast}$ is always a circular interval. If $r^{\ast}-s^{\ast}=\emptyset$ or $s^{\ast}-r^{\ast}=\emptyset$ for any two rows $r,s$, then we trivially treat them as circular interval. Again for some full row $r$ and two nontrivial rows $s,t$, if $r^{\ast}-s^{\ast}-t^{\ast}=\emptyset$ or $r^{\ast}-t^{\ast}-s^{\ast}=\emptyset$, then also we consider them as circular interval.

\vspace{0.6em}
\begin{exmp}\label{f1f2f3}

Consider the matrices $F_{1},F_{2}, F_{3}$ in Figure \ref{forbid}. The matrix $F_1$ is also mentioned in \cite{BDGS} in the context of proper circular-arc bigraphs. One can note that each of these matrices satisfies circular $1$'s property along rows.  If we delete the full row from each $F_{i}$ then the remaining submatrices become $(i+2)\times (i+2)$ binary matrices where each row contains exactly $i$ number of $1$'s and each column contains exactly two $0$'s along columns for $i=1,2,3$. Moreover, one can verify that $F_{1}$ is not a submatrix of $F_{2}, F_{3}$ and $F_{2}$ is not a submatrix of $F_{3}$. We will prove that $F_{1}, F_{2}, F_{3}$ do not satisfy circular $1$'s property along rows and columns. Below we do the verification for the matrix $F_{2}$. Others can be shown similarly.

\begin{figure}[ht]
\[
F_{1}=
\left[\begin{array}{ccc}
  1 & 0 & 0  \\
  0 & 1 & 0  \\
  0 & 0 & 1  \\
  1 & 1 & 1  \\
\end{array}\right]
\\
F_{2}=\left[\begin{array}{cccc}
1 & 1 &0 & 0\\
0 & 1 & 1 & 0\\
0 & 0 & 1 & 1\\
1 & 0 & 0 & 1\\
1 & 1 & 1 & 1\\
\end{array}\right]
\\
F_{3}=\left[\begin{array}{ccccc}
1 & 1 &1 & 0 & 0\\
0 & 1 & 1 & 1 & 0\\
0 & 0 & 1 & 1 & 1\\
1 & 0 & 0 & 1 &1 \\
1 & 1 & 0 & 0 & 1\\
1 & 1 & 1 & 1 & 1\\
\end{array}\right]
\]
\caption{Matrices which does not satisfy circular $1$'s property along columns}\label{forbid}
\end{figure}
\vspace{0.3em}

\noindent Firstly, one can verify that the circular $1$'s property along rows of $F_2$ is retained only for cyclic permutation of columns and their reverse permutations. Any such column permuted matrix of $F_2$ can be transformed again to the original $F_2$ matrix by row permutations only. Thus for proving that $F_2$ does not satisfy circular $1$'s property along rows and columns under any independent permutation of rows and columns, it is sufficient to show that $F_2$ itself has no row permutation that satisfies circular $1$'s property along columns. Next one can note that under any row permutation of $F_{2}$, the number of zeros in each row and each column remains intact. Furthermore, since any two columns in $F_{2}$ are distinct, the columns remain distinct regardless of row permutation. Now on contrary, suppose $F_{2}$ or any of its row permutation, say $A=(a_{i,j})$  satisfies circular $1$'s property along columns. We note any cyclic permutation of the rows keeps the circular $1$'s property along columns intact. So without loss of generality we may assume that the first row of $A$ is the full row. Let $a_{2,i}=a_{2,j}=0$ for some $j\neq i$ as each row contains exactly two zeros. Now since $A$ satisfies circular $1$'s property along columns, stretches of $1$'s in each column must be circular. Then since the first row is a full row, zeros must be consecutive in each column. Thus we have $a_{3,i}=a_{3,j}=0$. But this implies that the $i^{\text{th}}$ column and $j^{\text{th}}$ column are identical as every column contains exactly two zeros. Hence contradiction arises since each column differs from the others in $F_{2}$.
\end{exmp}

\begin{prop}\label{prop1}
Matrices $F_{1}, F_{2}, F_{3}$ does not satisfy circular $1$'s property along rows and columns.
\end{prop}

\begin{proof}
The proof is evident from the discussion mentioned in Example \ref{f1f2f3}. 
\end{proof}

\vspace{0.4em}
\noindent In the following theorem we present characterization of a proper circular-arc catch digraph.

\begin{thm}\label{pcacd}
Let $G=(V,E)$ be a simple digraph. Then $G$ is a proper circular-arc catch digraph if and only if 
\begin{enumerate}
\item The augmented adjacency matrix $A(G)$ of $G$ satisfies circular $1$'s property along rows and columns of it. Let $B$ be the transformed matrix obtained 
by permuting rows and columns of $A(G)$, where $1$'s in each row and column of $B$ are circular.

\item  If $r,s$ be any two rows of $B$. Then $r^{\ast}-s^{\ast}$ and $s^{\ast}-r^{\ast}$ are circular intervals.

\item For any two nontrivial rows $s,t$ of $B$, $r^{\ast}-s^{\ast}-t^{\ast}$ and $r^{\ast}-t^{\ast}-s^{\ast}$ are circular intervals where $r$ is a full row.
\end{enumerate}
\end{thm}

\begin{proof}
\textit{\textbf{Necessity:}} Let $G=(V,E)$ be a proper circular-arc catch digraph with representation 
$\{(I_{i},p_{i})|i=1,\hdots,n\}$ where $I_{i}=[a_{i},b_{i}]$ be the circular arc and $p_{i}$ be the point within $I_{i}$ attached to each vertex $v_{i}\in V$ in $G$. We construct the matrix $B$ by arranging rows and columns of the augmented adjacency matrix $A(G)$ according to increasing order of $a_{i}$'s (say $\prec_{r}$) and according to increasing order of $p_{i}$'s (say $\prec_{c}$) respectively. Now  it is clear that $v_{i}v_{j}=1$ in $B$ if and only if $p_{j}\in I_{i}$.

\vspace{0.4em}
 \noindent\textbf{\textit{$1.$ Circular $1$'s property along rows and columns is satisfied by $B$.}}

\vspace{0.4em}

\noindent Now let $v_{1}\prec_{c}v_{3}$ and $i$ be an arbitrary row of $B$. Now if $v_{i}v_{1}, v_{i}v_{3}\in E$ then either the circular interval $[p_{1},p_{3}]\subseteq I_{i}$ or $[p_{3},p_{1}]\subseteq I_{i}$. Now if $v_{i}v_{1}, v_{i}v_{3}\in E$ then either the circular interval $[p_{1},p_{3}]\subseteq I_{i}$ or $[p_{3},p_{1}]\subseteq I_{i}$. If $[p_{1},p_{3}]\subseteq I_{i}$ then $p_{2}\in I_{i}$ for all $v_{1}\prec_{c}v_{2}\prec_{c} v_{3}$, which imply $v_{i}v_{2}\in E$. Again when $[p_{3},p_{1}]\subseteq I_{i}$ then $p_{w},p_{k}\in I_{i}$, i.e., $v_{i}v_{w}, v_{i}v_{k}\in E$ for all $v_{w}\succ_{c}v_{3}$ and $v_{k}\prec_{c} v_{1}$. Hence $B$ satisfies circular $1$'s property for rows.

\vspace{0.3em}
\noindent Again let $v_{1}\prec_{r} v_{3}$ and $j$ be an arbitrary column of $B$. Now if $v_{1}v_{j},v_{3}v_{j}\in E$, then  $p_{j}\in I_{1}\cap I_{3}$.
As $a_{1}<a_{3}$ and the arcs are not contained in each other,  $I_{1}\cap I_{3}=[a_{1},b_{3}]$ or $[a_{3},b_{1}] $ or their union. Now if $I_{1}\cap I_{3}=[a_{1},b_{3}]$ then 
$v_{w}v_{j},v_{k}v_{j}\in E$ for all $v_{w}\succ_{r} v_{3}$ and $v_{k}\prec_{r} v_{1}$ as $[a_{1},b_{3}]\subset I_{w}, I_{k}$ and $I_{w}, I_{k}$ are not contained in $I_{3}$. Again when $I_{1}\cap I_{3}=[a_{3},b_{1}]$, then $p_{j}\in [a_{3},b_{1}]\subset [a_{2},b_{2}]=I_{2}$ for all $v_{1}\prec_{r} v_{2} \prec_{r} v_{3}$, which imply $v_{2}v_{j}\in E$. Lastly, when $I_{1}\cap I_{3}=[a_{1},b_{3}]\cup [a_{3},b_{1}]$ then from above we find $j$ to be a column containing all entries as $1$. Therefore in each of these cases $B$ satisfies circular $1$'s property for columns.

\vspace{0.4em}

\noindent\textbf{\textit{$2.$ Let $r,s$ be any two rows of $B$. Then $r^{\ast}-s^{\ast}$ and $s^{\ast}-r^{\ast}$ are circular intervals.}}\\
 (cf. in matrix $B$ of Figure \ref{fg1}, if we take $r=7$, $s=1$, then $r^{\ast}-s^{\ast}=\{6,7\}$, $s^{\ast}-r^{\ast}=\{4\}$. Both are circular intervals.)

\vspace{0.3em}

\noindent To prove $r^{\ast}-s^{\ast}$ and $s^{\ast}-r^{\ast}$ are circular intervals it is sufficient to show neither of them is properly contained in other when both of them are nontrivial. 

\vspace{0.3em}
\noindent On contrary let us assume $s^{\ast}\subset r^{\ast}$. Now if $r,s$ both are of $1$st type then $r_{1}<s_{1}<s_{2}<r_{2}$ and so $p_{r_{1}}<p_{s_{1}}<p_{s_{2}}<p_{r_{2}}$ clearly. Now as $p_{r_{1}},p_{r_{2}}\in I_{r}$ either $[p_{r_{1}},p_{r_{2}}]$ or $[p_{r_{2}},p_{r_{1}}]\subseteq I_{r}$. As $r,s$ are nontrivial and are of $1$st type, it is easy to verify that $[p_{r_{1}},p_{r_{2}}]\subseteq I_{r}$ and $[p_{s_{1}},p_{s_{2}}]\subseteq I_{s}$. Therefore as $p_{r_{1}},p_{r_{2}}\in I_{r}$ and $p_{r_{1}},p_{r_{2}}\notin I_{s}$ imply $[p_{s_{1}},p_{s_{2}}]\subseteq I_{s}\subset [p_{r_{1}},p_{r_{2}}]\subseteq I_{r}$. Hence $I_{s}\subset I_{r}$, which is a contradiction as the arcs are proper.

\vspace{0.2em}

\noindent Next if $r,s$ both are of 2nd type, then $p_{s_{4}}<p_{r_{4}}<p_{r_{3}}<p_{s_{3}}$ clearly. Now as both $r,s$ are nontrivial and are of $2$nd type $[p_{r_{3}},p_{r_{4}}]\subseteq I_{r}$ and $[p_{s_{3}},p_{s_{4}}]\subseteq I_{s}$. Again as $p_{r_{4}},p_{r_{3}}\in I_{r}$ and $p_{r_{4}},p_{r_{3}}\notin I_{s}$ it follows that $[p_{s_{3}},p_{s_{4}}]\subseteq I_{s}\subset [p_{r_{3}},p_{r_{4}}]\subseteq I_{r}$, i.e., $I_{s}\subset I_{r}$, which is not true. Next when $r$ is row of 2nd type and $s$ is of 1st type. Then $[p_{r_{3}},p_{r_{4}}]\subseteq I_{r}$ and $[p_{s_{1}},p_{s_{2}}]\subseteq I_{s}$ clearly. Again $s^{\ast}\subset r^{\ast}$ implies either $s_{2}<r_{4}$ or $r_{3}<s_{1}$. When $s_{2}<r_{4}$ then $p_{s_{1}}<p_{s_{2}}<p_{r_{4}}<p_{r_{3}}$. Hence $p_{r_{4}},p_{r_{3}}\in I_{r}$ and $p_{r_{4}},p_{r_{3}}\notin I_{s}$ imply $[p_{s_{1}},p_{s_{2}}]\subseteq I_{s}\subset [p_{r_{3}}, p_{r_{4}}]\subseteq I_{r}$, i.e., $I_{s}\subset I_{r}$, which is a contradiction. Similar contradiction will arise when $r_{3}<s_{1}$. One can find similar contradiction when $r^{\ast}\subset s^{\ast}$.
    
\vspace{0.5em}  
 
\noindent\textbf{\textit{$3.$ For any two nontrivial rows $s,t$ of $B$, $r^{\ast}-s^{\ast}-t^{\ast}$ and $r^{\ast}-t^{\ast}-s^{\ast}$ are circular intervals where $r$ is a full row.}}\\
(cf. in matrix $B$ of Figure \ref{fg1}, if we take $r=5$, $s=3$, $t=1$, then $r^{\ast}-s^{\ast}-t^{\ast}=\{7\}, r^{\ast}-t^{\ast}-s^{\ast}=\{7\}$. Both are circular intervals.)

\vspace{0.3em}
   
\noindent Now to prove for any two nontrivial rows $s,t$ and a full row $r$ of $B$, $r^{\ast}-s^{\ast}-t^{\ast}$ and $r^{\ast}-t^{\ast}-s^{\ast}$ are circular intervals, it is sufficient to show the followings.
\begin{enumerate}
\item $t_{1}-s_{2}\leq 1$ or $s_{1}=1, t_{2}=n$ when $s,t$ both are rows of $1$st type (assume $s_{1}\leq t_{1}$)

\item either $t_{1}-s_{4}\leq 1$ or $s_{3}-t_{2}\leq 1$ when $s$ be a row of 2nd type and $t$ be a row of 1st type. 
\end{enumerate}

\noindent Now on contrary to $1$, let $t_{1}-s_{2}>1$ when $s_{1}\neq 1$ or $t_{2}\neq n$.
 In particular we take $s_{1}\neq 1$. Then $p_{1}<p_{s_{1}}\leq p_{s_{2}}<p_{s_{2}+1}<p_{t_{1}}\leq p_{t_{2}}\leq p_{n}$. Now as $p_{s_{2}+1},p_{1}\in I_{r}$, either $[p_{s_{2}+1},p_{1}]$ or $[p_{1},p_{s_{2}+1}]\subseteq I_{r}$. If $[p_{s_{2}+1},p_{1}]\subseteq I_{r}$ then $[p_{t_{1}},p_{t_{2}}]\subseteq I_{t}\subset [p_{s_{2}+1},p_{1}]\subseteq I_{r}$ as $p_{s_{2}+1},p_{1}\notin I_{t}$, which imply $I_{t}\subset I_{r}$. Again when $[p_{1},p_{s_{2}+1}]\subseteq I_{r}$ then $[p_{s_{1}},p_{s_{2}}]\subseteq I_{s} \subset [p_{1},p_{s_{2}+1}]\subseteq I_{r}$ as $p_{1},p_{s_{2}+1}\notin I_{s}$
which imply $I_{s}\subset I_{r}$. Hence contradiction arises in both of the cases as the arcs are proper. Similar contradiction will arise when $t_{2}\neq n$.

\vspace{0.3em}

\noindent Now on contrary to $2$, let $t_{1}-s_{4},s_{3}-t_{2}>1$. Then $p_{1}\leq p_{s_{4}}<p_{s_{4}+1}<p_{t_{1}}\leq p_{t_{2}}<p_{s_{3}-1}<p_{s_{3}}\leq p_{n}$. Now as $p_{s_{4}+1},p_{s_{3}-1}\in I_{r}$, either $[p_{s_{4}+1},p_{s_{3}-1}]$ or 
$[p_{s_{3}-1},p_{s_{4}+1}]\subseteq I_{r}$. Now if  $[p_{s_{4}+1},p_{s_{3}-1}]\subseteq I_{r}$ then $[p_{t_{1}},p_{t_{2}}]\subseteq I_{t}\subset [p_{s_{4}+1},p_{s_{3}-1}]\subseteq I_{r}$ as $p_{s_{4}+1},p_{s_{3}-1}\notin I_{t}$, which imply $I_{t}\subset I_{r}$. Again when $[p_{s_{3}-1},p_{s_{4}+1}]\subseteq I_{r}$ then $[p_{s_{3}},p_{s_{4}}]\subseteq I_{s}\subset [p_{s_{3}-1},p_{s_{4}+1}]\subseteq I_{r}$
as $p_{s_{3}-1},p_{s_{4}+1}\notin I_{s}$, which imply $I_{s}\subset I_{r}$. 
Thus we get contradiction as the arcs are proper.

\vspace{0.4em}
\noindent If both $s,t$ are of $2$nd type, then
both $r^{\ast}-s^{\ast}-t^{\ast}$ and $r^{\ast}-t^{\ast}-s^{\ast}$ are circular intervals.

\vspace{0.57em}

\noindent Below we proceed to prove the \textit{\textbf{Sufficiency}} of \textbf{Theorem \ref{pcacd}}.  

\vspace{0.57em}

\noindent Let the augmented adjacency matrix of $G$ ($A=A(G)=(a_{i,j})$) satisfy all the above conditions after an independent permutation of rows and columns of it. Let $B=(b_{i,j})$ be the transformed matrix after these permutations. As permutation of rows and columns of $A$ does not change the adjacency of $G$, if the ${i^{\prime}}^{\text{th}}$ row of $A$ is shifted to the $i^{\text{th}}$ row of $B$ and the ${j^{\prime}}^{\text{th}}$ column of $A$ is shifted to the $j^{\text{th}}$ column of $B$ then 

\begin{equation}\label{l1}
a_{i^{\prime},j^{\prime}}=1\Leftrightarrow b_{i,j}=1  
\end{equation}

\vspace{0.7 em}
\noindent \textbf{Step I:  We construct a submatrix $D=(d_{i,j})$ from $B$ in following manner.}
 
 \begin{enumerate}
 \item Delete all full rows from $B$.
 \item  Next arrange all rows of  $1$st type according to increasing order of $i_{1}$'s. If there are two rows $i,j$ satisfying $i_{1}=j_{1}$, then we place $i$ before $j$ if $i_{2}<j_{2}$. Again if $i_{1}=j_{1},i_{2}=j_{2}$ hold then we can place them in any order. We call the submatrix of $B$ formed only by the rows of $1$st type by $D_{1}$.
 
 \item Now we arrange the rows of $2$nd type according to increasing order of $i_{3}$'s. If there are two rows $i,j$ satisfying $i_{3}=j_{3}$ then we place $i$ before $j$ only if $i_{4}<j_{4}$. Again if $i_{3}=j_{3}, i_{4}=j_{4}$ then we can place them in any order. We call the submatrix of $B$ formed only by the rows of $2$nd type by $D_{2}$.
 
 \item After the row arrangements are done for $D_{1}, D_{2}$, we place $D_{2}$ below $D_{1}$ and call the whole submatrix containing $D_{1}$ and $D_{2}$ by $D$.
 \end{enumerate}

\begin{lem}\label{dm}
Matrix $D$ satisfies monotone circular ordering.
\end{lem}

\begin{proof}
\noindent As permutation of rows does not change circular $1$'s property along rows of a matrix, it is easy to conclude that $D$ satisfies circular $1$'s property along rows and also conditions $2,3$ of the Theorem \ref{pcacd} remain preserved in $D$ for all nontrivial rows of $B$. Therefore for any two arbitrary rows $i,j$ of $D$ where $i$ is a row of $1$st type and $j$ is a row of $2$nd type, we get 

\begin{equation}\label{eq1}
i_{1}\leq j_{3} \hspace{0.37em} \mbox{and} \hspace{0.37em} j_{4}\leq i_{2}.
\end{equation}

 \noindent Hence it is sufficient to prove now that $D$ satisfy circular $1$'s property  along columns.

\vspace{0.39em}
\noindent On contrary, let $l$ be a column in $D$ containing $(0,1,0,1)$ along rows $x,y,z,w$. Now if $d_{x,l}=0\in D_{2}$, then $x_{4}<l\leq y_{4}$ as $d_{y,l}=1$ where $y>x$ in $D$. Again as $z>y$ in $D$, $z_{4}\geq y_{4}$. Therefore $z_{4}\geq l$, which imply $d_{z,l}=1$, but this is not true. Hence $d_{x,l}=0\in D_{1}$. If $x_{1}>l$ then $d_{y,l}=1\in D_{2}$, otherwise $y_{1}\geq x_{1}>l$ imply $d_{y,l}=0$ which is not true. Therefore $x_{2}<l$. Again note that $d_{w,l}=1\in D_{2}$ from above arrangement of $D$ and $w_{4}\leq x_{2}<l$ as $D$ satisfies (\ref{eq1}). Hence $w_{3}\leq l$.
Now if $z\in \mbox{row}(D_{1})$ then $w_{3}\leq l<z_{1}$ contradict  (\ref{eq1}) for the rows $w,z$. Again if 
$z\in \mbox{row}(D_{2})$, then $w_{3}\geq z_{3}>l$ introduces contradiction. 
 
\vspace{0.37em}
\noindent Similarly one can verify that $(1,0,1,0)$ can never appear in any column of $D$. Hence $D$ satisfies circular $1$'s property along columns and therefore $D$ satisfies condition $1$ for all nontrivial rows of $B$. 
\end{proof}

\vspace{0.57em}

\noindent \textbf{Step II: Construction of matrix $M$ from $D$ satisfying monotone circular ordering.}

\vspace{0.7em}

\noindent Depending upon the matrix structure of $D$ we will insert the full rows in $D$ so that the transformed matrix $M=(m_{i,j})$ also satisfy monotone circular ordering. Below we describe the method of inserting a full row in $D$. We place all the full rows consecutively in $D$ to get $M$. It is important to note that condition $2$, $3$ remain intact for all rows of $M$ as we do only row permutation of $B$ here. Again circular $1$'s property along rows also remain preserved for all the rows in $M$.

\vspace{0.57em}
\noindent Let $r$ be an arbitrary full row of $B$ and $S_{1}=\{i|i_{1}=1,i\in \mbox{row}(D_{1})\}$, $S_{2}=\{i|i_{1}>1,i_{2}<n, i\in \mbox{row}(D_{1})\}$, $S_{3}=\{i|i_{2}=n,i\in \mbox{row}(D_{1})\}$ be the set of all possible rows of the matrix $D_{1}$. It is not necessary that all of these sets are nonempty every time. But from monotone circular ordering of $D$ (see Lemma \ref{dm}) it is clear that any row of $S_{1}$ occurs prior to every row of $S_{2}$ and any row of $S_{2}$ occurs prior to every row of $S_{3}$. Let $i,i^{\prime}$ be the first, last row of $S_{1}$, $j,j^{\prime}$ be the first, last row of $S_{2}$, $k$ be the last row of $S_{3}$ and $m$ be the first row of $D_{2}$ occurred in $D$.  

\vspace{0.57em}

\noindent To show that $M$ satisfy monotone circular ordering, it is sufficient to verify that after inserting the full rows in each of the following cases, $M$ satisfies circular $1$'s along columns of it. In this sequel we come across three matrices $F_{1}, F_{2}, F_{3}$ as mentioned in Proposition \ref{prop1} (see Figure \ref{forbid}) which do not satisfy circular $1$'s property along columns for any permutation of rows. Here we use the same notation $F_{1}, F_{2}, F_{3}$ for every row permuted matrices obtained from the matrices of Example \ref{f1f2f3}. From condition $1$ as $B$ satisfies circular $1$'s property along columns, it is clear that $F_{1}, F_{2}, F_{3}$ can not appear as a submatrix in $B$ as well as in $M$.

\vspace{0.57em}
\begin{obs}\label{ob1}
Let $S_{1}, S_{2}, S_{3}\neq \emptyset$. Then $k_{1}\leq i_{2}+1$.
\end{obs}

\begin{proof}
It is known that $j_{1}\leq i_{2}+1$ and $k_{1}\leq j_{2}+1$ from condition $3$ applying for the rows $i,j,r$ and $k,j,r$ respectively. Now if $k_{1}>i_{2}+1$ then $i_{2}<j_{2}$, which imply $j_{1}\leq i_{2}+1\leq j_{2}$, i.e., $d_{j,i_{2}+1}=1$. Therefore $F_{1}$ get induced as a submatrix in $M$ formed by rows $i,j,k,r$ and columns $1,i_{2}+1,n$, which can not be true. Hence $k_{1}\leq i_{2}+1$.
\end{proof}

\vspace{1em}

\noindent \textbf{Method of inserting full rows in $D$ to construct $M$ so that it satisfy circular $1$'s property along columns of it:}
\vspace{0.57em}

\noindent {\bf\textit{Case (i):}} Let $D_{2}\neq \emptyset$. Several cases may arise here, all of them are listed below. 

\begin{enumerate}
\item \noindent Consider the case when $S_{2}=\emptyset$.

\noindent If $S_{1}\neq \emptyset$ then we put $r$ at end of $i^{\prime}$ in $M$. 

\vspace{0.27em}
\noindent Now if in a column $l$ of $M$, $(0,1,0,1)$ occurs along rows $x,r,y,z$ then 
$z_{4}\geq l>x_{2}$, which contradicts condition $2$ for rows $x,z$.

\vspace{0.367em}
\noindent Again if $S_{1}=\emptyset$ then we put $r$ at end of $D$.

\item Next we consider the case when $S_{2}\neq \emptyset$.

\begin{itemize}

\item If $S_{1}, S_{3}\neq \emptyset$ then from condition $2$ we get

\begin{equation}\label{eq2}
m_{4}\leq i_{2}, k_{1}\leq m_{3} \hspace{0.39em}\mbox{and} \hspace{0.39em}j_{1}\leq k_{1}. 
\end{equation}
\noindent  Again condition $3$ tells us    
 
\begin{equation}\label{eq3}
j_{1}\leq i_{2}+1 \hspace{0.39em}\mbox{and}\hspace{0.39em} k_{1}\leq j_{2}+1. 
\end{equation}

\vspace{0.57em}
\noindent \textbf{a)} If $\mathbf{k_{1}-m_{4}\leq 1}$ then we place $r$ just before $m$ in $D$. 

\vspace{0.39em}
\noindent Now if a column $l$ of $M$ contains $(0,1,0,1)$ along rows $x,y,z,r$ then $z_{1}-x_{2}>1$, which imply $x\in S_{1}$ and $z\in S_{3}$ otherwise if any one of $x,z\in S_{2}$ then condition $3$ gets contradicted for rows $x,z,r$. Now as $i\leq x$ in $D$, $z_{1}-m_{4}\geq z_{1}-i_{2}\geq z_{1}-x_{2}>1$ using (\ref{eq2}). Again $z\leq k$ helps us to conclude $z_{1}-m_{4}\leq k_{1}-m_{4}\leq 1$. Hence contradiction arises.

Again if there exists a column $l$ containing $(1,0,1,0)$ along rows $x,y,r,z$ in $M$. Then $y_{1}-z_{4}>1$ clearly. But as $y\leq k$ and $m\leq z$, we get $y_{1}-z_{4}\leq k_{1}-m_{4}\leq 1$, which is a contradiction. 

\vspace{0.27em}
\noindent Therefore $M$ satisfies circular $1$'s property along columns.

\vspace{0.377em}
\noindent\textbf{b)} Next we consider the case when $\mathbf{k_{1}-m_{4}>1}$.
\vspace{0.39em}

\noindent From Observation \ref{ob1} it is known that $k_{1}\leq i_{2}+1$. Therefore $m_{4}<i_{2}$.

\vspace{0.24em}
 
\noindent \textbf{(i)} Now if for all $m^{\prime}\geq m$, $m^{\prime}_{3}\leq i_{2}+1$ then we place $r$ at end of $D$.
  \vspace{0.3em}
  
\noindent \textbf{(ii)} Next we consider the case when there exist some $m^{\prime}\in \mbox{row}(D_{2})$
satisfying $m^{\prime}_{3}>i_{2}+1$ (we take first such row as $m^{\prime}$). In this case we place $r$ just before $m^{\prime}$ in $D$.

\vspace{0.37em}
 \noindent \textbf{Claim 1: $\mathbf{k_{1}-m^{\prime}_{4}\leq 1}$.}

\vspace{0.27em}

\noindent \textit{proof:} If not let $k_{1}-m^{\prime}_{4}>1$. From construction of the matrix $D$, it is known that $i_{2}\leq j_{2}$. Now when $j_{2}>i_{2}$, $F_{2}$ get induced as a submatrix in $M$ consisting of rows $i,j,k,m^{\prime},r$ and columns $1,k_{1}-1,i_{2}+1,n$ if $j_{1}<k_{1}$ and $F_{1}$ is induced consisting of rows $i,j,m^{\prime},r$ and columns $k_{1}-1,i_{2}+1,n$ if $j_{1}=k_{1}$. Again when $j_{2}=i_{2}$, $F_{1}$ get induced by the rows $j,k,m^{\prime},r$ and columns $1, k_{1}-1,i_{2}+1$ in $M$. Hence contradiction arises in both of these cases. Therefore $k_{1}-m^{\prime}_{4}\leq 1$.

\vspace{0.37em}

\noindent Clearly $m^{\prime}>m$ in this case follows from above claim.

\vspace{0.37em}
\noindent \textbf{Claim 2: If $m^{\prime}_{4}<i_{2}$ then $k_{1}<i_{2}+1$. Moreover, if $d_{m,l}=0$ for any $l, k_{1}\leq l\leq i_{2}$, then $d_{m^{\prime},l}=1$.}

\vspace{0.27em}

\noindent \textit{proof:} If $m^{\prime}_{4}<i_{2}$ then $k_{1}<i_{2}+1$ as $k_{1}-m^{\prime}_{4}\leq 1$ from Claim $1$. Now in this case one can able to show if $d_{m,l}=0$ for any $l, k_{1}\leq l\leq i_{2}$, then $d_{m^{\prime},l}=1$. 

\vspace{0.27em}
\noindent In other cases when $j_{2}>i_{2}$, we get $F_{3}$ (see Figure \ref{forbid}) as a induced submatrix in $M$ formed by rows $i,j,k,m,m^{\prime},r$ and columns $1,k_{1}-1,l,i_{2}+1,n$ when $j_{1}<k_{1}$ and $F_{2}$ get induced by rows $i,j,m,m^{\prime},r$ and columns $k_{1}-1,l,i_{2}+1,n$ when $j_{1}=k_{1}$. Again if $j_{2}=i_{2}$, $F_{2}$ is formed by rows $j,k,m,m^{\prime},r$ and columns $1,k_{1}-1,l,i_{2}+1$ when $j_{1}<k_{1}$ and $F_{1}$ is induced by rows $j,m,m^{\prime},r$ and columns $k_{1}-1,l,i_{2}+1$ when $j_{1}=k_{1}$ where $k_{1}\leq l\leq i_{2}$. Both of these cases introduce contradiction as $B$ satisfies circular $1$'s property along columns follows from condition $1$.

\vspace{0.27em}

\vspace{0.27em}
\noindent We now prove that after placing $r$ in above way $M$ will satisfy circular $1$'s property along columns.
If $m^{\prime}_{4}<i_{2}$ then the verification is trivial from Claim $2$. When $m^{\prime}_{4}=i_{2}$ then if a column ($l$ say) of $M$ contains $(0,1,0,1)$ along rows $x,y,z,r$, then $z$ must come from $S_{3}$ or $\mbox{row}(D_{2})$ and $x$ must come from $S_{1}$, otherwise condition $3$ gets contradicted for rows $x,z,r$. Now if $z\in \mbox{row} (D_{2})$,
then $z_{3}-x_{2}>1$. But as $i<x$, $i_{2}\leq x_{2}$ which imply $z_{3}-i_{2}\geq z_{3}-x_{2}>1$. But this situation can not arise as $m^{\prime}$ is the first such row satisfying $m^{\prime}_{3}>i_{2}+1$. Again if $z\in S_{3}$, then $z_{1}-x_{2}>1$. This imply $k_{1}-m^{\prime}_{4}\geq z_{1}-m^{\prime}_{4}\geq z_{1}-x_{2}>1$, which is not true from Claim $1$. Hence $M$ satisfies circular $1$'s property along columns.

\vspace{0.57em}

\item  If $S_{1}\neq \emptyset, S_{3}=\emptyset$.  Applying condition $3$ on rows $i,j^{\prime},r$ we get $j^{\prime}_{1}\leq i_{2}+1$. Again from condition $2$ we get $m_{4}\leq i_{2}$.

\vspace{0.27em}

\noindent a) If $\mathbf{j^{\prime}_{1}-m_{4}\leq 1}$ then we place $r$ just before $m$ in $D$.

\vspace{0.27em}

\noindent b) We now consider the case when $\mathbf{j^{\prime}_{1}-m_{4}>1}$.
\vspace{0.27em}

If for all $m^{\prime}\geq m$ in $D_{2}$, $m^{\prime}_{3}\leq i_{2}+1$, then we place $r$ at end of $D$. Next we consider the case when there exist some $m^{\prime}\in \mbox{row}(D_{2})$ satisfying $m^{\prime}_{3}>i_{2}+1$ (if there exist more than one such rows take first such row as $m^{\prime}$).  In this case we place $r$ just before $m^{\prime}$ in $D$ whether $m^{\prime}_{4}<i_{2}$ or $m^{\prime}_{4}=i_{2}$.

\vspace{0.27em}
\noindent \textbf{Claim $1$:} $\mathbf{j^{\prime}_{1}-m^{\prime}_{4}\leq 1}$.

\vspace{0.27em}
\textit{proof:} Now if $j^{\prime}_{1}-m^{\prime}_{4}>1$ then $m^{\prime}_{3}-j^{\prime}_{2}\leq 1$ follows from condition $3$ applying on the rows $m^{\prime},j,r$. Hence 
$i_{2}<m^{\prime}_{3}-1\leq j^{\prime}_{2}$. Therefore $F_{1}$ (see Figure \ref{forbid}) get induced in $M$ as a submatrix consists of rows $i,j^{\prime},m^{\prime},r$ and columns $j^{\prime}_{1}-1,i_{2}+1,n$. But this is not possible as $B$ satisfies circular $1$'s property along columns follows from condition $1$. 
Therefore $j^{}_{1}-m^{\prime}_{4}\leq 1$.

\vspace{0.29em}
\noindent Hence from Claim $1$, $m^{\prime}>m$ clearly.

\vspace{0.27em}
\textbf{Claim $2$: If $m^{\prime}_{4}<i_{2}$, then if $d_{m,l}=0$ for some $j^{\prime}_{1}\leq l\leq i_{2}$, then $d_{m^{\prime},l}=1$} 
 
\vspace{0.27em}
 
\noindent \textit{proof:} 
Otherwise $F_{2}$ (see Figure \ref{forbid}) get induced by the rows $i,j^{\prime},m,m^{\prime},r$ and columns $j^{\prime}_{1}-1, l, i_{2}+1, n$ for some $j^{\prime}_{1}\leq l \leq i_{2}$ when $j^{\prime}_{2}> i_{2}$ and $F_{1}$ is induced by the rows
$j^{\prime},m,m^{\prime},r$ and columns $j^{\prime}_{1}-1,l,i_{2}+1$ when $j^{\prime}_{2}=i_{2}$. In both of these cases we get contradiction. 

\vspace{0.3em}
\noindent The proof that circular $1$'s property along columns is satisfied by $M$ is similar to the previous case, which can be verified using above claims.

\vspace{0.37em}

\item If $S_{1}=\emptyset$, $S_{3}\neq \emptyset$.  From condition $2$ and condition $3$ applying on rows $j,k,r$ we get $j_{1}\leq k_{1}\leq j_{2}+1$.

\vspace{0.37em}
\vspace{0.27em}

\noindent a) If $\mathbf{k_{1}-m_{4}\leq 1}$ then we place $r$ just before $m$ in $D$.
\vspace{0.27em}

\noindent b) Now consider the case when $\mathbf{k_{1}-m_{4}>1}$.

\vspace{.31em}
\noindent  Then if for all $m^{\prime}\geq m$ in $D_{2}$, $m^{\prime}_{3}\leq j_{2}+1$, then we place $r$ at end of $D$. If there exist some $m^{\prime}\in \mbox{row} (D_{2})$ satisfying $m^{\prime}_{3}>j_{2}+1$ (take first such row as $m^{\prime}$) then $j_{1}-1\leq m^{\prime}_{4}\leq j_{2}$, follows from condition $2,3$ applying on rows $m^{\prime},j,r$. In this case we place $r$ just before $m^{\prime}$ in $D$.

\vspace{0.39em}

\noindent \textbf{Claim $1$: $\mathbf{k_{1}-m^{\prime}_{4}\leq 1}$}.

\textit{proof:} If $k_{1}-m^{\prime}_{4}>1$ then $j_{1}<k_{1}$ clearly. Therefore $F_{1}$ get induced in $M$ by the rows $\{j,k,m^{\prime},r\}$ and columns $j_{1}-1,k_{1}-1,j_{2}+1$, which introduces contradiction as $B$ satisfies circular $1$'s property along columns from condition $1$.

\vspace{0.27em}
\noindent Therefore $m^{\prime}>m$ clearly from above claim.
\vspace{0.27em}

\noindent \textbf{Claim $2$: If $m^{\prime}_{4}<j_{2}$, then if $d_{m,l}=0$ then $d_{m^{\prime},l}=1$ for all $k_{1}\leq l\leq j_{2}$}.

\vspace{0.27em}

 \noindent\textit{proof}
 Otherwise $F_{2}$ get induced in $M$ by the rows $j,k,m,m^{\prime},r$ and columns $1,k_{1}-1,l,j_{2}+1$ when $j_{1}<k_{1}$ and $F_{1}$ get induced by rows $j,m,m^{\prime},r$ and columns $k_{1}-1,l,j_{2}+1$ when $j_{1}=k_{1}$, which is not true as $B$ satisfies circular $1$'s property along columns. 
 
 \vspace{0.27em}
 \noindent  Following above claims analogously one can verify that $M$ satisfies circular $1$'s property along columns. 
 
 \vspace{0.27em}
 
\item If $S_{1}=S_{3}=\emptyset$. Then proceeding similarly as the previous cases one can able to place full row $r$ in $M$ so that it satisfy circular $1$'s property along columns. 
\end{itemize}

\end{enumerate}

\vspace{1em}

\noindent {\bf\textit{Case (ii):}} Let $D_{2}=\emptyset$. 

\vspace{0.37em}

\noindent If $S_{1}\neq \emptyset$ and $S_{2}=\emptyset$ we place $r$ after $i^{\prime}$ in $M$ otherwise we place $r$ at end of $D$ in $M$.

\vspace{0.37em}

\noindent Now if there is a column $l$ in $M$ containing $(0,1,0,1)$ along rows $x,y,z,r$ then $z_{1}-x_{2}>1$. Therefore from condition $3$ applying for the rows $x,z,r$ one can get $x\in S_{1}$ and $z\in S_{3}$, i.e., $S_{1}, S_{3}\neq \emptyset$. Again as row $r$ occurs after row $z$ in $M$, $S_{2}$ must be nonempty according to the preceding rule of placement of $r$. Now from definition of $i,k$ we get $i\leq x$ and $z\leq k$ in $D$. Hence $z_{1}-x_{2}\leq k_{1}-i_{2}\leq 1$ (see Observation \ref{ob1}), which is a contradiction.

\vspace{0.6em}
\noindent Now it is easy to conclude from Definition \ref{MCO} that constructing $M$ by above method actually gives the monotone circular ordering of it.

\begin{figure}\centering
\[
\begin{array}{c|ccccccc|}
\multicolumn{1}{c}{} & v_1 & \multicolumn{1}{c}{v_2} & \multicolumn {1}{c}{v_3} & \multicolumn{1}{c}{v_4} &\multicolumn{1}{c}{v_{5}} &\multicolumn{1}{c}{v_{6}} &\multicolumn{1}{c}{v_{7}} \\
\cline{2-8}
v_1 & 1 & 1 & 1 & 1 & 1 & 1 & 1\\
v_2 & 1 & 1 & 1 & 1 & 1 & 0 & 0\\ 
v_3 & 0 & 0 & 1 & 1 & 1 & 1 & 0\\
v_4 & 0 & 0 & 1 & 1 & 1 & 0 & 0 \\
v_5 & 1 & 1 & 1 & 1 & 1 & 1 & 1 \\
v_6 & 1 & 1 & 1 & 0 & 0 & 1 & 1\\
v_7 & 1 & 1 & 0 & 0 & 0 & 1 & 1\\
\cline{2-8}
\end{array}
\qquad
\begin{array}{c|ccccccc|}
\multicolumn{1}{c}{} & v_2 & \multicolumn{1}{c}{v_1} & \multicolumn {1}{c}{v_3} & \multicolumn{1}{c}{v_5} &\multicolumn{1}{c}{v_{4}} &\multicolumn{1}{c}{v_{6}} &\multicolumn{1}{c}{v_{7}} \\
\cline{2-8}
v_2 & 1 & 1 & 1 & 1 & 1 & 0 & 0\\ 
v_4 & 0 & 0 & 1 & 1 & 1 & 0 & 0 \\
v_3 & 0 & 0 & 1 & 1 & 1 & 1 & 0\\
v_5 & 1 & 1 & 1 & 1 & 1 & 1 & 1 \\
v_6 & 1 & 1 & 1 & 0 & 0 & 1 & 1\\
v_7 & 1 & 1 & 0 & 0 & 0 & 1 & 1\\
v_1 & 1 & 1 & 1 & 1 & 1 & 1 & 1\\
\cline{2-8}
\end{array}
\]
\[
\begin{array}{c|ccccccc|}
\multicolumn{1}{c}{} & v_2 & \multicolumn{1}{c}{v_1} & \multicolumn {1}{c}{v_3} & \multicolumn{1}{c}{v_5} &\multicolumn{1}{c}{v_{4}} &\multicolumn{1}{c}{v_{6}} &\multicolumn{1}{c}{v_{7}} \\
\cline{2-8}
v_2 & 1 & 1 & 1 & 1 & 1 & 0 & 0\\ 
v_4 & 0 & 0 & 1 & 1 & 1 & 0 & 0 \\
v_3 & 0 & 0 & 1 & 1 & 1 & 1 & 0\\
v_1 & 1 & 1 & 1 & 1 & 1 & 1 & 1\\
v_5 & 1 & 1 & 1 & 1 & 1 & 1 & 1 \\
v_7 & 1 & 1 & 0 & 0 & 0 & 1 & 1\\
v_6 & 1 & 1 & 1 & 0 & 0 & 1 & 1\\
\cline{2-8}
\end{array}
\]

\[
\begin{array}{c|c|c|c|c|c|c|c|ccc}
\multicolumn{1}{c}{} & \multicolumn{1}{c}{p_2} & \multicolumn{1}{c}{p_{1}} & \multicolumn{1}{c}{p_3} & \multicolumn{1}{c}{p_5} &\multicolumn{1}{c}{p_{4}}& \multicolumn{1}{c}{p_{6}} & \multicolumn{1}{c}{p_{7}}  & \multicolumn{1}{c}{} & \\

\multicolumn{1}{c}{}&\multicolumn{1}{c}{1.87}&\multicolumn{1}{c}{2.75}&\multicolumn{1}{c}{3.91}&\multicolumn{1}{c}{4}&\multicolumn{1}{c}{5} & \multicolumn{1}{c}{8.88}& \multicolumn{1}{c}{10} & \multicolumn{1}{c}{}& \\

\multicolumn{1}{c}{} & \multicolumn{1}{c}{v_2} & \multicolumn{1}{c}{v_{1}} & \multicolumn{1}{c}{v_3} & \multicolumn{1}{c}{v_5} &\multicolumn{1}{c}{v_{4}}& \multicolumn{1}{c}{v_6} & \multicolumn{1}{c}{v_7}  &\multicolumn{1}{c}{} &  \\
\cline{2-8}
v_2 & \multicolumn{1}{c}{1} &\multicolumn{1}{c}{2} & \multicolumn{1}{c}{3} & \multicolumn{1}{c}{4} & 5 & \multicolumn{1}{c}{6}& &  &[1.87,6]& I_{2} \\ 

\cline{2-3}

v_4 & \multicolumn{1}{c}{} & & \multicolumn{1}{c}{} & \multicolumn{1}{c}{} &  &\multicolumn{1}{c}{7} &  & &[3.5,7]& I_{4} \\ 

\cline{7-7}

v_3 & \multicolumn{1}{c}{} & & \multicolumn{1}{c}{} & \multicolumn{1}{c}{} & \multicolumn{1}{c}{}  & 8&9  & &[3.58,9]& I_{3} \\ 
\cline{2-3}
\cline{8-8}

v_1 & \multicolumn{1}{c}{} & & \multicolumn{1}{c}{} & \multicolumn{1}{c}{} & \multicolumn{1}{c}{}  & \multicolumn{1}{c}{}& 10 & 11&[3.66,2.75]& I_{1} \\ 

v_5 & \multicolumn{1}{c}{} & & \multicolumn{1}{c}{} & \multicolumn{1}{c}{} & \multicolumn{1}{c}{}  & \multicolumn{1}{c}{}&  &12 &[3.75,2.83]& I_{5} \\

\cline{4-6}

v_7 & \multicolumn{1}{c}{} & & \multicolumn{1}{c}{} & \multicolumn{1}{c}{} &  &\multicolumn{1}{c}{} & &13 &[8.77,2.91]& I_{7} \\
\cline{4-4} 

v_6 & \multicolumn{1}{c}{} &\multicolumn{1}{c}{} &  & \multicolumn{1}{c}{} &  & \multicolumn{1}{c}{}& & 14 &[8.88,3.91]& I_{6} \\ 

\cline{2-8}
\end{array}
\]
\caption{Augmented adjacency matrix $A=A(G)$ of a proper circular-arc catch digraph $G$ (left up), the row and column permutated matrix $B$ from $A$ satisfying circular $1$'s property (along both rows and columns) (right up), row permuted matrix $M$ from $B$ satisfying monotone circular ordering (center), a proper circular-arc catch digraph representation of $G$ (bottom)}\label{pcad}
\end{figure}

\begin{lem}
Monotone circular ordering is satisfied by $M$.
\end{lem}

\noindent Moreover, we can separate the zeros from the ones of $M$ by drawing upper and lower stairs (see Figure \ref{pcad}) in it as it satisfies monotone circular ordering. Note that permutation of rows of $B$ does not effect the adjacency of the graph $G$, if the $i^{\text{th}}$ row of $B$ is shifted to 
the $k^{\text{th}}$ row of $M$ then 
\begin{equation}\label{l2}
b_{i,j}=1\Longleftrightarrow m_{k,j}=1
\end{equation}

\vspace{0.3em}
\noindent \textbf{Step III}: Defining circular-arcs $I_{i}=[a_{i},b_{i}]$ for the $i^{\text{th}}$ row of $M$.

\vspace{0.57em}

\noindent \noindent First we separate the $1$'s from the $0$'s by drawing  polygonal paths (known as stairs) as in Figure \ref{pcad}. The existence of the stairs is guaranteed from construction of $M$. Note that the upper stair is drawn from top left to bottom right of the matrix and there can exist atmost three stairs in $M$. We associate natural numbers in increasing order on the upper stair of $M$ starting from the top of first column. Let $r_{i}$ be the number on the stair in the $i^{\text{th}}$ row and $l_{i}$ be the number on the stair in the $i^{\text{th}}$ column (e.g., in Figure \ref{pcad}, $l_{i}=1,2,3,4,5,8,10$ and $r_{i}=6,7,9,11,12,13,14$
for $i=1,2,3,4,5,6,7$ respectively). Let $i_{1},i_{2}(i_{3},i_{4})$ denote the first and last column numbers where the circular stretch of ones starts and ends for rows of $1$st ($2$nd) type in $M$. From construction of $M$ one can note that the rows of $1$st type occur consecutively before the $2$nd type rows. Further we prove in Step II
that $M$ satisfies monotone circular ordering. Then by definition of $l_{i},r_{i}$ we get $r_{i}>l_{j}$ for all $j$ belonging to the circular interval 
$[i_{1},i_{2}]$ or $[i_{3},i_{4}]$. Thus 
\begin{equation}
m_{i,j}=1\Rightarrow r_{i}>l_{j}
\end{equation}

\vspace{0.3em}
\noindent Let $s_{i}=n+\text{max}\{j | j_{1}=i_{1}\hspace{0.1em}\mbox{or} \hspace{0.1em}j_{3}=i_{1} \hspace{0.1em}\mbox{or} \hspace{0.1em} j_{4}=i_{1}\}$,
$s^{\prime}_{i}=n+ \text{max}\{j | j_{3}=i_{3} \hspace{0.1em}\mbox{or} \hspace{0.1em} j_{4}=i_{3}\}$, $k_{i}=n+ \text{max}\{j | j_{4}=i_{4}\}$ where $|V|=n$.\\
(cf. in matrix M of Figure \ref{pcad}, for $i=3$, $i_{1}=3, i_{2}=6, s_{3}=14$ and for $i=7$, $i_{3}=6, i_{4}=3, s^{\prime}_{7}=k_{7}=14$.) 

\vspace{0.3em}
\noindent We assign circular-arc $I_{i}=[a_{i},b_{i}]$ where
 
\begin{equation}
a_{i}=\begin{cases} 
l_{i_{1}}+\dfrac{n+i-i_{1}}{s_{i}-i_{1}+1} \hspace{0.2em} \mbox{for rows of 1st type}\\
l_{i_{3}}+\dfrac{n+i-i_{3}}{s^{\prime}_{i}-i_{3}+1} \hspace{0.2em} \mbox{for rows of 2nd type}
\end{cases} 
\mbox{and} \hspace{0.35em} 
b_{i}=\begin{cases}
r_{i} \hspace{0.2em} \mbox{ for rows of 1st type}\\ 
l_{i_{4}}+\dfrac{n+i-i_{4}}{k_{i}-i_{4}+1} \hspace{0.2em} \mbox{for rows of 2nd type}\\
 \end{cases}
\end{equation}

\vspace{1em}

\noindent If the $i^{\text{th}}$ row is of first type, $b_{i}=r_{i}\geq l_{i_{1}}+1>a_{i}$ and when the $i^{\text{th}}$ row is of 2nd type $b_{i}=l_{i_{4}}+\dfrac{n+i-i_{4}}{k_{i}-i_{4}+1} <l_{i_{4}}+1\leq l_{i_{3}}<l_{i_{3}}+\dfrac{n+i-i_{3}}{s^{\prime}_{i}-i_{3}+1}=a_{i}$.

\vspace{1em}

\noindent Now if $i$ is a row of 1st type and $j$ is a row of 2nd type then from condition $2$ it follows that $j_{4}\leq i_{2}$ and $i_{1}\leq j_{3}$.
As $j_{4}\leq i_{2}$ we get $l_{j_{4}}\leq l_{i_{2}}$ and hence it follows that $b_{j}<l_{j_{4}}+1\leq r_{i}=b_{i}$. Moreover, $i<j$ in $M$ follows from its arrangement to satisfy monotone circular ordering. Hence $a_{i}<l_{i_{1}}+1\leq l_{j_{3}}<a_{j}$ when $i_{1}<j_{3}$
and $a_{i}<a_{j}$ when $i_{1}=j_{3}$ as $i<j$ and $s_{i}=s^{\prime}_{j}$.

\vspace{1 em}

\noindent Again if $i,j$ both are of $1$st type and $i<j$ in $M$ then from condition $2$ we get $i_{1}\leq j_{1}$ and $i_{2}\leq j_{2}$. Hence $l_{i_{1}}\leq l_{j_{1}}$. 
Now for $i_{1}<j_{1}$ we get $a_{i}<l_{i_{1}}+1\leq l_{j_{1}}<a_{j}$ and when $i_{1}=j_{1}$, $s_{i}=s_{j}$ imply $a_{i}<a_{j}$ as $i<j$.
Next $b_{i}=r_{i}<r_{j}=b_{j}$ (as $i<j$). 

\vspace{1 em}

\noindent Similarly one can verify that if $i,j$ both are of rows of $2$nd type and $i<j$ in $M$ then $a_{i}<a_{j}$ and $b_{i}<b_{j}$.
 
\vspace{1 em} 
\noindent Therefore the set of circular-arcs $[a_{i},b_{i}]$ gives proper representation.

\vspace{1 em}

\noindent\textbf{Step IV:} Defining points $p_{j}$ for each column $j$ of $M$.

\vspace{1 em}

\noindent Let $S_{j}=\{a_{i}|i_{1}=j \hspace{0.2em}\text{or}\hspace{0.2em} i_{3}=j\}$, $Q_{j}=\{b_{i}|i_{4}=j\}$. 
\vspace{0.5em}

\noindent We assign $p_{j}=\text{max}\{l_{j},\text{max}\{S_{j}\},\text{min}\{Q_{j}\}\}$.

\vspace{0.7em}

\noindent \textbf{Step V:} We show that $m_{i,j}=1$ if and only if $p_{j}\in I_{i}$.

\vspace{0.77em}

\noindent From definition of $S_{j},Q_{j}$ one can note $l_{j}<\text{max}\{S_{j}\},\text{min}\{Q_{j}\}<l_{j}+1$. Hence $l_{j}<p_{j}<l_{j}+1$ when at least one of $S_{j}$, $Q_{j}$ is nonempty and $p_{j}=l_{j}$ otherwise. Moreover, when $S_{j}, Q_{j}\neq \emptyset$ then $l_{j}<\text{max}\{S_{j}\}<\text{min}\{Q_{j}\}\leq p_{j}$ as $i<j$ in $M$ follows from monotone circular ordering.

\vspace{0.57em}

\noindent \textit{Case (i):} $m_{i,j}=1$.

\vspace{0.57em}
\noindent Now if $i$ is a row of $1$st type then $i_{1}\leq j\leq i_{2}$. Therefore when $i_{1}<j$, $l_{i_{1}}<a_{i}<l_{i_{1}}+1\leq l_{j}\leq p_{j}$. Again if $i_{1}=j$ then $a_{i}\in S_{j}$, which imply $S_{j}\neq \emptyset$. Thus $l_{j}=l_{i_{1}}<a_{i}\leq\text{max}\{S_{j}\}\leq p_{j}$. Moreover, $p_{j}<l_{j}+1\leq r_{i}=b_{i}$. Hence $p_{j}\in[a_{i},b_{i}]$ as $a_{i}<b_{i}$. 

\vspace{0.57em}
\noindent Next if $i$ is a row of $2$nd type then either $1\leq j\leq i_{4}$ or $i_{3}\leq j \leq n$ in $M$. 

\vspace{0.57em}

\noindent If $j<i_{4}$ then $l_{j}\leq p_{j}<l_{j}+1\leq l_{i_{4}}<b_{i}$. Again if $j=i_{4}$ then $b_{i}\in Q_{j}$ which imply $Q_{j}\neq \emptyset$. Hence
$p_{j}=\text{min} \{Q_{j}\}\leq b_{i}$. Next if $i_{3}<j$ then $l_{i_{3}}<a_{i}<l_{i_{3}}+1\leq l_{j}\leq p_{j}$. Again when $i_{3}=j$, $a_{i}\in S_{j}$ imply $a_{i}\leq \text{max}\{S_{j}\}\leq p_{j}$. Therefore $p_{j}\in [a_{i},b_{i}]$ as $b_{i}<a_{i}$.

\vspace{1 em}

\noindent\textit{Case (ii)}  $m_{i,j}=0$.

\vspace{0.59em}
\noindent Now if $i$ is a row of $1$st type then either $j<i_{1}$ or $j>i_{2}$. Now if $j<i_{1}$, $l_{j}\leq p_{j}<l_{j}+1\leq l_{i_{1}}<a_{i}$. Again if $j>i_{2}$, $b_{i}=r_{i}<l_{j}\leq p_{j}$. Hence $p_{j}\notin [a_{i},b_{i}]$ as $a_{i}<b_{i}$. Again when $i$ is a row of $2$nd type then $i_{4}<j<i_{3}$ imply $l_{i_{4}}<b_{i}<l_{i_{4}}+1\leq l_{j}\leq p_{j}<l_{j}+1\leq l_{i_{3}}<a_{i}$. Hence $p_{j}\notin I_{i}$ as $b_{i}<a_{i}$.

\vspace{1 em}

\noindent  This completes all the verifications. Thus from (\ref{l1}) and (\ref{l2}) it is easy to conclude now that $A(G)$ satisfies its adjacency with respect to the above circular-arc catch digraph representation $\{(I_{v},p_{v})|v\in V\}$. Therefore $G$ becomes a proper circular-arc catch digraph. Thus the sufficiency part is proved for Theorem \ref{pcacd}.
\end{proof}

\vspace{0.7em}

\begin{defn}\label{pocac}
Let $G=(V,E)$ be an oriented circular-arc catch digraph. We call $G$ as {\em  proper oriented circular-arc catch digraph} if there exists a circular-arc catch digraph representation $\{(I_{v},p_{v})|v\in V\}$ of $G$ where the circular-arcs are proper, i.e; no arc is properly contained in other. 
\end{defn}

\noindent Now using Theorem \ref{pca} we prove that the underlying undirected graph of a proper circular-arc catch digraph is a proper circular-arc graph.

\begin{thm}\label{prp1}
Let $G$ be a proper oriented circular-arc catch digraph. Then $U(G)$ is a proper circular-arc graph.
\end{thm}

\vspace{0.3em}
\noindent In the following we proceed to prove the above Theorem.

\vspace{0.37em}

\noindent Let $G=(V,E)$ be a proper oriented circular-arc catch digraph with representation $\{(I_{i},p_{i})|v_{i}\in V\}$ where $I_{i}=[a_{i},b_{i}]$ is a circular-arc and $p_{i}$ is a point within it. We can consider each point $p_{i}$ to be distinct otherwise if $p_{i}=p_{j}$ for some $i,j$ then $v_{i}v_{j},v_{j}v_{i}\in E$ leads to a contradiction as $G$ is oriented. Starting from a base point we arrange the vertices according to increasing (clockwise) order of $p_{i}$'s around the circle and consider the circular ordering as $<_{\sigma}$. For convenience henceforth we mean $v_{i}<_{\sigma}v_{j}\Leftrightarrow p_{i}<_{\sigma} p_{j}$.

\begin{defn}\label{fourset}
For each vertex $v_{i}\in V$ we define the following four sets associated to it,
$(S_{1})_{i}=\{v_{j}|p_{i}\in[p_{j},b_{j}]\}, (S_{2})_{i}=\{v_{j}|p_{j}\in[p_{i},b_{i}]\}, (S_{3})_{i}=\{v_{j}|p_{j}\in[a_{i},p_{i}]\}, (S_{4})_{i}=\{v_{j}|p_{i}\in[a_{j},p_{j}]\}$. 
\end{defn}

\vspace{0.3em}

\noindent Let us denote the increasing (clockwise) order of the left end points ($a_{i}$) and right end points ($b_{i}$) of the circular-arc $I_{i}$'s by $<_{r}$ and $<_{c}$ respectively.

\vspace{0.3em}

\noindent Below we list some important Lemmas which will lead us to the main proof.

\vspace{0.37em}

\begin{lem}\label{a}
 For any $i,j,k$, $p_{i}<_\sigma p_{k}<_{\sigma} p_{j} \Longleftrightarrow a_{i}<_{r} a_{k}<_{r} a_{j}\Longleftrightarrow b_{i}<_{c} b_{k}<_{c} b_{j}$ (up to some cyclic permutation) 
\end{lem}

\begin{proof}
As no arc can be properly contained in other, up to some cyclic permutation $a_{i}<_{r} a_{k}<_{r} a_{j}\Longleftrightarrow b_{i}<_{c} b_{k}<_{c} b_{j}$ clearly. \\
Let $p_{i}<_{\sigma} p_{k}<_\sigma p_{j}$. Now we assume on contrary that $a_{i}<_{r} a_{j}<_{r} a_{k}$.  As $a_{j}<_{r} a_{k}$, $p_{k}\in [a_{j},p_{j}]\subseteq I_{j}$. This imply $v_{j}v_{k}\in E$. Now as $G$ is oriented, $v_{k}v_{j}\notin E$, i.e., $p_{j}\notin I_{k}$. Hence $I_{k}=[a_{k},b_{k}]\subsetneq [a_{j},p_{j}]\subseteq I_{j}$, which is a contradiction as the arcs are proper. Similar contradiction will arise for two other permutations of $a_{i}$'s. 
\end{proof}

\vspace{0.3em}
\noindent Henceforth we will use $<_{\sigma}$ throughout this proof instead of $<_{r}$ and $<_{c}$ for convenience. 

\begin{lem}\label{c}
Let $v_{j}<_{\sigma} v_{i}$ and $v_{i}v_{j}\in E$ then $v_{j}\in (S_{3})_{i}$ if and only if $v_{i}\in (S_{4})_{j}$. Again if $v_{i}<_{\sigma} v_{j}$ then $v_{j}\in (S_{2})_{i}$ if and only if 
$v_{i}\in (S_{1})_{j}$.
\end{lem}

\begin{proof}
Follows from Definition \ref{fourset}.
\end{proof}

\begin{lem}\label{g}
If $v_{i}\in (S_{1})_{j}$ or $(S_{3})_{j}$ then $v_{i}\notin (S_{2})_{j}$ or $(S_{4})_{j}$. Again if $v_{i}\in (S_{2})_{j}$ or $(S_{4})_{j}$ then $v_{i}\notin (S_{1})_{j}$ or $(S_{3})_{j}$.
\end{lem} 
 
\begin{proof}
If $v_{i}\in (S_{1})_{j}$ then $p_{j}\in [p_{i},b_{i}]\subseteq I_{i}$. As $G$ is oriented $p_{i}\notin I_{j}$. This imply $v_{i}\notin (S_{2})_{j}$. Again when $p_{j}\in [p_{i},b_{i}]$ then $p_{j}\notin [a_{i},p_{i}]$ as $p_{i}\neq p_{j}$. Therefore $v_{i}\notin (S_{4})_{j}$. Next when $v_{i}\in (S_{3})_{j}$ then $p_{i}\in [a_{j},p_{j}]\subseteq I_{j}$ imply $p_{j}\notin I_{i}$ as $G$ is oriented. Hence $v_{i}\notin (S_{4})_{j}$.  Again as $p_{i}\neq p_{j}$, $p_{i}\in [a_{j},p_{j}]$ and $p_{i}\in [p_{j},b_{j}]$ can not happen simultaneously. Therefore $v_{i}\notin (S_{2})_{j}$. Similar verifications will follow when $v_{i}\in (S_{2})_{j}$ or $(S_{4})_{j}$.
\end{proof}

\begin{lem}\label{e}
Let $v_{i}<_\sigma v_{j}$ such that $v_{j}\in (S_{2})_{i}$. Then for all $v_{k}$ such that $v_{i}<_\sigma v_{k}<_\sigma v_{j}$, $v_{i}v_{k}\in E$ satisfying
$v_{k}\in (S_{2})_{i}$ and $v_{k}v_{j}\in E$ satisfying $v_{j}\in (S_{2})_{k}$.    
\end{lem}

\begin{proof}
As $v_{j}\in (S_{2})_{i}$, $v_{i}v_{j}\in E$. Now as $p_{i}<_\sigma p_{k}<_\sigma p_{j}$, $p_{k}\in [p_{i},p_{j}]\subsetneq [p_{i},b_{i}]\subseteq I_{i}$. As $p_{k}\in [p_{i},b_{i}]$ we have $v_{i}v_{k}\in E$ satisfying $v_{k}\in (S_{2})_{i}$. Now as $p_{i}<_\sigma p_{k}<_{\sigma} p_{j}$, from Lemma \ref{a} we get $b_{i}<_\sigma b_{k}<_{\sigma} b_{j}$.
Hence $p_{j}\in [p_{k},b_{i}]\subsetneq [p_{k},b_{k}]\subseteq I_{k}$. Therefore we get $v_{k}v_{j}\in E$ satisfying  $v_{j}\in (S_{2})_{k}$.
\end{proof}
 
\begin{lem}\label{d}
Let $v_{i}<_\sigma v_{j}$ such that $v_{j}\in (S_{4})_{i}$. Then for all $v_{k}$ such that $v_{i}<_\sigma v_{k}<_\sigma v_{j}$, $v_{j}v_{k}\in E$ satisfying $v_{j}\in (S_{4})_{k}$ and $v_{k}v_{i}\in E$ satisfying $v_{k}\in (S_{4})_{i}$.
\end{lem}

\begin{proof}
 As $v_{j}\in (S_{4})_{i}$, $v_{j}v_{i}\in E$ clearly. Now as $p_{i}<_\sigma p_{k}<_\sigma p_{j}$, $p_{k}\in [p_{i},p_{j}]\subsetneq [a_{j},p_{j}]\subseteq I_{j}$. This imply $v_{j}v_{k}\in E$ and $v_{k}\in (S_{3})_{j}$. Therefore from Lemma \ref{c} we get $v_{j}\in (S_{4})_{k}$. Now as $p_{i}<_{\sigma} p_{k}<_\sigma p_{j}$ imply $a_{i}<_{\sigma} a_{k}<_\sigma a_{j}$ from Lemma \ref{a}, we get $p_{i}\in [a_{j},p_{k}]\subsetneq [a_{k},p_{k}]\subseteq I_{k}$. Thus we get $v_{k}\in (S_{4})_{i}$ satisfying $v_{k}v_{i}\in E$. 
\end{proof}

\begin{lem}\label{f}
If $(S_{2})_{i}\neq\emptyset$ then $(S_{4})_{i}=\emptyset$ and if $(S_{4})_{i}\neq\emptyset$ then $(S_{2})_{i}=\emptyset$. Again  $(S_{1})_{i}, (S_{3})_{i}$ can not be nonempty simultaneously.
\end{lem}

\begin{proof}
On contrary let us assume $v_{k}\in (S_{2})_{i}$ and $v_{j}\in (S_{4})_{i}$. 
Then $p_{k}\in [p_{i},b_{i}]$ and $p_{i}\in [a_{j},p_{j}]$ imply $p_{i}<_{\sigma} p_{k}<_{\sigma} p_{j}$. Hence from Lemma \ref{a} we get $a_{i}<_{\sigma} a_{k}<_{\sigma} a_{j}$. Therefore $p_{i}\in [a_{j},p_{k}]\subsetneq [a_{k},p_{k}]\subseteq I_{k}$ imply $v_{k}v_{i}\in E$. But this introduces contradiction as
$v_{i}v_{k}\in E$ and $G$ is oriented. Similar contradiction will arise if $(S_{1})_{i}, (S_{3})_{i}$  both are nonempty.
\end{proof}

\vspace{1em}
\newpage
\noindent \textbf{Proof of Theorem \ref{prp1}}
\begin{proof} 
\noindent From Lemma \ref{f}, as  at least one among $(S_{2})_{i}, (S_{4})_{i}$ is empty,
we can conclude from Lemma \ref{e} and Lemma \ref{d} that vertices of $(S_{2})_{i}$ (or $(S_{4})_{i}$) occur in a consecutive manner clockwise right to $v_{i}$ in $<_\sigma$. Similarly using Lemma \ref{c} one can verify that
vertices of $(S_{1})_{i}$ (or $(S_{3})_{i}$) occur consecutively anticlockwise left to $v_{i}$ in $<_{\sigma}$. Moreover, from Lemma \ref{e} it follows that if $v_{i}<_\sigma v_{j}$ satisfying $v_{i}v_{j}\in E$ then $\{v_{i},v_{i+1},\hdots, v_{j}\}$ form a transitive tournament. Similarly, when $v_{i}<_\sigma v_{j}$ satisfying $v_{j}v_{i}\in E$ then also $\{v_{i},v_{i+1},\hdots, v_{j}\}$ form a transitive tournament according to Lemma \ref{d}.

\vspace{0.3em}
\noindent We now assign a new orientation $\mu$ for each edge of $U(G)$ and $E_{\mu}$ be the new set of arcs. After the vertices of $V$ are arranged in the circular ordering $<_{\sigma}$, for each vertex $v_{i}\in V$ we assign edge direction from $v_{i}$ to each vertex of $(S_{2})_{i}$ (or $(S_{4})_{i}$) considering $(S_{2})_{i}$ (or $(S_{4})_{i}$) as the outset. Next 
from vertices of $(S_{1})_{i}$ (or $(S_{3})_{i}$) we assign edge direction towards $v_{i}$ considering $(S_{1})_{i}$ (or $(S_{3})_{i}$) as the inset of $v_{i}$.

\vspace{0.3em}
\noindent First of all, we verify that $\mu$ is well defined on the set $E_{\mu}$.  Let $v_{i}<_{\sigma}v_{j}$ satisfying $v_{i}v_{j}\in E_{\mu}$. Then $v_{j}\in (S_{2})_{i}$ or $(S_{4})_{i}$. Now if $v_{j}\in (S_{2})_{i}$ then $v_{i}\in (S_{1})_{j}$ from Lemma \ref{c}. Again if $v_{j}\in (S_{4})_{i}$ then $v_{j}v_{i}\in E$ from definition of $(S_{4})_{i}$ and hence from Lemma \ref{c} we get $v_{i}\in (S_{3})_{j}$. Now from Lemma \ref{g} we get when $v_{i}\in (S_{1})_{j}$ or $(S_{3})_{j}$ then $v_{i}\notin (S_{2})_{j}$ or $(S_{4})_{j}$. Therefore $v_{i}$ can not belong to the outset of $v_{j}$ from the above paragraph. Hence $v_{j}v_{i}\notin E_{\mu}$.

\vspace{0.3em}
\noindent Now to prove $U(G)$ has a round orientation, it is sufficient to show that when $v_{i}<_{\sigma} v_{j}$ satisfying $v_{i}v_{j}\in E_{\mu}$ then $\{v_{i},v_{i+1},\hdots, v_{j}\}$ forms a transitive tournament with respect to $\mu$ as the other part is similar.

\vspace{0.7em}
\noindent Let $v_{i}<_{\sigma} v_{j}$ satisfy $v_{i}v_{j}\in E_{\mu}$. Then $v_{j}\in (S_{2})_{i}$ or $(S_{4})_{i}$ from the definition of $\mu$. If $v_{j}\in (S_{2})_{i}$ then for all $v_{k_{1}}, v_{k_{2}}$ satisfying $v_{i}<_{\sigma} v_{k_{1}}<_{\sigma} v_{k_{2}}<_{\sigma} v_{j}$, $v_{j}\in (S_{2})_{k_{1}}, (S_{2})_{k_{2}}$ and $v_{k_{1}}, v_{k_{2}}\in (S_{2})_{i}$ from Lemma \ref{e}.
Therefore $v_{i}v_{k_{1}},v_{i}v_{k_{2}},v_{k_{1}}v_{j},v_{k_{2}}v_{j}\in E_{\mu}$ from definition of $\mu$.  Now as we have $v_{j}\in (S_{2})_{k_{1}}$ and $v_{k_{1}}<_{\sigma} v_{k_{2}}<_{\sigma} v_{j}$, from Lemma \ref{e} it follows that $v_{k_{2}}\in (S_{2})_{k_{1}}$. Hence $v_{k_{1}}v_{k_{2}}\in E_{\mu}$.  

\noindent Again if $v_{j}\in (S_{4})_{i}$ then applying Lemma \ref{d} one can able to show that for all $v_{k_{1}},v_{k_{2}}$ satisfying $v_{i}<_{\sigma} v_{k_{1}}<_{\sigma} v_{k_{2}}<_{\sigma} v_{j}$, $v_{j}\in (S_{4})_{k_{1}}, (S_{4})_{k_{2}}$ and $v_{k_{1}}, v_{k_{2}}\in (S_{4})_{i}$. Therefore $v_{i}v_{k_{1}}, v_{i}v_{k_{2}},v_{k_{1}}v_{j},v_{k_{2}}v_{j}\in E_{\mu}$ from definition of $\mu$. Now as $v_{k_{2}}\in (S_{4})_{i}$ and $v_{i}<_{\sigma} v_{k_{1}}<_{\sigma} v_{k_{2}}$ from Lemma \ref{d} it follows that $v_{k_{2}}\in (S_{4})_{k_{1}}$. Hence $v_{k_{1}}v_{k_{2}}\in E_{\mu}$.

\vspace{0.7em}

\noindent Thus we are able to show that vertices of the outset of $v_{i}$ form a transitive tournament. Therefore from the above verifications it follows that $U(G)$ has a round orientation and hence by Theorem \ref{pca}, $U(G)$ becomes a proper circular-arc graph.
\end{proof}

\noindent Although we see in Theorem \ref{prp1} that underlying graph of a proper oriented circular-arc catch digraph is a proper circular-arc graph, but the converse is not always true.

\vspace{0.4em}
\begin{exmp}
Consider the digraph $G_{a}=(V,E)$ in Figure \ref{counter}.

\begin{figure}[ht]
\centering
\begin{tikzpicture}[scale=0.6]
\draw[-][draw=black,thick] (1,0) -- (5.2,0);
\draw[-][draw=black,thick] (1,-3.6) -- (5.2,-3.6);
\draw[-][draw=black,thick] (1,0) -- (1,-3.6);
\draw[-][draw=black,thick] (5.2,0) -- (5.2,-3.6);
\draw[-][draw=black,thick] (1,-3.6) -- (-2.2,-1.8);
\draw[-][draw=black,thick] (1,0) -- (-2.2,-1.8);

\node[above] at (1,0) {\tiny{$v_{1}$}};
\node[above] at (5.2,0) {\tiny{$v_{2}$}};
\node[below] at (5.2,-3.6) {\tiny{$v_{3}$}};
\node[below] at (1,-3.6) {\tiny{$v_{4}$}};
\node[left] at (-2.2,-1.8) {\tiny{$v_{5}$}};

\draw [fill=black] (1,0) circle [radius=0.1];
\draw [fill=black] (5.2,0) circle [radius=0.1];
\draw [fill=black] (1,-3.6) circle [radius=0.1];
\draw [fill=black] (5.2,-3.6) circle [radius=0.1];
\draw [fill=black] (-2.2,-1.8) circle [radius=0.1];

\draw[- angle 90](1,0)--(2.6,0);
\draw[- angle 90] (1,0)--(1,-1.8);
\draw[- angle 90](5.2,-3.6)--(5.2,-1.8);
\draw[- angle 90](5.2,-3.6)--(2.6,-3.6);
\draw[- angle 90](1,0)--(-0.6,-.9);
\draw[- angle 90](1,-3.6)--(-0.6, -2.7);

\node[] at (3.6,-5.6) {$G_{a}$};


\draw[-][draw=black,thick] (12,0) -- (16.2,0);
\draw[-][draw=black,thick] (12,-3.6) -- (16.2,-3.6);
\draw[-][draw=black,thick] (12,0) -- (12,-3.6);
\draw[-][draw=black,thick] (16.2,0) -- (16.2,-3.6);
\draw[-][draw=black,thick] (12,-3.6) -- (9.2,-1.8);
\draw[-][draw=black,thick] (12,0) -- (9.2,-1.8);

\node[above] at (12,0) {\tiny{$v_{1}$}};
\node[above] at (16.2,0) {\tiny{$v_{2}$}};
\node[below] at (16.2,-3.6) {\tiny{$v_{3}$}};
\node[below] at (12,-3.6) {\tiny{$v_{4}$}};
\node[left] at (9.2,-1.8) {\tiny{$v_{5}$}};

\node[] at (14.4,-5.6) {$U(G_{a})$};

\end{tikzpicture}
\caption{$U(G_{a})$ is a proper circular-arc graph whereas $G_{a}$ is not a proper oriented circular-arc catch digraph}\label{counter}
\end{figure}

\noindent Note that $G_{a}$ is an oriented graph satisfying $U(G_{a})$ to be a proper circular-arc graph with respect to the circular-arc representation $[1,5]$, $[4,7]$, $[6,9]$, $[8,2]$, $[10,3]$ 
corresponding to the vertices $v_{1},v_{2},v_{3},v_{4},v_{5}$. But below we will show that $G_{a}$ is not proper oriented circular-arc catch digraph.

\vspace{0.4em}
\noindent On contrary let $G_{a}$ have a proper circular-arc catch digraph representation $\{(I_{v},p_{v})|v\in V\}$. As the $4$-cycle formed by the vertices $\{v_{1},v_{2},v_{3},v_{4}\}$ in $U(G_{a})$ is alternatively oriented in $G_{a}$, one can check that $v_{1}\prec v_{2}\prec v_{3}\prec v_{4}$  and $v_{1}\prec v_{4}\prec v_{3}\prec v_{2}$ are the only possible circular catch orderings of $V\setminus \{v_{5}\}$ so that the transformed matrix $B$ obtained from the augmented adjacency matrix $A(G_{a})$ satisfies all conditions of Theorem \ref{pcacd}. Now as $v_{1}v_{5},v_{4}v_{5}\in E$, when $v_{1}\prec v_{2}\prec v_{3}\prec v_{4}$ 
then $ v_{1}\prec v_{2}\prec v_{3}\prec v_{4}\prec v_{5}$ and when $v_{1}\prec v_{4}\prec v_{3}\prec v_{2}$ then $v_{1}\prec v_{5} \prec v_{4} \prec v_{3} \prec v_{2}$ becomes the only possible circular-catch ordering of $V$. But in each of these cases with respect to any permutation of the rows of $A(G_{a})$, the transformed matrix $B$ does not satisfy condition $2$ of  Theorem \ref{pcacd} as the circular stretch of $1$'s corresponding to row $v_{5}$ is properly contained in the circular stretch of $1$'s corresponding to row $v_{1}$. Thus we get the contradiction. 
\end{exmp}

\noindent In the following theorem we characterize a proper oriented circular-arc catch digraph with the help of Theorem \ref{cop1} and Theorem \ref{pcacd}.

\begin{thm}\label{4pt}
Let $G=(V,E)$ be an oriented graph. Then $G$ is a proper oriented circular-arc catch digraph if and only if there exists a circular vertex ordering $\prec_{\sigma}$ of $V$ such that for any $u\prec_{\sigma}v\prec_{\sigma}w\prec_{\sigma}x$ if $uw\in E$ then either $uv,vw\in E$ or $ux,xw\in E$.
\end{thm}

\begin{proof}
\textit{\textbf{Necessity:}} Let $G$ be a proper oriented circular-arc catch digraph with representation 
$\{(I_{v},p_{v})|v\in V\}$. Then from Theorem \ref{pcacd} there exists vertex orderings (say $\prec_{r}, \prec_{c}$) for rows and columns of the augmented adjacency matrix $A(G)$ so that $B=(b_{i,j})$ (the transformed matrix) satisfy all conditions of Theorem \ref{pcacd}. From the proof of Theorem \ref{pcacd} one can check $\prec_{r}$ and $\prec_{c}$ are the increasing order of $a_{i}$'s and $p_{i}$'s respectively. As $G$ is oriented from Lemma \ref{a} one can easily verify that $\prec_{r}=\prec_{c}=\prec_{\sigma}$ (\text{say}) (upto some cyclic permutation). Hence we can order the rows and columns of $B$ with the same vertex ordering $\prec_{\sigma}$ so that it still satisfy all the conditions and $b_{i,i}=1$ for all $i$.  

\vspace{0.2em}
\noindent Let $u\prec_{\sigma}v\prec_{\sigma}w\prec_{\sigma}x$ satisfy $uw\in E$. Suppose $u,v,w,x$ corresponds to $i,j,k,l$ rows (columns) in $B$. Then $uw\in E$ implies $b_{i,k}=1$. Again by circular $1$'s property (along rows) of the $i^{\text{th}}$ row of $B$ we have either $b_{i,j}=1$ or $b_{i,l}=1$ or both. Hence $uv\in E$ or $ux\in E$ or both. 

\vspace{0.2em}

\noindent Now when $uv\in E$ and $ux\notin E$, then $b_{j,i}=0$ as $G$ is oriented. Moreover, $u$ is not a full row as $b_{i,l}=0$. Hence if we assume on contrary that $vw\notin E$ then $b_{j,k}=0$, which contradict condition $2$ for rows $i,j$. Similarly if $ux\in E$ and $uv\notin E$, then $b_{l,i}=0$ and $x$ is not a full row. Therefore if $xw\notin E$ then condition $2$ gets contradicted for rows $i,l$. Again when $uv,ux\in E$ then if $vw,xw\notin E$ then $b_{j,k}=b_{l,k}=0$, which contradicts circular $1$'s property (along columns) for $k^{\text{th}}$ column in $B$ as $b_{i,k}=b_{k,k}=1$. Hence $vw\in E$ when $uv\in E$ and $xw\in E$ when $ux\in E$.

\vspace{0.37em}

\noindent \textit{\textbf{Sufficiency:}} Conversely, let $G=(V,E)$ be an oriented graph satisfying the given condition with respect to a circular vertex ordering $\prec_{\sigma}$. Now we arrange all the vertices in same order 
$\prec_{\sigma}$ along rows and columns of the augmented adjacency matrix $A(G)=(a_{i,j})$. To prove $G$ is proper circular-arc catch digraph we will show that $A(G)$ satisfies all conditions of Theorem \ref{pcacd}.

\vspace{0.7em}

\noindent \textbf{\textit{Circular $1$'s property along rows and columns is satisfied by $A(G)$.}}

\vspace{0.27em}
  
\noindent We assume on contrary that the $i^{\text{th}}$ row of $A(G)$ contains $(1,0,1,0)$ along columns $u,v,w,x$. As $a_{i,i}=1$ without loss of generality we can consider column $u$ or column $w$ as $i^{\text{th}}$ column. Now in $A(G)$ as the rows and columns are ordered in the same ordering $\prec_{\sigma}$ when row $u$ corresponds to the $i^{\text{th}}$ row then as $u\prec_{\sigma}v\prec_{\sigma}w\prec_{\sigma}x$ and $uw\in E$ from given condition $uv\in E$ or $ux\in E$, which is not true.
Similarly if row $w$ corresponds to the $i^{\text{th}}$ row then from the given condition $wu\in E$ in the circular ordering $w\prec_{\sigma} x \prec_{\sigma} u<_{\sigma} v$ imply $wx,wv\in E$. But this leads to contradiction.
Analogously, one can show $A(G)$ can not contain $(0,1,0,1)$ along any row.

\vspace{0.37em}

\noindent Again we assume on contrary that $j^{\text{th}}$ column of $A(G)$ contains $(0,1,0,1)$ along rows $u,v,w,x$. As $a_{j,j}=1$ without loss of generality we can consider row $v$ or row $x$ as $j^{\text{th}}$ row. Now when row $v$ corresponds to $j^{\text{th}}$ row then from our given condition in the circular ordering $x\prec_{\sigma} u \prec_{\sigma} v\prec_{\sigma} w$, $xv\in E$ imply $uv\in E$ or $wv\in E$. But this is not true. Similarly when row $x$ corresponds to $j^{\text{th}}$ row then from the given condition $vx\in E$ in the circular ordering $v\prec_{\sigma} w\prec_{\sigma} x\prec_{\sigma} u$ imply $wx,ux\in E$, which is again not true. Again if $A(G)$ contains $(1,0,1,0)$ in $j^{\text{th}}$ column along rows $u,v,w,x$ then one can find similar contradiction by considering row $u$ or $w$ as $j^{\text{th}}$ row of $A(G)$. 

\vspace{0.9em}
\noindent \textbf{\textit{Any two rows of $A(G)$ satisfies condition $2$ of Theorem \ref{pcacd}.}}
\vspace{0.39em}

\noindent We assume on contrary there exists a pair of rows $\{r,s\}$ in $A(G)$ for which $r^{\ast}-s^{\ast}$ is not a circular interval.
Let $r<s$ satisfying $s^{\ast}\subset r^{\ast}$ in $A(G)$ and $r$ is not a full row.

\vspace{0.27em}

\noindent If both of them are of $1$st type then as $a_{r,r}=a_{s,s}=1$ and $G$ is oriented we can always choose $u,v,w$ corresponding to the columns $r,s,r_{2}$
in such a way so that $a_{r,s}=a_{r,w}=1,a_{s,w}=0$. Again as $r$ is not a full row there must exist some vertex $x$ corresponding to the column $r_{2}+1$ or $r_{1}-1$ in $\prec_{\sigma}$ satisfying $u\prec_{\sigma} v\prec_{\sigma}w\prec_{\sigma}x$ such that $a_{r,x}=a_{s,x}=0$. But this contradicts our given condition as $uw\in E$ and $vw,ux\notin E$.

\vspace{0.27em}

\noindent Again if $r$ be a row of $2$nd type and $s$ be a row of $1$st type then either $r_{3}<s_{1}$ or $s_{2}<r_{4}$. We consider the case when $r_{3}<s_{1}$. If $r\geq r_{3}$ then $r<s_{1}$ as $G$ is oriented and $r<s$. Therefore we can choose $u,v,w,x$ corresponding to the columns $r,s,r_{4},r_{4}+1$ respectively so that $u\prec_{\sigma} v\prec_{\sigma}w\prec_{\sigma}x$ and $uw\in E$ but $ux,vw\notin E$. Again if $r\leq r_{4}$ then
as $r_{3}<s_{1}$ we can choose $u,v,w,x$ corresponding to the columns $r, r_{4}+1, r_{3}, s$ so that $uw\in E$ but $uv,xw\notin E$. Therefore contradiction arises in both of the cases.
Similar contradiction will follow when $s_{2}<r_{4}$. 

\vspace{0.27em}

\noindent Next when both $r,s$ are rows of $2$nd type, then $s_{4}<r_{4}$ and $r_{3}<s_{3}$. If $r\geq r_{3}$ then $a_{r,s}=1$ as $r<s$, which imply $a_{s,r}=0$ as $G$ is oriented. Hence we get $r<s_{3}$. We now consider columns $u,v,w,x$ 
corresponding to the columns $r,s,r_{4},r_{4}+1$. Therefore in the circular vertex ordering $u\prec_{\sigma} v\prec_{\sigma} w\prec_{\sigma}x$, $uw\in E$ but $vw,ux\notin E$. Again if $r\leq r_{4}$ then $s_{4}<r$ as before. In this case we choose columns $u,v,w,x$ corresponding to the columns $r, r_{4}+1,r_{3},s$ satisfying $u\prec_{\sigma}v \prec_{\sigma} w\prec_{\sigma}x$ so that $uw\in E$ but $uv,xw\notin E$. Thus contradictions arise in both cases. 

\vspace{0.3em}
\noindent Similar contradictions will occur when $s<r$.

\vspace{1em}

\noindent \textbf{\textit{Condition $3$ of Theorem \ref{pcacd} is satisfied by rows of $A(G)$.}}

\vspace{0.2em}

\noindent We assume on contrary there exists a triple $\{r,s,t\}$ for which $r^{\ast}-s^{\ast}-t^{\ast}$ is not a circular interval where $r$ is a full row and $s,t$ are two nontrivial rows of the matrix. 

\vspace{0.27em}

\noindent Let $s,t$ both be of $1$st type and $s_{2}+1<t_{1}$. Then $s<t$ in $A(G)$. If both $s,t$ occur on the same side of $r$  (i.e., $s<t<r$ or $r<s<t$) in the arrangement of rows and columns of $A(G)$, then we can consider $u,v,w,x$ corresponding to $r,s,s_{2}+1,t$ along rows and columns so that $u\prec_{\sigma} v\prec_{\sigma} w\prec_{\sigma}x$. But one can find contradiction as $uw\in E$ but $vw,xw\notin E$. Again if $s,t$ occurs opposite to $r$ (i.e., $s<r<t$) along both rows and columns of $A(G)$ then either $s_{1}\neq 1$ or $t_{2}\neq n$ as $r^{\ast}-s^{\ast}-t^{\ast}$ is not a circular interval as per our consideration. If $s_{1}\neq 1$ then we consider $u,v,w,x$ corresponding to the rows and columns of $r,t,s_{1}-1,s$ respectively. Again when $t_{2}\neq n$ then $u,v,w,x$ can be considered corresponding to the rows and columns of $r,t,t_{2}+1,s$. In both cases one can find similar contradiction. 

\vspace{0.27em}

\noindent Now we consider the case when $s$ is of $2$nd type and $t$ is of $1$st type. Without loss of generality we assume $s_{4}+1<t_{1}$ and $t_{2}+1<s_{3}$
as otherwise $r^{\ast}-s^{\ast}-t^{\ast}$ is a circular interval. Now as $G$ is oriented one can check in this case $s,t$ must occur opposite to $r$ in $A(G)$. Now when $s<r<t$ then  considering $u,v,w,x$ corresponding to $r,t,t_{2}+1,s$ and when $t<r<s$ then  considering $u,v,w,x$ corresponding to $r,s,s_{4}+1,t$ along both rows and columns of $A(G)$, one can find in the circular ordering $u\prec_{\sigma} v\prec_{\sigma} w\prec_{\sigma}x$, $uw\in E$ but $vw,xw\notin E$.   
\end{proof}

\begin{figure}

\begin{tikzpicture}[scale=0.6]

\draw[-][draw=black,thick] (1,0) -- (5.2,0);
\draw[-][draw=black,thick] (1,-3.6) -- (5.2,-3.6);
\draw[-][draw=black,thick] (1,0) -- (1,-3.6);
\draw[-][draw=black,thick] (5.2,0) -- (5.2,-3.6);

\draw[-][draw=black,thick] (1,-3.6) -- (-2.2,-1.8);
\draw[-][draw=black,thick] (1,0) -- (-2.2,-1.8);

\draw[-][draw=black,thick] (5.2,0) -- (1,-3.6);
\node[above] at (1,0) {\tiny{$v_{1}$}};
\node[above] at (5.2,0) {\tiny{$v_{2}$}};
\node[below] at (5.2,-3.6) {\tiny{$v_{3}$}};
\node[below] at (1,-3.6) {\tiny{$v_{4}$}};
\node[left] at (-2.2,-1.8) {\tiny{$v_{5}$}};

\draw [fill=black] (1,0) circle [radius=0.1];
\draw [fill=black] (5.2,0) circle [radius=0.1];
\draw [fill=black] (1,-3.6) circle [radius=0.1];
\draw [fill=black] (5.2,-3.6) circle [radius=0.1];
\draw [fill=black] (-2.2,-1.8) circle [radius=0.1];

\draw[- angle 90](1,0)--(2.6,0);
\draw[- angle 90](1,-3.6)--(1,-1.8);
\draw[- angle 90](1,-3.6)--(3.1,-1.8);
\draw[- angle 90](5.2,0)--(5.2,-1.8);
\draw[- angle 90](5.2,-3.6)--(5.2,-1.8);
\draw[- angle 90](5.2,-3.6)--(2.6,-3.6);
\draw[- angle 90](-2.2,-1.8)--(-0.6,-.9);
\draw[- angle 90](1,-3.6)--(-0.6, -2.7);

\node[] at (3.6,-5.6) {$G_{1}$};


\draw[-][draw=black,thick] (7.6,0) -- (11.6,0);
\draw[-][draw=black,thick] (7.6,0) -- (9.6,-1.8);
\draw[-][draw=black,thick] (11.6,0) -- (9.6,-1.8);
\draw[-][draw=black,thick] (9.6,-1.8) -- (9.6,-3.6);

\draw [fill=black] (7.6,0) circle [radius=0.1];
\draw [fill=black] (11.6,0) circle [radius=0.1];
\draw [fill=black] (9.6,-1.8) circle [radius=0.1];
\draw [fill=black] (9.6,-3.6) circle [radius=0.1];

\node[above] at (7.6,0) {\tiny{$v_{1}$}};
\node[above] at (11.6,0) {\tiny{$v_{2}$}};
\node[right] at (9.6,-1.8) {\tiny{$v_{3}$}};
\node[below] at (9.6,-3.6) {\tiny{$v_{4}$}};

\draw[- angle 90](7.6,0)--(9.6,0);
\draw[- angle 90](11.6,0)--(10.6,-0.9);
\draw[- angle 90] (9.6,-1.8)--(8.6,-0.9);
\draw[- angle 90] (9.6,-1.8)--(9.6,-2.7);

\node [] at (9.6,-5.6){$G_{2}$};

\draw[-][draw=black,thick] (14.6,0) -- (19.6,0);
\draw[-][draw=black,thick] (14.6,-3.6) -- (19.6,-3.6);
\draw[-][draw=black,thick] (19.6,0) -- (19.6,-3.6);
\draw[-][draw=black,thick] (14.6,0) -- (14.6,-3.6);

\draw [fill=black] (14.6,0) circle [radius=0.1];
\draw [fill=black] (19.6,0) circle [radius=0.1];
\draw [fill=black] (14.6,-3.6) circle [radius=0.1];
\draw [fill=black] (19.6,-3.6) circle [radius=0.1];

\node[above] at (14.6,0) {\tiny{$v_{1}$}};
\node[above] at (19.6,0) {\tiny{$v_{2}$}};
\node[below] at (19.6,-3.6) {\tiny{$v_{3}$}};
\node[below] at (14.6,-3.6) {\tiny{$v_{4}$}};

\draw[- angle 90](14.6,0)--(17.1,0);
\draw[- angle 90](19.6,0)--(19.6,-1.8);
\draw[- angle 90](19.6,-3.6)--(17.1,-3.6);
\draw[- angle 90](14.6,-3.6)--(14.6,-1.8);

\node[] at (17.6,-5.6) {$G_{3}$};

\end{tikzpicture}
\end{figure}
\begin{figure}\centering
\begin{tikzpicture}[scale=.6]
\draw[-][draw=black,thick] (1,0) -- (20,0);
\draw[-][draw=black,thick] (1,0) -- (1,16);
\draw[-][draw=black,thick] (1,16) -- (20,16);
\draw[-][draw=black,thick] (20,0) -- (20,16);
\node [right] at (6,15) {circular-arc catch digraph};
\draw[-][draw=black,thick] (2,1) -- (12,1);
\draw[-][draw=black,thick] (2,1) -- (2,10);
\draw[-][draw=black,thick] (2,10) -- (12,10);
\draw[-][draw=black,thick] (12,1) -- (12,10);
\node [right] at (3.1,7.6) {oriented};
\node [right] at (2.5,6.9) {circular-arc};
\node [right] at (2.5,6.2) {catch digraph};
\draw[-][draw=black,thick] (8.6,1.32) -- (19.6,1.32);
\draw[-][draw=black,thick] (8.6,14.2) -- (19.6,14.2);
\draw[-][draw=black,thick] (8.6,1.32) -- (8.6,14.2);
\draw[-][draw=black,thick] (19.6,1.32) -- (19.6,14.2);
\node [right] at (9.3,11.6) {proper circular-arc catch digraph};
\node [right] at (8.9,8.3) {proper};
\node [right] at (8.9,7.7) {oriented};
\node [right] at (8.9,7.1) {circular-};
\node [right] at (8.9,6.4) {arc};
\node [right] at (8.9,5.7) {catch};
\node [right] at (8.9,4.9) {digraph};
\node [right] at (3,12) {$G_{1}$};
\node[right] at (3.2,3.66) {$G_{2}$};
\node [right] at (9.6,3.66) {$G_{3}$};
\node [right] at (15.2,8) {$G_{4}$};


\draw[-][draw=black,thick] (21.6,1.6) -- (26.6,1.6);
\draw[-][draw=black,thick] (21.6,4.6) -- (26.6,4.6);
\draw[-][draw=black,thick] (26.6,1.6) -- (26.6,4.6);
\draw[-][draw=black,thick] (21.6,1.6) -- (21.6,4.6);

\draw [fill=black] (21.6,1.6) circle [radius=0.1];
\draw [fill=black] (26.6,1.6) circle [radius=0.1];
\draw [fill=black] (21.6,4.6) circle [radius=0.1];
\draw [fill=black] (26.6,4.6) circle [radius=0.1];

\node[above] at (21.6,4.6) {\tiny{$v_{1}$}};
\node[above] at (26.6,4.6) {\tiny{$v_{2}$}};
\node[below] at (26.6,1.6) {\tiny{$v_{3}$}};
\node[below] at (21.6,1.6) {\tiny{$v_{4}$}};

\draw[- angle 90] (21.6,4.6)--(24.1,4.6);

\draw[- angle 90] (26.6,4.6)--(26.6,3.1);

\draw[- angle 90] (26.6,1.6)--(26.6,3.1);

\draw[- angle 90] (26.6,1.6)--(24.1,1.6);
\draw[- angle 90] (21.6,1.6)--(21.6,3.1);

\node[] at (24.6,0) {$G_{4}$};

\end{tikzpicture}
\caption{Relations between some subclasses of circular-arc catch digraph}\label{cacd1}
\end{figure}

\section{Conclusion}

\noindent In this article we consider two classes of circular-arc catch digraphs, namely, oriented circular-arc catch digraph and proper circular-arc catch digraph. We obtain characterization of an oriented circular-arc catch graph, when it is a tournament. Also we characterize a proper circular-arc catch digraph by defining a notion monotone circular ordering of the vertices of it.
Further we characterize a proper oriented circular-arc catch digraph by providing a specific circular vertex ordering of the vertex set. In Figure \ref{cacd1} we show the relationships between these digraphs whose proofs follow from the characterization theorems. Several combinatorial problems remain open for these digraphs. Finding the recognition algorithm for proper circular-arc catch digraphs or to find the complete list of forbidden digraphs of an oriented circular-arc catch digraph are very challenging among them.

{\small


\begin{thebibliography}{00}

\bibitem{Bang}
J.~Bang-Jensen, G.~Z. Gutin, Digraphs, Springer Monographs in Mathematics, 2008.

\bibitem{BDGS}
A.~Basu, S.~Das, S.~Ghosh and M.~Sen, \emph{Circular-arc bigraphs and its subclasses}, J. Graph Theory \textbf{73} (2013), 361--376.


\bibitem{Con}
G. Confessore, P. Dell’olmo, S. Giordani, An approximation result for a periodic allocation problem, \emph{Discrete Applied Mathematics} \textbf{112} (2001) 53--72.

\bibitem{DHH}
X.~Deng, P.~Hell and J.~Huang, \emph{Linear time representation algorithms for proper circular arc graphs and proper interval graphs}, SIAM J. Comput. \textbf{25} (1996) 390--403.


\bibitem{G}
M.~C.~Golumbic, Algorithmic graph theory and perfect graphs, Annals of Disc. Math., \textbf{57}, Elsevier Sci., USA, 2004.

\bibitem{GT}
M.~C.~Golumbic and A.~Trenk, Tolerence Graphs, Cambridge studies in advanced mathematics, Cambridge University Press, 2004.


\bibitem{HBH}
P. Hell, J. Bang-Jensen, J. Huang, \emph{Local tournaments and proper circular-arc graphs}, Lecture Notes in Computer Science \textbf{450} (1990) 101–109.

\bibitem{Hubert}
L. Hubert, Some applications of graph theory and related non symmetric techniques to problems of approximate seriation: The case of symmetric
proximity measures, \emph{British Journal of Mathematical and Statistical Psychology} \textbf{27} (1974) 133--153.

\bibitem{Lekker}
 C.G. Lekkerkerker, J.C. Boland, "Representation of a finite graph by a set of intervals on the real line", \emph{Fund. Math.}, \textbf{51} (1962),: 45–64.
 
\bibitem{Lin}
Min Chih Lin, Jayme L. Szwarcfiter, \emph{Characterizations and recognition of circular-arc graphs and subclasses: A survey}, Discrete Mathematics, \textbf{309} (2009) 5618–5635
 
\bibitem{Maehara} 
H.~Maehara, A Digraph Represented by a Family of Boxes or Spheres, \emph{J. Graph Theory}, \textbf{8} (1984), 431--439. 


\bibitem{Mc}
R.~M.~McConnell, \emph{Linear time recognition of circular-arc graphs}, Algorithmica, \textbf{37} (2003), 93-147.

 
\bibitem {Nussbaum}
Y. Nussbaum, \emph{Recognition of circular-arc graphs and some subclasses}, M. Sc. Thesis, Tel-Aviv University, School of Computer Science, Tel-Aviv, Israel,
2007.

\bibitem{PG}
S. Paul, S. Ghosh, On some subclasses of interval catch digraphs, \emph{Electronic Journal of Graph Theory and Applications}, \textbf{10 (1)}, 2022, 157--171.

\bibitem{PG1}
S. Paul, S. Ghosh, On central max-point tolerance graphs and some subclasses of interval catch digraphs, arXiv:1712.00008.

\bibitem{Prisner} 
E.~Prisner, A Characterization of Interval Catch Digraphs, \emph{Discrete Mathematics} \textbf{73} (1989), 285-289.

\bibitem{Prisner1}
E. Prisner, \emph{Algorithms for interval catch digraphs}, Discrete Applied Mathematics \textbf{51} (1994), 147-157.

\bibitem{Roberts}
F. S. Roberts, \emph{ Discrete mathematical models}, Prentice-Hall, Upper Saddle River, NJ, 1976.

\bibitem{Robert2}
F.S. Roberts, Graph Theory and its Applications to Problems of Society, \emph{ SIAM Publishing}, Philadelphia, 1978.

\bibitem{Stahl}
F. Stahl, Circular genetic maps, \emph{Journal Cell. Physiol.} \textbf{70} (1967) 1--12.

\bibitem{Stou}
K. Stouffers, Scheduling of traffic lights—a new approach, \emph{Transportation Res.} \textbf{2} (1968) 199--234.

\bibitem{Ste}
S. Stefanakos, T. Erlebach, Routing in all-optical ring network revisited, in: \emph{Proceedings of the 9th IEEE Symposium on Computers and Communications},
2004, pp. 288–293.



\bibitem{Skrien}
D. Skrien, \emph {A relationship between triangulated graphs, comparability graphs, proper interval graphs, proper circular-arc graphs and nested interval graphs}, Journal of Graph Theory \textbf{6} (1982), 309–316.


\bibitem {Tucker}
A. Tucker, \emph{Matrix characterizations of circular-arc graphs}, Pacific J. Math., (1971), \textbf{39}: 535–545. 

\bibitem{Tuck}
A. Tucker, A structure theorem for the consecutive 1’s property, \emph{Journal of Combinatorial Theory} \textbf{12} (1972) 153--162.


\bibitem{Tuc}
A. Tucker, \emph{Two characterizations of proper circular-arc graphs}, PhD thesis, Stanford University, 1969.

\bibitem{Tu}
A. Tucker, Circular-arc graphs: New uses and a new algorithm, in: Theory and Application of Graphs, \emph{Lecture Notes in Mathematics} \textbf{642} (1978)
580--589.

\bibitem{T}
A. Tucker, An efficient test for circular-arc graphs, \emph{SIAM Journal on Computing} \textbf{9} (1980) 1--24.


\end{thebibliography}
\end{document}